\documentclass[11pt]{amsart}
\usepackage{amsmath,amssymb,amsbsy,amsfonts,amsthm,latexsym,
	amsopn,amstext,amsxtra,euscript,amscd,stmaryrd,mathrsfs,
	cite,array,mathtools,enumerate,bm}

\usepackage{color}
\usepackage{graphics,epsfig, graphicx}
\usepackage{float} 
\usepackage[english]{babel}

\usepackage{url}
\usepackage[colorlinks,linkcolor=blue,anchorcolor=blue,citecolor=green]{hyperref}

\usepackage{geometry}
	\newtheorem{thm}{Theorem}
	\newtheorem{lem}[thm]{Lemma}
	
	\newtheorem{cor}[thm]{Corollary}
	\newtheorem{prop}[thm]{Proposition} 
	
	\newtheorem*{definition*}{Definition}

	\newtheorem*{hyp}{Low Lying Zeros Hypothesis}
	\newtheorem*{shifted hyp E}{Hypothesis on Zeros for Twists of Elliptic Curves}

	\numberwithin{equation}{section}
	\numberwithin{thm}{section}
	\numberwithin{table}{section}
	
    \def\cF{{\mathcal F}}

\begin{document}

\title[Linear Combinations of Logarithms of $L$-functions over Function Fields]
{Linear Combinations of Logarithms of 
$L$-functions \\
over Function Fields 
 at Microscopic Shifts and Beyond
}

\author[F. \c{C}\.{\i}\c{c}ek] {Fatma \c{C}\.{\i}\c{c}ek}
\address{\.{I}stanbul, 34528, T\tiny{\"{U}}RK\.{I}YE}
\email{cicek.ftm@gmail.com}

\author[P. Darbar] {Pranendu Darbar}
\address{School of Mathematics and Statistics, University of New South Wales, Sydney NSW 2052, Australia}
\email{darbarpranendu100@gmail.com}

\author[A. Lumley]{Allysa Lumley}
\address{Department of Mathematics and Statistics, York University, N520 Ross,
	4700 Keele Street, Toronto, ON M3J 1P3, Canada}
\email{alumley2@yorku.ca}

\begin{abstract} 
In the function field setting with a fixed characteristic, it is known~\cite{DL} that the values $\log \big|L\big(\frac12, \chi_D\big)\big|$ as $D$ varies over monic and square-free polynomials are asymptotically Gaussian distributed on the assumption of a low lying zeros hypothesis as the degree of $D$ tends to $\infty$. For real distinct shifts $t_j$ all of microscopic size or all of nonmicroscopic size relative to the genus, we consider linear combinations of $\log\big|L\big(\frac12+it_j, \chi_D\big)\big|$ with real coefficients, and separately, of $\arg L\big(\frac12+it_j, \chi_D\big).$ We provide estimates for their distribution functions under the low lying zeros hypothesis. We similarly study distribution functions of linear combinations of $\log\big|L\big(\frac12+it_j, E\otimes \chi_D\big)\big|$, and separately $\arg L\big(\frac12+it_j, E\otimes\chi_D\big)$, for quadratic twists of elliptic curves $E$ with root number one as the conductor gets large. As an application, we prove a central limit theorem for the fluctuation of the number of nontrivial zeros of such $L$-functions from its mean, and thus recover previous results from  \cite{Faif} and \cite{FR}. Correlations of these fluctuations are in harmony with the results of~\cite{Bou}, \cite{CD} and \cite{Wie} for zeros of the Riemann zeta function and for eigenangles of unitary random matrices.
\end{abstract}


\maketitle

\section{Introduction}

Values of logarithms of $L$-functions play a fundamental role in analytic number theory, particularly for the zero distributions of  $L$-functions and moments of $L$-functions (see~\cite{Sound2009}).

Many a significant insight into the behavior of the logarithm of the Riemann zeta function on the critical line $\operatorname{Re}(s)=\tfrac12$ is provided through Selberg's central limit theorem, which states that as $t\to \infty$,
\[
 \frac{\log \zeta(\tfrac12+it)}{\sqrt{\log \log t}}\longrightarrow \mathcal{N},
\]
 where $\mathcal{N}$ is the standard complex Gaussian random variable. Selberg also extended his theorem to general $L$-functions~\cite{Selberg92}. 
 
 A natural object of further inquiry is the multidimensional distribution of logarithms of distinct $L$-functions. This brings about the study of the correlations of $L$-values of type 
 \begin{equation}\label{eq:log L values}
 \log L\big(\tfrac12+i(t+t_1)\big), \dots, \log L\big(\tfrac12+i(t+t_k)\big),
 \end{equation}
(further to be normalized as necessary) where $t_1, t_2, \dots, t_k$ are functions of $t$, as $t$ tends to $\infty$.

It was Bourgade \cite[Theorem 1.1]{Bou} who examined such local statistics for the Riemann zeta function in the case where the $t_j$ are close to one another. More precisely, Bourgade showed that for $\epsilon_t\to 0$ with $\epsilon_t\gg 1/\log t$ as $t\to\infty$, if the shifts are such that $0\leq t_1<\cdots< t_k<c<\infty$  as functions of $t$, and also 
  \[
  \frac{\log |t_{j_1}-t_{j_2}|}{\log \epsilon_t}\longrightarrow c_{j_1, j_2}\in [0, \infty]
  \quad \text{for all } \,  j_1\neq j_2,
  \]
then 
 \begin{align*}
 \frac{1}{\sqrt{-\log \epsilon_t}}\Big(\log \zeta\big(\tfrac12+\epsilon_t+i(t+t_1)\big), \ldots, \log \zeta\big(\tfrac12+\epsilon_t+i(t+t_k)\big)\Big) 
 \longrightarrow \big(\mathcal{N}_1, \ldots, \mathcal{N}_k\big).
 \end{align*}
Here, if $\epsilon_t\ll 1/\log t$, then the factor $-\log \epsilon_t$ is replaced by $\log \log t$.

Bourgade's proof uses Selberg’s approximation for $\frac{\zeta'}{\zeta}(s)$ (see \cite[p. 491]{Bou}) and an $L^2$-bound for 
\[
 \log \zeta\big(\tfrac12+\epsilon_t+i(t+t_j)\big)
 -\sum_{p\leq t}\frac{1}{p^{1/2+\epsilon_t+i(t+t_j)}}.
\]
These allow him to apply Slutsky’s lemma and the Cram\'{e}r--Wald device. He also used these results to observe the correlation structure in \eqref{correlation structure} for fluctuations of the number of zeta zeros from its expected value (see~\cite[Corollary 1.3]{Bou}). We note that this correlation result is the number-theoretic analogue to that of Diaconis and Evans~\cite{DE} in their work for fluctuations of the number of eigenangles of random unitary matrices from its expected value. Recently, extensions of Bourgade's result to Dirichlet $L$-functions and Dedekind zeta functions of quadratic fields were proven by Hsu and Wong~\cite{Hsu Wong}.

For this paper, we were motivated to study correlations of values in the form \eqref{eq:log L values} for of $L$-functions from symplectic or orthogonal families. A generalization of Selberg's theorem holds for logarithms of $L$-functions that form a unitary family. Since real parts of logarithms of distinct $L$-functions have approximately independent distributions on the critical line, one can obtain another central limit theorem for the distribution of linear combinations of such logarithms where coefficients are real numbers. This is a mirror image of the probabilistic fact that sums of independent Gaussian random variables are also Gaussian. This more general central limit theorem was cleverly exploited in \cite{Bom H} to prove that linear combinations of the $L$-functions themselves that share a functional equation, where the coefficients are again real, have a full proportion of all their nontrivial zeros on the critical line assuming suitable hypotheses on the zeros of individual $L$-functions including the Riemann hypothesis. A random matrices analogue for characteristic polynomials of independent random unitary matrices was later proven in \cite{Bar Hug}. 

However, adapting these arguments to the symplectic and orthogonal families of $L$-values is challenging even when one doesn't speak of linear combinations or shifts. This is due to the presence of low lying zeros, as predicted by Katz and Sarnak’s philosophy in \cite{KatzandSarnak, KatzSarnak2}. Establishing Selberg’s central limit theorem for these families requires a detailed understanding of the local statistics of their zeros, making the proof much more difficult. In the next sections, we will introduce two such families in the function field setting.

Inspired by the random matrix theory analogues, Keating and Snaith \cite{KeatingandSnaith} proposed conjectures about the asymptotic distribution of the real parts of logarithms of $L$-functions forming symplectic or orthogonal families in the spirit of Selberg's theorem. However, these conjectures remain out of reach, since the state of the art methods in analytic number theory have only been able to show that a positive proportion of critical $L$-values in a given family are nonzero. Moreover, even in the cases where the central value is known to be nonnegative, the logarithm remains highly sensitive to nearby low lying zeros.

In this paper, we explore the validity of a multidimensional version of Selberg's central limit theorem for symplectic and orthogonal families over function fields on the critical line where the  heights lie in three regimes to be categorized as one of microscopic, mesoscopic or macroscopic. Among important consequences are the emergence of the number of nontrivial zeros of the corresponding $L$-functions as Gaussian processes, and also the fact that correlations between these numbers match our expectations formed by corresponding results from the random matrix theory. In the next two sections, we define these symplectic and orthogonal families and present our results.


\subsection{Quadratic Dirichlet $L$-functions over Function Fields}\label{different regimes}


We begin by introducing key notations. Let $D\in \mathbb{F}_q[t]$  be a monic square-free polynomial, and define the associated primitive quadratic character using the Kronecker symbol $\left(\frac{D}{\cdot}\right)$.

The hyperelliptic ensemble $\mathcal{H}_{n,q}$, simplified as $\mathcal{H}_{n}$, is defined as
\[
\mathcal{H}_n :=\left\{ D\in \mathbb{F}_{q}[t] : \, D \text{ is monic, square-free, and} \, \deg(D)=n \right\}.
\] 
For each $D$ in $\mathcal{H}_n$, there is an associated hyperelliptic curve given by $C_D \,:\, y^2=D(t)$. Each such curve is nonsingular and of genus $g$ for 
\begin{align}\label{delta}
2g=n-1-\eta,
\end{align}
with
\begin{equation}\label{eq:eta}
\eta= \eta(D)=
\begin{cases}
1 \quad &\text{if } \,\, n=\operatorname{deg}(D) \text{ is even,}\\
0 \quad &\text{if } \,\, n=\operatorname{deg}(D) \text{ is odd.}
\end{cases}
\end{equation}
Note that $g\to \infty$ as $n$ does so. We will only work in the large degree aspect where $q$ fixed and $n$ varies. 


 Selberg's theorem for the family $\{L(\tfrac12, \chi_D)\}_{D\in \mathcal{H}_n}$ has been studied by the second and the third authors in \cite[Theorem 1.4]{DL}. They proved that as $D\in \mathcal{H}_n$ and $n\to \infty$,
\[
\frac{1}{\sqrt{\log n}}\left(\log\left| L\left(\tfrac{1}{2}, \chi_D\right)\right|-\tfrac{1}{2}\log n\right)  \longrightarrow N(0, 1)
\]
for the standard real Gaussian distribution $N(0,1)$, conditionally on the low lying zero hypothesis below.

\begin{hyp}
Let $\{\theta_{j, D}\}_{j=1}^{2g}$ be the eigenphases associated to a hyperelliptic curve $D$ of genus $g$. If $y=y(g)\to \infty$ then as $g\to \infty$,
	\[
	\frac{1}{| \mathcal{H}_{n}|}\#\Big\{D\in  \mathcal{H}_{n}: \min_{j} \big|\theta_{j, D}\big|<\frac{1}{yg}\Big\}=o(1).
	\]
\end{hyp}

We note that this hypothesis is motivated by the one-level density estimate \eqref{one level density},  as discussed in~\cite[Section~2.4]{DL}.

In our discussion, we want to study distributions of logarithms at points $\frac12+it$ for a real number $t$ away from the critical point. These shifts will provide insight into the statistical behavior of the logarithm of $L$-functions on the critical line where the height $t$ lies in various regimes.

We specify three regimes based on the size of this shift $t$ relative to the size of the genus $g$.

$\bullet$ {\it \bf{Microscopic regime:}} A shift 
$t$ is said to be in the {\it microscopic} if $|t|g<\infty \pmod{\frac{4\pi }{\log q}}$. This implies that $t$ is close to the height of a typical zero of $L(\frac12+it, \chi_D)$, meaning it is highly probable for $t$ to coincide with a zero. The behavior of the logarithm of $L$-functions in this regime is strongly influenced by low lying zeros. 

In~\cite[Theorem~2 and Eq.~(1.3)]{CCM}, the authors investigated the nonvanishing of various families of $L$-functions in the microscopic regime via one-level density estimates. We are thus justified in studying linear combinations of $L$-functions along microscopic shifts under the low lying zeros hypothesis.

$\bullet$ {\it \bf{Mesoscopic regime:}} A shift 
$t$ is {\it mesoscopic} if $|t|\to 0$ but $|t|g\to \infty \pmod{\frac{4\pi }{\log q}}$. In this regime, a low lying zeros hypothesis is not necessary to prove a result on the distribution of the logarithm.

$\bullet$ {\it \bf{Macroscopic regime:}}
A shift $t$ is said to be {\it macroscopic} if $|t|\in (0, 2\pi)$. Hence, $t$ is independent of the genus in this case.

Before stating our main results, we need to introduce further notations.

Let $k\geq 1$ be a fixed integer. For $\vec{a}=(a_1, \ldots, a_k), \vec{t}=(t_1, \ldots, t_k) \in \mathbb{R}^k, s\in \mathbb{C}$ and $D\in  \mathcal{H}_{n}$, let
\begin{equation}\label{def-L-linear combo}
\mathfrak{L}_{\vec{a}, \vec{t}}\left(s, \chi_D\right)
:=L\left(s+i t_1, \chi_D\right)^{a_1}\cdots L\left(s+it_k, \chi_D\right)^{a_k} .
\end{equation}
Additionally, we define terms to later stand for the mean of the real part of the logarithm
\begin{align}\label{mean for real}
\mathcal{M}(\vec{a}, \vec{t}, n)
:=\frac{a_1}{2} \log\Big(\min\big\{n, \tfrac{1}{2|t_1|}\big\}\Big)
+\cdots+\frac{a_k}{2} \log\Big(\min\big\{n, \tfrac{1}{2|t_k|}\big\}\Big) ,
\end{align}
and the variances of the real and imaginary parts of the distribution of the logarithm.
\begin{equation}\label{variance for real}
\begin{split}
 \mathcal{V}_{\operatorname{Re}}(\vec{a}, \vec{t}, n)
:=&\,
\frac12\Big(\sum_{j=1}^ka_j^2\Big)\log n
+\frac12 \sum_{j=1}^k a_j^2 \log\Big(\min\big\{n, \tfrac{1}{2|t_j|}\big\}\Big)
\\
&+\sum_{\substack{j_1, j_2=1\\j_1<j_2}}^k a_{j_1} a_{j_2}
 \log\Big( \min\big\{n, \tfrac{1}{|t_{j_1}-t_{j_2}|}\big\}\cdot 
\min\big\{n, \tfrac{1}{|t_{j_1}+t_{j_2}|}\big\}\Big).
\end{split}
\end{equation}
        \begin{equation}\label{variance for imaginary}
	\begin{split}
	\mathcal{V}_{\operatorname{Im}}(\vec{a}, \vec{t}, n)
	:=&\, \frac12\Big(\sum_{j=1}^k a_j^2\Big)\log n
	-\frac12 \sum_{j=1}^k a_j^2 \log \Big(\min\big\{n, \tfrac{1}{2|t_j|}\big\}\Big)
	\\
	&+\sum_{\substack{j_1, j_2=1\\j_1<j_2}}^k a_{j_1} a_{j_2} \log \Big(\min\big\{n, \tfrac{1}{|t_{j_1}-t_{j_2}|}\big\} \Big/\min\big\{n, \tfrac{1}{|t_{j_1}+t_{j_2}|}\big\}\Big).
\end{split}
\end{equation}
The following are our main unconditional results for the symplectic family \emph{near} the critical line.


\begin{thm}\label{Unconditional CLT for real near to half line:re} 
Fix $b\in\mathbb{R}$, and let $\sigma_0=\sigma_0(g)$ be a function of $g$, tending to $\frac12$ as $g\to \infty$ in such a way that $g \big(\sigma_0-\tfrac12\big)\to \infty$, but $g\big(\sigma_0-\tfrac12\big)=o\left(\sqrt{\log n}\right)$. Then for $D\in  \mathcal{H}_{n}$, 
	\[
	\frac{1}{| \mathcal{H}_n|}	
	\# \bigg\{ D\in  \mathcal{H}_n:
	 \frac{1}{\sqrt{ \mathcal{V}_{\operatorname{Re}}(\vec{a}, \vec{t}, n)}}
	\Big(\log\big|\mathfrak{L}_{\vec{a}, \vec{t}}\big(\sigma_0, \chi_D\big)\big| 
	-\mathcal{M}(\vec{a}, \vec{t}, n)\Big) >b\bigg\}
	 \longrightarrow \frac1{\sqrt{2\pi}}\int_b^\infty e^{-\frac{u^2}{2}}\mathop{du} 
	\]	
as $n\to \infty$. Here, $\mathfrak{L}_{\vec{a}, \vec{t}}$, $\mathcal{M}(\vec{a}, \vec{t}, n)$ and $\mathcal{V}_{\operatorname{Re}}(\vec{a}, \vec{t}, n)$ are defined by \eqref{def-L-linear combo}, \eqref{mean for real} and \eqref{variance for real} respectively.

Note that $\frac{1}{\sqrt{2 \pi}}\int_b^{\infty}e^{-\frac{u^2}{2}}du$ is the probability distribution function of the standard Gaussian at $b$.	
\end{thm}


\begin{thm}\label{Unconditional CLT for real near to half line:im}
Let $\sigma_0$ tend to $\tfrac12$ as $g\to \infty$ in such a way that $g\big( \sigma_0-\tfrac12\big)\to \infty$ but $g\big(\sigma_0-\tfrac12\big)=o\left(\sqrt{\log n}\right)$. Then for a real number $b$, 
	\[
	\frac{1}{| \mathcal{H}_{n}|}\# \bigg\{D\in  \mathcal{H}_n: \frac{\arg{\mathfrak{L}_{\vec{a}, \vec{t}}\big(\sigma_0, \chi_D \big)}}
	{\sqrt{ \mathcal{V}_{\operatorname{Im}}(\vec{a}, \vec{t}, n)}} > b
	\bigg\}
    \longrightarrow \frac1{\sqrt{2\pi}}\int_b^\infty e^{-\frac{u^2}{2}}\mathop{du}  \quad \text{as} \quad n\to \infty,
	\]
where $\mathfrak{L}_{\vec{a}, \vec{t}}$ and $\mathcal{V}_{\operatorname{Im}}(\vec{a}, \vec{t}, n)$ are defined by \eqref{def-L-linear combo} and \eqref{variance for imaginary} respectively. 	
\end{thm}

On the critical line, we will prove three theorems for $t$ in different regimes. We say that a $k$-dimensional vector $\vec{t}$ lies in a given regime only if all of its coordinates lie in that regime. Our next theorem is conditional. 


\begin{thm}\label{main theorem conditional microscopic}
Let $\vec{t}=(t_1, \ldots, t_k)$ belong to the microscopic regime and $b\in \mathbb{R}$. Assume that the low lying zeros hypothesis is true. Then
	\[
	\frac{1}{| \mathcal{H}_n|}	
	\# \bigg\{ D\in  \mathcal{H}_n: \frac{1}{\sqrt{ \mathcal{V}_{\operatorname{Re}}(\vec{a}, \vec{t}, n)}}
	\Big(\log\big|\mathfrak{L}_{\vec{a}, \vec{t}}\big(\tfrac12, \chi_D\big)\big| -\mathcal{M}(\vec{a}, \vec{t}, n)\Big) >b\bigg\}
	 \longrightarrow \frac1{\sqrt{2\pi}}\int_b^\infty e^{-\frac{u^2}{2}}\mathop{du}.
	\]	
\end{thm}

The second theorem, in contrast, is unconditional. It concerns a single $L$-function with a shift in the microscopic regime. 

\begin{thm}\label{mild conditioning lower bound}
	Assume that $t=\frac{2\pi\alpha}{n\log q}$ with $0\leq \alpha< \infty$, so that $t$ belongs to the microscopic regime. Let $b$ be a real number. Then as $n\to \infty$, 
	\begin{align*}
	 \frac{r(\alpha)}{\sqrt{2 \pi}}\int_b^{\infty}e^{-\frac{u^2}{2}}du+o_b(1)
	& \leq 
	 \frac1{| \CMcal{H}_{n}| }\#\bigg\{D\in  \CMcal{H}_{n} : 
	\frac{\log \left|L\left(\tfrac{1}{2}+i\frac{2\pi\alpha}{n\log q}, \chi_D\right)\right|
	-	\tfrac{ \log n}2}{\sqrt{\log n}}
	>b \bigg\} 
	\\
	& \leq  \frac{1}{\sqrt{2 \pi}}\int_b^{\infty}e^{-\frac{u^2}{2}}du+o_b(1),
	\end{align*}
where $r(\alpha)$ is defined by \eqref{values of nonvanishing} and it has the maximum  $\lim_{\alpha\to 0^+}r(\alpha)=0.9427\ldots$.
\end{thm}

The third theorem is multidimensional and unconditional, and applies to the other two regimes.   


\begin{thm}\label{main theorem conditional mesoscopic and macroscopic}
Let all the components of $\vec{t} \in \mathbb{R}^k$ lie in either the mesoscopic or the macroscopic regime. Then for a real number $b$, as $n\to \infty$,
       \[
	\frac{1}{| \mathcal{H}_n|}	
	\# \bigg\{ D\in  \mathcal{H}_n: \frac{1}{\sqrt{ \mathcal{V}_{\operatorname{Re}}(\vec{a}, \vec{t}, n)}}
	\Big(\log\big|\mathfrak{L}_{\vec{a}, \vec{t}}\big(\tfrac12, \chi_D\big)\big| -\mathcal{M}(\vec{a}, \vec{t}, n)\Big) >b\bigg\}
	 \longrightarrow \frac1{\sqrt{2\pi}}\int_b^\infty e^{-\frac{u^2}{2}}\mathop{du} 
	\]	
and 	
	\[
	\frac{1}{| \mathcal{H}_{n}|}\# \bigg\{D\in  \mathcal{H}_n: \frac{\arg{\mathfrak{L}_{\vec{a}, \vec{t}}\big(\tfrac12, \chi_D \big)}}{\sqrt{\mathcal{V}_{\operatorname{Im}}(\vec{a}, \vec{t}, n)}} > b
	\bigg\}
    \longrightarrow \frac1{\sqrt{2\pi}}\int_b^\infty e^{-\frac{u^2}{2}}\mathop{du},
	\]
where again the mean and variances are as given in \eqref{mean for real}, \eqref{variance for real} and \eqref{variance for imaginary}. 
\end{thm}
   
This shows that for $D \in \mathcal{H}_n$, as $n \to \infty$, the quantities $\log \lvert L(\tfrac12 + it_j, \chi_D) \rvert$ for $j = 1,\ldots, k$ are approximately independent only when 
$
\lvert t_{j_1} \pm t_{j_2} \rvert \gg 1/g^{o(1)}
$ for all $1\leq j_1\neq j_2\leq k$.
The arguments $\arg L(\tfrac12 + it_j, \chi_D)$ are always approximately independent provided that the $t_j$ themselves lie in the mesoscopic regime.

\subsection{$L$-functions of Quadratic Twists of Elliptic Curves over Function Fields} 

Let $E$ be an elliptic curve defined over rational numbers. The distribution of the logarithm of the (nonzero) central $L$-values $L(\tfrac{1}{2}, E\otimes \chi_d)$, where $d$ varies over fundamental discriminants was studied by Radziwi\l\l{} and Soundararajan \cite[Theorem 2]{SR}. They established a one-sided central limit theorem for this orthogonal family, providing evidence for the Keating-Snaith conjectures.

A key feature of their approach is the use of truncated Euler products, which are particularly amenable to analysis. These are easily invertible, and the distribution of their logarithms can be effectively studied via the method of moments. They first demonstrated that, in most cases, the logarithms of these Euler products are small, and that in such cases, one can accurately approximate the full Euler product by short Dirichlet polynomials. Since $L(\tfrac12, E\otimes \chi_d)\geq 0$, this approach is quite effective. For nonzero heights $t$, however, analyzing $L(\tfrac12+it, E\otimes \chi_d)$ is more delicate, since then the imaginary part of its logarithm comes into play.

In~\cite{SRlower}, Radziwi\l\l{}  and Soundararajan established a general principle that lower bounds on the proportion of nonvanishing central $L$-values can be obtained by studying the one-level density of low lying zeros. Their work leads to a lower bound, conditionally on the generalized Riemann hypothesis for $L(s, E\otimes \chi)$ for every Dirichlet character $\chi$, for the count
\[
g(X; b) := 
\#\bigg\{
d \in \mathcal{E},\, X < |d| \leq 2X :
\frac{\log L\left( \tfrac{1}{2}, E\otimes\chi_d \right) + \tfrac{1}{2} \log \log |d|}{\sqrt{ \log \log |d| }}
\geq b
\bigg\},
\]
where $b$ is any fixed real number$, \varepsilon_E(d)$ denotes the root number, $N_E$ is the conductor of $E$, and
\[
\mathcal{E} = \left\{ d :\, d \text{ is a fundamental discriminant with } (d, 2N_E) = 1 \text{ and } \varepsilon_E(d) = 1 \right\}.
\]
In other words, by combining \cite[Theorem~2]{SR} with \cite[Theorem~1]{SRlower}, they established that
\begin{align*}
\#\{d\in \mathcal{E}:\, X<|d|\leq 2X\}
&\left(\frac14 \frac{1}{\sqrt{2\pi}}\int_b^\infty e^{-\frac{u^2}{2}}du+o(1)\right)
\\
& \leq g(X; b)
\leq \#\{d\in \mathcal{E}:\, X<|d|\leq 2X\}\left( \frac{1}{\sqrt{2\pi}}\int_b^\infty e^{-\frac{u^2}{2}}du+o(1)\right).
\end{align*}

In our work, $E$ is a fixed elliptic curve over the function field $\mathbb{F}_q[t]$.
For fixed vectors $\vec{a}, \vec{t} \in \mathbb{R}^k$, and for $D\in \mathcal{H}_{n}$, we set
\begin{equation}\label{def-L-linear comboe}
\mathfrak{L}_{\vec{a}, \vec{t}}\left(s, E\otimes\chi_D\right)
:=L\left(s+i t_1, E\otimes\chi_D\right)^{a_1}\cdots L\left(s+it_k, E\otimes\chi_D\right)^{a_k} .
\end{equation}
Define
\[
\mathcal{H}^{\Delta}_{n} := \left\{D \in \mathbb{F}_q[t] :  D \text{ is monic and square-free,} \deg D = n, 
 (D, \Delta) =1\right\},
\]
where $\Delta$ is the discriminant of the elliptic curve $E$ such that $\deg_t
(\Delta)$ is minimal. We will base our discussion on the set of polynomials, $ \mathcal{H}^{\Delta, +}_{n}$, which is the subset of  $D\in  \mathcal{H}^{\Delta}_{n}$ for which $L(s, E\otimes \chi_D)$ has root number $1$.  
Let $m$ be the conductor of $L(s, E\otimes \chi_D)$ defined as in \eqref{m and n connection for elliptic curve}. Note that $m\to \infty$ if and only if $n\to \infty$, given that $E$ is fixed. 

Analogously to the low lying zeros hypothesis in Section~\ref{different regimes}, one can formulate the following which will be needed to study the distribution at microscopic shifts. Note that the terms microscopic, mesoscopic and macroscopic are as defined in Section \ref{different regimes}, and a $k$-dimensional vector $\vec{t}$ lies in a given regime only if all of its coordinates lie in that regime. 

\begin{shifted hyp E}\label{low lying zeros for elliptic curve}
For each $(C, N_E)=1$, let $\{\theta_{j, C}(E\otimes \chi_D)\}_{j=1}^m$ be the eigenphases associated to $\{L(s, E\otimes \chi_D)\}_{D\in  \CMcal{H}_{n}^{\Delta}(C)}$. If $y=y(m)\to \infty$, then as $m\to \infty$
	\[
	\frac{1}{\big| \CMcal{H}_{n}^{\Delta,+}\big|}
	\# \bigg\{D\in  \CMcal{H}_{n}^{\Delta}(C): \min_{j}\Big|\theta_{j, C}(E\otimes \chi_D)\Big|< \frac{1}{ym} \bigg\} =o_E(1).
	\]
\end{shifted hyp E}

Conditionally on this type of low-lying zeros hypothesis, one can readily establish an analogue of Theorem~\ref{main theorem conditional microscopic} for $
\log\bigl|\mathfrak{L}_{\vec{a}, \vec{t}}\bigl(\tfrac12, E\otimes\chi_D\bigr)\bigr|$
with microscopic $\vec{t}$ and $D\in \mathcal{H}^{\Delta,+}_{n}$, using the material developed for the two families. We omit the statement in order to concentrate on the unconditional results, and also to avoid potential confusion arising from the presence of two similar hypotheses.

The theorem below is a function field analogue of the result by Radziwi\l\l{} and Soundararajan, where the point is close to $\frac12$ by a microscopic shift. 


\begin{thm}\label{unconditional result for orthogonal family}
	Assume that $\alpha\in \mathbb{R}$, so that $\frac{2\pi\alpha}{m\log q}$ falls into microscopic regime. For $b$ a real number, as $m\to \infty$,
	\begin{align*}
	\frac{r_E(\alpha)}{\sqrt{2 \pi}}\int_b^\infty e^{-\frac{u^2}{2}}du+o_{E, y}(1)
	&\leq \frac1{| \mathcal{H}^{\Delta, +}_{n}|} \# \bigg\{D\in \mathcal{H}^{\Delta, +}_{n}\, : \, \frac{\log\left|L\left(\tfrac12+i\frac{2\pi\alpha}{m\log q}, E\otimes\chi_D\right)\right| +\tfrac12 \log m}{\sqrt{\log m}}
	 >b\bigg\}
	 \\
	&\leq  \frac{1}{\sqrt{2 \pi}}\int_b^\infty e^{-\frac{u^2}{2}}du+o_{E, y}(1),
	\end{align*}
where $r_E(\alpha)$ is as defined in \eqref{r_E} and has the maximum $\lim_{\alpha\to 0^+}r_E(\alpha)=\frac14$. 
	\end{thm}


\begin{thm}\label{Uncon mes or micro}
Suppose that $\vec{t}$ lies in either of the mesoscopic or macroscopic regimes. Then for a real number $b$, as the conductor $m$ of $E$ tends to $\infty$, 
\begin{align*}
\frac1{|\mathcal{H}^{\Delta, +}_{n}|}\#\bigg\{D\in \mathcal{H}^{\Delta, +}_{n}\, : \, 
\tfrac{1}{\sqrt{ \mathcal{V}_{\operatorname{Re}}(\vec{a}, \vec{t}, m)}}
\Big(\log\left|\mathfrak{L}_{\vec{a}, \vec{t}}\left(\tfrac12, E\otimes\chi_D\right)\right| +\mathcal{M}(\vec{a}, \vec{t}, m)\Big)
>b\bigg\}
 \longrightarrow \frac{1}{\sqrt{2 \pi}}\int_b^\infty e^{-\frac{u^2}{2}}du,
	\end{align*}
and 
	\[
	\frac1{|\mathcal{H}^{\Delta, +}_{n}|}
	\#\bigg\{D\in \mathcal{H}^{\Delta, +}_{n}\, : \, \frac{\arg\mathfrak{L}_{\vec{a}, \vec{t}}\left(\tfrac12, E\otimes\chi_D\right)}{\sqrt{ \mathcal{V}_{\operatorname{Im}}(\vec{a}, \vec{t}, m)}}
	>b\bigg\}
	 \longrightarrow\frac{1}{\sqrt{2 \pi}}\int_b^\infty e^{-\frac{u^2}{2}}du,
	\]
where $\mathcal{M}(\vec{a}, \vec{t}, m)$, $\mathcal{V}_{\operatorname{Re}}(\vec{a}, \vec{t},m)$, and $\mathcal{V}_{\operatorname{Im}}(\vec{a}, \vec{t}, m)$ are defined by \eqref{mean for real}, \eqref{variance for real} and \eqref{variance for imaginary}, respectively.
\end{thm}
	
Note that, our approach in proving Theorems \ref{unconditional result for orthogonal family} and \ref{Uncon mes or micro} differs fundamentally from that of results of Radziwi\l\l{} and Soundararajan \cite{SR, SRlower}. We employ Selberg’s approximation method, which enables us to study the distribution of $\arg L\big(\frac12, E\otimes \chi_D\big)$ in addition to the real part of the logarithm. A study of the imaginary part of the logarithm has been out of reach via the techniques used by Radziwi\l\l{} and Soundararajan.

We will now add some general remarks about the proofs of our theorems. Our approach follows a different path than Bourgade’s method in \cite{Bou}. The proofs of our results rely on Selberg’s approximation of  $\log L(s, \mathcal{F})$ by a Dirichlet polynomial, where $\mathcal{F}$ is either $\chi_D$ or $E\otimes \chi_D$, and $s$ is slightly away from the critical line due to the influence of low lying zeros. We begin by computing large moments of linear combinations of these Dirichlet polynomials at points with different heights. One of the key steps is to show that the discrepancy between the logarithm of $L$-functions near and on the critical line is small. The analysis varies based on the regime that the heights belong to, and requires separate treatment for the real and imaginary parts of the linear combinations of logarithms of $L$-functions.

Notably, low lying zeros are significant only at the average height of the zeros, which is within the microscopic regime. In contrast, the distribution at the mesoscopic regime seems to remain unaffected by low lying zeros, allowing for a well-structured correlation of fluctuations, as described in \eqref{correlation structure}. Within the microscopic regime, we cannot draw conclusions about the distribution of the imaginary part of the linear combinations of logarithms of $L$-functions, as the variance of the linear combinations of the suitable Dirichlet polynomials tends to zero.


\subsection{Applications}\label{Application}


\subsubsection{Gaussian Processes}
\begin{definition*}
Let $G=\{(w_x)_{x\in I}\}$ be a stochastic process for a continuous time or location, where $I$ is a nonempty subset of $\mathbb{R}$. 

G is called a Gaussian process if for any finite collection of  times or locations $x_1, \ldots, x_\ell\in I$, the vector $ (w_{x_1}, \ldots, w_{x_\ell})$ is a multivariate Gaussian random variable. That is, every linear combination of $(w_{x_1}, \ldots, w_{x_\ell})$ has a Gaussian distribution. A Gaussian process is said to be mean stationary if the mean values are in the form $\mathbb{E}(w_{x})=\mu$ for all $x$.
\end{definition*}

We consider the stochastic processes
\begin{align*}
\mathfrak{R}=\bigg\{\bigg(\frac{\log |\mathcal{L}(q^{-1/2}e(\theta_{\delta}), \cF)|-(\epsilon_{\cF}/2) \log(\kappa_\cF)}{\sqrt{1/2(1+\delta)\log(\kappa_\cF)}}: \theta_\delta=\frac{1}{\kappa_\cF^\delta}\bigg)_{0\leq \delta < 1}\bigg\}
\end{align*}
and 
\[
\mathfrak{I}=\bigg\{\bigg(\frac{ \arg \mathcal{L}(q^{-1/2}e(\theta_\delta), \cF)}{\sqrt{1/2(1-\delta)\log(\kappa_\cF)}}: \theta_\delta=\frac{1}{\kappa_\cF^\delta}\bigg)_{0\leq \delta < 1}\bigg\},
\]
where $\kappa_\cF$ represents either $2g$ or $m$ depending on the family under consideration. Also, the sign of the mean value is given by
\begin{align}\label{sign of MV}
\epsilon_{\mathcal{F}} =
\begin{cases}
1  & \text{if } \mathcal{F} = \chi_D, \\[6pt]
-1  & \text{if } \mathcal{F} = E \otimes \chi_D.
\end{cases}
\end{align}

The following corollary is a direct application of Theorems \ref{Unconditional CLT for real near to half line:re} and \ref{Unconditional CLT for real near to half line:im}.


\begin{cor}\label{Gaussian process}
Both $\mathfrak{R}$ and $\mathfrak{I}$ form mean stationary Gaussian processes as $\kappa\to \infty$. Moreover, the covariance matrix of $\mathfrak{I}$ is an identity matrix while the entries of covariance matrix for $\mathfrak{R}$ are given by 
	\begin{align*}
	\operatorname{Cov}(\mathfrak{R}_{j_1}, \mathfrak{R}_{j_2})=
	\begin{cases}
	1 & \text{ if }\,\, j_1=j_2, \\ 
	\tfrac{2 \left(\delta_{j_1}\wedge \delta_{j_2}\right)}{\sqrt{1+\delta_{j_1}} \sqrt{1+\delta_{j_2}}} & \text{ otherwise}. 
	\end{cases} 
	\end{align*}
Here for $i\neq j$,
\[
x_i\wedge x_j := \min\big\{x_i, x_j\big\}.
\]		 
Thus, the covariance matrix of $\mathfrak{R}$ has nonzero entries except its first row and column, and all of its diagonal entries are 1. 
\end{cor}


\subsubsection{Fluctuations of the Number of Zeros}

In this section, our goal is to understand the fluctuations of the number of zeros of $\mathcal{L}(u, \mathcal{F})$ on a circular arc of radius $q^{-1/2}$.
Define 
\[
S(\theta, \mathcal{F}):=\frac{1}{\pi} \arg \mathcal{L}\big(q^{-1/2}e(\theta), \mathcal{F}\big)=\frac{1}{\pi} \operatorname{Im} \log \mathcal{L}\big(q^{-1/2}e(\theta), \mathcal{F}\big).
\]
Set $e(x) = \exp(2\pi i x)$. Let $N(\theta, \mathcal{F})$ denote the number of zeros of $\mathcal{L}(u, \mathcal{F})$ on the circular arc $q^{-1/2}e(\psi)$ with $0\leq \psi \leq \theta\leq 1$. That is,
\[
N(\theta, \mathcal{F})=\sum_{\theta_j\leq \theta}1.
\]
Let $\theta \in [0,1)$. Following the approach in \cite[Theorem 7]{AGK} for $\mathcal{F} = \chi_D$, and applying a similar argument together with the functional equation in \eqref{eq:functional_eq} for $\mathcal{F} = E \otimes \chi_D$, we obtain 
\[
N(\theta, \mathcal{F})=\kappa_\cF \theta+ S(\theta, \mathcal{F}). 
\]

We define for $0\leq \theta_1 < \theta_2$,
\[
\Delta(\theta_1, \theta_2, \mathcal{F})
=N(\theta_2, \mathcal{F})-N(\theta_1, \mathcal{F})+\kappa_\cF(\theta_2-\theta_1),
\]
which represents the number of zeros of $\mathcal{L}(u, \mathcal{F})$ with the angles between $\theta_1$ and $\theta_2$. As a direct consequence of second parts of Theorems \ref{main theorem conditional mesoscopic and macroscopic} and \ref{Uncon mes or micro}, we have the following corollary.


\begin{cor}
	Let $\delta_1, \delta_2\in [0, 1)$ such that $ \theta_1\asymp \frac{1}{\kappa_{\mathcal{F}}^{\delta_1}}, \theta_2 \asymp \tfrac{1}{\kappa_{\mathcal{F}}^{\delta_2}}$. Then as $\kappa\to \infty$,
	\[
	\frac{\Delta(\theta_1, \theta_2, \mathcal{F})}{\frac{1}{\pi}\sqrt{\log (\kappa_{\mathcal{F}})}}\longrightarrow \sqrt{1-\frac{\delta_1+\delta_2}{2}}N(0, 1).
	\]
	in probability law.
\end{cor} 

The above result recovers a result of Faifman--Rudnick \cite[eq. (1.3)]{FR} for the family $\mathcal{H}_{2g+2}$ by using the explicit formula. It was also obtained independently by Faifman \cite[Theorem 3 (i)]{Faif} via Selberg's technique for the symplectic family.

The following result is another application of Theorems \ref{main theorem conditional mesoscopic and macroscopic} and \ref{Uncon mes or micro}. It can be seen as evidence towards repulsion of the zeros of $L(s, \chi_D)$ (and $L(s, E\otimes \chi_D)$) over functions fields in the symplectic (and orthogonal) family in the mesoscopic regime. The correlations appear when counting the zeros of $\mathcal{L}(u, \mathcal{F})$ that fall in distinct intervals. The result can be compared to  \cite[Corollary 1.3]{Bou}, \cite[Theorem 3.1]{CD} and \cite[Theorem 1 and (3)]{Wie}.


\begin{cor}\label{mesoscopic fluctuation}
	Let $\delta\in [0, 1)$. Consider the stochastic process 
	\[
	\mathcal{W}=\left(\frac{\Delta\left(\theta_1, \theta_2, \mathcal{F}\right)}{\frac{1}{\pi}\sqrt{(1-\delta)\log (\kappa_{\mathcal{F}})}}: \theta_j=\frac{\alpha_j}{\kappa_{\mathcal{F}}^{\delta}}, j=1, 2\right)_{0\leq \alpha_1<\alpha_2<\infty}.
	\]
	Then $\mathcal{W}$ follows a Gaussian process $(Z(\alpha_1, \alpha_2): 0\leq \alpha_1<\alpha_2<\infty)$ as $\kappa_{\mathcal{F}}\to \infty$ with the covariance function
	\begin{align}\label{correlation structure}
	\operatorname{Cov}\left(Z(\alpha_1, \alpha_2), Z(\alpha_3, \alpha_4)\right)=\begin{cases}
	1 & \text{ if } \alpha_1= \alpha_3,  \alpha_2=\alpha_4\\ 
	\frac12 & \text{ if }  \alpha_1= \alpha_3,  \alpha_2\neq\alpha_4,\\
	\frac12 & \text{ if }  \alpha_1\neq \alpha_3,  \alpha_2=\alpha_4,\\
	-\frac12 & \text{ if }  \alpha_2=\alpha_3, \text{ or }  \alpha_1=\alpha_4,\\
	0 & \text{ otherwise}.
	\end{cases}
	\end{align}
\end{cor}

The above theorem can be interpreted as follows. $Z(\alpha_1, \alpha_2)$ and $Z(\alpha_3, \alpha_4)$ are independent as long as the interval $[\alpha_1, \alpha_2]$ is strictly included in $[\alpha_3, \alpha_4]$, but they correlate if the inclusion is not strict.


\subsubsection{Future direction of research on extreme values}\label{subsubsec:future}

The preceding theorems reveal a log-correlated Gaussian field structure for the orthogonal and symplectic families of $L$-functions over function fields. The covariance structure is likely governed by the real part of the Dirichlet polynomial given in \eqref{correlation of reals} at various shifts. For instance, due to our computations in the proof of Proposition \ref{Gaussian for real part of DP}, the covariance between
$\log |L(1/2+it_1, \mathcal{F})|$ and
$\log |L(1/2+it_2, \mathcal{F})|$ is expected to be
\begin{equation}\label{cov struc}
\frac12
\log \min\left\{\kappa_{\mathcal{F}}, \frac{1}{|t_1 - t_2|}\right\}
+
\frac12\log \min\left\{\kappa_{\mathcal{F}}, \frac{1}{|t_1 + t_2|}\right\}
,
\end{equation}
whereas for the unitary family of Riemann zeta functions,  as shown in \cite{Bou}, one has for $T$ is chosen uniform from the interval $t\in [T, 2T]$, that
\[
\mathrm{Cov}\!\left(\log\big|\zeta\big(\tfrac12+it+it_1\big)\big|,\; \log\big|\zeta\big(\tfrac12+it+it_2\big)\big|\right)
=
\frac12 \log \min\left\{\log T,\; \frac{1}{|t_1 - t_2|}\right\}.
\]
 The appearance of $|t_1 + t_2|$ in~\eqref{cov struc} reflects the symmetry induced by the conjugate-pair structure of zeros in the orthogonal and symplectic families. In contrast, Bourgade’s covariance formula for the Riemann zeta function depends solely on $|t_1 - t_2|$ and such covariance structures are characteristic of log-correlated Gaussian fields.

Using this log-correlated framework for $\zeta\big(\tfrac12+it\big)$, Arguin, Belius, Bourgade, Radziwi{\l\l} and Soundararajan initiated a rigorous investigation of a conjecture of Fyodorov, Hiary and Keating (FHK) from \cite{FHK1, FHK2} concerning the maximum of $\zeta\big(\tfrac12+it\big)$ on random intervals of fixed length. They established the main term predicted by the conjecture~\cite{ABBRS}. The conjecture in its full form was later proven in the works of Arguin, Bourgade and Radziwi{\l \l} \cite{ABR20, ABR23}.

A key conceptual ingredient in these developments is the branching random walk analogy, which arises from the log-correlated covariance structure as depicted in \cite[Figure~1]{ABBRS}. A study of the maximum of the characteristic polynomials of random unitary matrices towards the FHK conjecture was conducted in \cite{ABB}. In contrast to the unitary case, for each of the orthogonal and symplectic random matrix ensembles, the corresponding branching random walk is expected to differ substantially, and may take the form of an inhomogeneous branching random walk arising from Dyson Brownian motion as described in \cite[Section~1.2.1]{CipLan} by Cipolloni and Landon. In \emph{loc cit}, they prove a result towards the FHK conjecture for the maximum of the characteristic polynomials of i.i.d. random nonHermitian matrices. This is a recent development that makes it compelling to study the FHK conjecture for the orthogonal and symplectic families of $L$-functions that we study in this paper.

In our setting, in view of \cite{CMN}, a natural formulation of the FHK conjecture is as follows.  
Let $\Delta$ be fixed. As $n\to\infty$, for $D$ ranging over either $\mathcal{H}_n$ or $\mathcal{H}_n^{\Delta}$,
\[
\max_{\theta\in[0,\pi)} 
\log \bigl|\mathcal{L}(q^{-1/2} e(\theta), \mathcal{F})\bigr|
=\sqrt{\frac{2}{\beta}}
\Big(\log n - \frac{3}{4}\log\log{n} + o(1)\Big),
\]
where $\beta$ denotes the Dyson index which is $4$ for the symplectic family, and $1$ for the orthogonal family. It is worth noting that, for the families we consider, establishing such a result for microscopic shifts likely requires the assumption of a low lying zeros hypothesis for the real part of the logarithm. For the imaginary part, at present it is not known whether Gaussian fluctuations are exhibited at microscopic shifts. This stands in contrast with the work of Najnudel~\cite{Naj}, who obtained the leading term for the maximum as proposed by the FHK conjecture both for the real and imaginary parts of $\log \zeta\big(\tfrac12+it\big)$.

\subsubsection{Towards ratio conjecture} 
The ratio conjecture for the symplectic family of quadratic Dirichlet L-functions associated with hyperelliptic curves was formulated in \cite[Conjecture 6]{AK}. Our main theorem allows one to directly obtain a Selberg-type central limit theorem for the ratio of these L-functions. This, in turn, provides a useful framework for studying the ratio conjecture, since the analysis requires upper bounds for the large deviations associated with the central limit theorem (see \cite{Sound2009}). In particular, our result gives a heuristic prediction for the order of magnitude appearing in the conjecture using \eqref{variance for real} and \eqref{variance for imaginary}, in accordance with the discussion on page 2 of \cite{Sound2009}.  
\subsection{Random Matrix Theory Perspective}\label{RMT side}

Connections between the fields of analytic number theory and random matrix theory equips us with more tools and conjectures to study logarithms of $L$-functions. These connections are based on Montgomery's work~\cite{Mont} on pair correlation of ordinates of nontrivial zeros of the Riemann zeta function $\zeta(s)$ and an observation of Dyson that the same pair correlation structure is exhibited by eigenangles of random unitary matrices in $\mathcal{U}(N)$ as $N$ gets large.

Keating and Snaith \cite{KeatingandSnaith} studied the correspondence between nontrivial zeros of $\zeta(s)$ with ordinates of size $T$ and the eigenvalues of $n\times n$ unitary matrices in $\mathcal{U}(N)$, where $N$ is of order $\log T$. In particular, letting $Z(U_n, \theta)$ denote the characteristic polynomial of a matrix $U$ in $\mathcal{U}(N)$, they showed that the real and imaginary parts of $\log Z(U, \theta)/(1/2)\log N$ have independent standard Gaussian distributions as $N\to \infty$. This coincides precisely with Selberg’s theorem. In the theory of spectral statistics, the natural parameter is the mean eigenvalue spacing. For the eigenphases $\theta_n$ of $U$, this is $2\pi/N$.
Similarly, the mean spacing between the zeros $t_n$ of $\zeta(s)$ at height $T$ along the critical line is $2\pi/ \log(T/2\pi)$. 

Various statistical tests based on random matrix theory have been applied to both local properties (such as eigenvalue spacings) and global properties (such as moments and variances). As we will see, for local properties, when suitably normalized, eigenphases of random unitary matrices exhibit an excellent fit to the ordinates of zeta zeros. However, for global properties, this correspondence can sometimes break down. Still, Conrey et al. \cite{Conreyetal} were able to derive detailed heuristic asymptotic expansions for moments of $\zeta(s)$ via a random matrix theory approach.

As for the local properties, Wieand \cite{Wie, Wiethesis} investigated the joint distribution of eigenvalues of random elements in $\mathcal{U}(N)$ across different intervals on the unit circle. She found that, when normalized by mean and variance, the number of eigenvalues in these intervals follows a characteristic correlation structure given in \eqref{correlation structure}. Building on this work, Coram and Diaconis \cite{CD} conducted numerical tests on the zeros of the zeta function, which Bourgade \cite{Bou} later confirmed. 

Later in \cite[Theorem 6.1]{DE}, Diaconis and Evans revisited Wieand’s results on the number of eigenvalues within a fixed arc and reaffirmed the correlation structure in \eqref{correlation structure} by using an elegant and elementary approach based on linear combinations of traces.

The statistical properties of eigenvalues of random matrices from classical compact groups, particularly the unitary, symplectic and orthogonal groups, exhibit striking similarities to those of zeros of families of $L$-functions (see~\cite{KatzSarnak2}). In addition to the unitary family, Keating and Snaith \cite{KeatingandSnaith} also studied the distribution of the logarithm of characteristic polynomials at the symmetry point $(\theta=0)$ for symplectic and orthogonal families of random matrices. Regarding the local eigenvalue statistics of symplectic and orthogonal matrices, Diaconis and Evans~\cite[Section~8]{DE} suggested, and provided a sketch of the argument, that the correlation structure \eqref{correlation structure} should hold. Wieand’s approach~\cite{Wie} does not apply in this setting because of the complicated behavior of the corresponding Toeplitz matrices. Motivated by these insights from the random matrix theory side, we were motivated to investigate the $L$-functions side and determine whether similar local statistics of zeros arise. In fact, in Corollary \ref{mesoscopic fluctuation}, we confirm that the same phenomenon holds for our symplectic and orthogonal families of $L$-functions over function fields which will be introduced in Sections \ref{sec:L fnc s} and \ref{sec:L fnc o}. 

Let $A$ be a random matrix chosen from the set of symplectic matrices $\text{USp(2N)}$ equipped with the Haar measure. Its eigenvalues are known to lie on the unit circle and to come in conjugate pairs as 
\[
e^{i\theta_1}, e^{-i\theta_1}, \ldots, e^{i\theta_N}, e^{-i\theta_N}.
\]
The characteristic polynomial of $A$ is given
by
\begin{equation}\label{characterristic poly}
Z_{A}(\theta)=\prod_{j=1}^N\left(1-e^{i(\theta_j-\theta)}\right)\left(1-e^{-i(\theta_j+\theta)}\right).
\end{equation}
Supposing that $Z_{A}$ is a good model for an $L$-function from a symplectic family, we will compute some mean values related to its logarithm in order to later compare them with those in our theorems in Section \ref{different regimes}.

By Mercator series expansion and \cite[Lemma 2]{KeatingOdgers}, we obtain the following.
\begin{align*}
\mathbb{E}_{\text{USp(2N)}}\big(\log Z_A(\theta)\big)
&=-\sum_{k=1}^{\infty}\frac{e^{-ik\theta}}{k}\mathbb{E}_{\text{USp(2N)}}\sum_{j=1}^N\big(e^{ik\theta_j}+e^{-ik\theta_j}\big)\\
&=-\sum_{k=1}^{\infty}\frac{e^{-ik\theta}}{k}\mathbb{E}_{\text{USp(2N)}}\big(\operatorname{Tr}(A^k)\big)
=\frac12\sum_{k=1}^{N}\frac{e^{-2i\theta k}}{k}.
\end{align*}
Then, Lemma \ref{logarithmic cosine sum} directly implies that 
\[
\mathbb{E}_{\text{USp(2N)}}\big(\operatorname{Re} \log Z_A(\theta)\big)
=\frac{1}{2}\log \min\Big\{N, \frac{1}{2|\theta|}\Big\}
\quad \text{and} \quad
 \mathbb{E}_{\text{USp(2N)}}\big(\operatorname{Im} \log Z_A(\theta)\big)=O(1).
\]
These coincide with the {\it means $($averages$)$} appearing in Theorem~\ref{main theorem conditional mesoscopic and macroscopic} for the real and imaginary parts of the logarithm of $L$-functions over function fields. From the above, we also observe that computing the mean of the logarithm of characteristic polynomials over the family $\mathrm{USp}(2N)$ amounts to computing the mean of the trace of any power of the matrix $A$. Now, on the $L$-function side, we have an analogous product expression to~\eqref{characterristic poly}, which is $\eqref{eq:L as polyl}$, but it lacks a well-structured description of the zeros of $L$-functions. So, we will use Dirichlet polynomial approximations of $\log L(s,\chi_D)$ and compute their moments instead.  

Now, we proceed to compute the joint mean of the characteristic polynomial at two different shifts, say $t_1, t_2\in [0, \pi)$. Using \cite[Lemma 2]{KeatingOdgers} about the mean of the product of trace of different powers of $A$, we find that
\begin{align*}
&\mathbb{E}_{\text{USp(2N)}}\big(\operatorname{Re}\left( \log Z_A(t_1) \log Z_A(t_2)\right)\big)
\\
&=\frac14 \sum_{\epsilon_1, \epsilon_2\in \{-1, +1\}}\sum_{j, k=1}^{\infty}\frac{e^{i(\epsilon_1 jt_1+ \epsilon_2 k t_2)}}{jk}\mathbb{E}_{\text{USp(2N)}}\big(\operatorname{Tr}(A^j)\operatorname{Tr}(A^k)\big)
\\
&=\frac14 \sum_{\epsilon_1, \epsilon_2\in \{-1, +1\}}\bigg[\sum_{j=1}^{2N}\frac{e^{ij(\epsilon_1 t_1+ \epsilon_2 t_2)}}{j}+2N\sum_{j>2N}\frac{e^{ij(\epsilon_1 t_1+ \epsilon_2 t_2)}}{j^2}-\sum_{N+1\leq j\leq 2N}\frac{e^{i(\epsilon_1 t_1+ \epsilon_2 t_2)}}{j^2}\\
& \hspace{3cm} +\sum_{\substack{1\leq j\leq k\leq N}}\frac{e^{i(\epsilon_1 t_1 j+ \epsilon_2 t_2 k)}}{jk}-\sum_{\substack{j, k=1\\ j-k>N+1\\ k-j\leq N}}^{\infty}\frac{e^{i(\epsilon_1 t_1 j+ \epsilon_2 t_2 k)}}{jk}\bigg],
\end{align*}
where the sums are coming from correlation of traces. One can show that contributions of the second, third and fifth sums are $O(1)$. Therefore, at the end, again using Lemma \ref{logarithmic cosine sum}, we obtain
\begin{align*}
&\mathbb{E}_{\text{USp(2N)}}\big(\operatorname{Re}\left( \log Z_A(t_1) \log Z_A(t_2)\right)\big)
\\
&=\frac12 \sum_{j\leq 2N}\frac{\cos ((t_1+t_2)j) +\cos ((t_1-t_2)j)}{j}
+\bigg(\sum_{j\leq N}\frac{\cos(jt_1)}{j}\bigg)\bigg(\sum_{j\leq N}\frac{\cos (jt_2)}{j}\bigg)+O(1)
\\
&=\frac{1}{2}\log\bigg( \min\Big\{N, \tfrac{1}{|t_1-t_2|}\Big\} 
\min\Big\{N, \tfrac{1}{|t_1+t_2|}\Big\}\bigg)
+\frac14 \log \min\bigg\{N, \frac{1}{2|t_1|}\bigg\}\log \min\bigg\{N, \frac{1}{2|t_2|}\bigg\}+O(1).
\end{align*}
In a similar manner, we have
\[
\mathbb{E}_{\text{USp(2N)}}\big(\operatorname{Im}\left(\log Z_A(t_1) \log Z_A(t_2)\right)\big)
=\frac{1}{2}\log\Big( \min\big\{N, \tfrac{1}{|t_1-t_2|}\big\}\Big/
\min\big\{N, \tfrac{1}{|t_1+t_2|}\big\}\Big)+O(1).
\]
Likewise, one finds
\[
\mathbb{E}_{\text{USp(2N)}}\big(\operatorname{Re} \log^2 Z_A(t)\big)
=\frac12 \log N+\frac12 \log \min\big\{N, \tfrac{1}{2|t|}\big\}
+\frac14 \log^2 \min\big\{N, \tfrac{1}{2|t|}\big\}+O(1), 
\]
and 
\[
\mathbb{E}_{\text{USp(2N)}}\big(\operatorname{Im} \log^2 Z_A(\alpha)\big)
=\frac12 \log N-\frac12 \log \min\big\{N, \tfrac{1}{2|\alpha|}\big\}+O(1).
\]
The above computations provide the following expressions for the variance and covariance.  
\[
\text{Var}\big(a_1\operatorname{Re} \log Z_A(t_1)+ a_2 \operatorname{Re} \log Z_A(t_2)\big)
=\mathcal{V}_{\operatorname{Re}}(\vec{a}, \vec{t}, N)+O(1),
\]
and 
\[
\operatorname{Cov}\big(\operatorname{Re} \log Z_A(t_1), \operatorname{Re} \log Z_A(t_2)\big)
=\frac{1}{2}\log\left( \min\left\{N, \tfrac{1}{|t_1-t_2|}\right\} 
\min\left\{N, \tfrac{1}{|t_1+t_2|}\right\}\right),
\]
where $\vec{a} = (a_1, a_2)$, $\vec{t} = (t_1, t_2)$, and $\mathcal{V}_{\operatorname{Re}}(\vec{a}, \vec{t}, N)$ is given as in \eqref{variance for real}. This covariance also matches the covariance found on the $L$-function side, given by \eqref{eq:cov RXtjD}. A similar result holds for the imaginary parts of the logarithms of linear combinations of characteristic polynomials at different shifts as shown in \eqref{covariance for imaginary}.

For a sketch of the proof of the multidimensional Gaussian behavior of the logarithms of characteristic polynomials, we refer the reader to \cite[Section~8]{DE}, which uses a different method.

 We note that following the arguments above, one can also derive the same structure for the orthogonal family.

\subsection{Notation}

Throughout, the following notations apply. $k$ will denote a positive fixed integer that is the number of shifts $t_j$. The $k$-dimensional real vector $\vec{a}$ will be fixed and consist of coefficients of the linear combination. $n$ will be used for the degree of a monic and square-free polynomial $D$ that belongs to $\mathcal{H}_n$. $m$ will denote the conductor of $E\otimes \chi_D$. $q$ will denote a power of an odd prime and it will be fixed. We will use the commonplace notation $s=\sigma+it$ for a complex variable $s$, and sometimes will use $\sigma_0$ to denote a real number. The ordinate $t$ unfortunately clashes with the $t$ in $\mathbb{F}_q[t]$, and both notations are standard. Note that $\mathcal{F}$ will denote a quadratic character $\chi_D$ or the twisted elliptic curve $E\otimes \chi_D$ according to the context.


\subsection{Plan of the Paper}

In the next section, we will provide more detailed information about the $L$-functions in our families. Section \ref{technical lemmas} contains several technical lemmas for both families, which will be used later in the proofs of the main theorems. In Section \ref{discrepancy}, we investigate the difference between the values at critical line and near the critical line of the real (and imaginary) part of $\log L(s, \mathcal{F})$.

In Section \ref{moments of lin comb}, we estimate the moments of linear combinations of Dirichlet polynomials associated with both families of $L$-functions over $\mathbb{F}_q[t]$. We prove Theorems \ref{Unconditional CLT for real near to half line:re}--\ref{main theorem conditional microscopic}, \ref{main theorem conditional mesoscopic and macroscopic}, and \ref{Uncon mes or micro} in Section \ref{proof of main theorems}. Finally, in the last section, we establish unconditional upper and lower bounds for the distribution of the logarithms of such $L$-functions, and prove Corollary \ref{mesoscopic fluctuation}.


\section{Background}\label{prelimsec}

\subsection{Set up} Let $q=p^e$ be fixed for an odd prime number $p$ and a positive integer $e$. The corresponding hyperelliptic ensemble $ \CMcal{H}_{n, q} $ or  $ \CMcal{H}_{n} $ is defined as
\[
\CMcal{H}_n = \{D \in \mathbb{F}_q[t]: D \text{ is monic and square-free with degree }n\}.
\]
It is not hard to show that $|\CMcal{H}_n| = q^{n-1}(q-1)$. 

Note that we sometimes will refer to the set $\CMcal{H}= \bigcup_{n\geq 1} \CMcal{H}_n $.

For each polynomial $D \in \CMcal{H}_n$, there is a nonsingular curve $y^2 = D(t)$ of genus $g=\frac{n-1-\eta}{2}$, where $\eta$ is defined based on the parity of $n$ as in \eqref{eq:eta}.

We also define the following sets of polynomials.
\[
\CMcal{M}_n= \{f \in \mathbb{F}_q[t]: f \text{ is monic and has degree }n\}
\]
and 
\[
\CMcal{P}_n = \{P \in \mathbb{F}_q[t]: P \text{ is monic, irreducible and has degree }n\}.
\]
We will use the notations
\[
\CMcal{M}_{\leq n} = \bigcup_{j\leq n} \CMcal{M}_j , \quad \CMcal{M}= \bigcup_{j\geq 1} \CMcal{M}_j,
\quad \text{and} \quad
\CMcal{P}_{\leq n} = \bigcup_{j\leq n} \CMcal{P}_j , \quad \CMcal{P}= \bigcup_{j\geq 1} \CMcal{P}_j.  
\]
For convenience, for each $f \in \CMcal{M}$, we set the notations
\[
d(f) := \deg{f} \quad \text{and} \quad |f|:= q^{d(f)}.
\]
In this setting, the von Mangoldt function is defined by
\[
\Lambda(f) :=
\begin{cases}
d(P) \quad &\text{if } \, f=P^k \text{ for some } P\in \CMcal{P},\\
0 \quad &\text{else}.
\end{cases}
\]
One easily sees that $|\CMcal{M}_n|=q^n$. By the prime polynomial theorem~\cite[Theorem 2.2]{Rosen},
\[
|\CMcal{P}_n| 
= \frac{q^n}{n}+ O \Big( \frac{q^{n/2}}{n}\Big).
\]
The zeta function of the ring of polynomials $\mathbb{F}_q[t]$ is denoted by $\zeta_{\mathbb{F}_q[t]}$ and
\[
\zeta_{\mathbb{F}_q[t]}(s)= \sum_{f\in \CMcal{M}} \frac{1}{|f|^s}
=\prod_{P\in \CMcal{P}} \Big( 1-\frac{1}{|P|^s} \Big)^{-1}  \quad \text{for }\, \operatorname{Re}(s)>1.
\]

For each $D\in \mathbb{F}_q[t]$, one can define a quadratic character on $ \mathbb{F}_q[t]$ as follows. For each $P\in \CMcal{P}$, there is a quadratic character $\left( \frac{f}{P}\right)$ given by
\[
\Big( \frac{f}{P}\Big)
=
\begin{cases}
1  \,\, &\text{if } f \text{ is a square modulo } \, P \,\, \text{and } \, P \nmid f,\\
-1  \,\, &\text{if }  f \text{ is not a square modulo } \, P \,\, \text{and } \, P \nmid f,\\
0 \,\, &\text{if } \, P \mid f.
\end{cases}
\]
This can be extended multiplicatively to define $\big( \tfrac{D}{f}\big)$ for each $D\in  \CMcal{H}$. We set
\[
\chi_D(f):= \Big( \frac{D}{f}\Big).
\]


\subsection{$L$-functions over Hyperelliptic Curves} 
\label{sec:L fnc s}

Throughout for simplicity, we assume $q \equiv 1 \pmod{4}$ in order to make use of quadratic reciprocity.
 Given a character $\chi_D$ for $D\in \CMcal{H}_n $ as in the above, the corresponding $L$-function is defined as
\begin{equation}\label{eq:Euler product symplectic}
	L(s, \chi_D) = 
	\sum_{f\in \CMcal{M}} \frac{\chi_D(f)}{|f|^s}
	=\prod_{P\in \CMcal{P}} \Big( 1-\frac{\chi_D(P)}{|P|^s} \Big)^{-1}  \quad \text{for } \,  \operatorname{Re}(s)>1.
\end{equation}
By introducing the variable $u=q^{-s}$, we instead write
\[
\mathcal{L}(u, \chi_D)
= \sum_{f \in \CMcal{M}} \chi_D(f) u^{d(f)}
= \prod_{P \in \CMcal{P}} \big(1-\chi_D(P) u^{d(P)}  \big)^{-1} \quad  \text{for }\, |u| < \frac{1}{q}. 
\]
Let $D$ be a monic square-free polynomial. Then $\mathcal{L}(u, \chi_D)$ has a zero at $u=1$, which is called the trivial zero, if and only if $d(D)$ is even. With the notation in \eqref{eq:eta}, 
we can write
\begin{equation}\label{eq: L and L ast}
	L(s, \chi_D)
	= \mathcal{L}(u, \chi_D)
	= (1-u)^\eta  \mathcal{L}^*(u, \chi_D)
	= (1-q^{-s})^\eta L^*(s, \chi_D).
\end{equation}
Here $\mathcal{L}^*(u, \chi_D)$ is a polynomial of degree $2g=d(D)-1-\eta$ (see~\cite[Proposition 4.3]{Rosen}), which we write as
\begin{equation}\label{eq:L as polyl}
	\mathcal{L}^*(u, \chi_D)
	= \prod_{j=1}^{2g} \big( 1-u\sqrt{q} \alpha_j \big).
\end{equation}
By the Riemann hypothesis, proven by Weil~\cite{Weil}, all zeros of $\mathcal{L}^*$ are on the circle $|u|=1/\sqrt{q}$.

We may therefore set, for $e(x) = \exp(2\pi i x)$, 
\[
\alpha_j = e(-\theta_{j, D}) \quad \text{for each } \, j=1, 2, \dots, 2g .
\]
The quantities $\theta_{j, D}$ are referred to as the eigenphases of $\mathcal{L}^*$.


\subsection{$L$-functions of Quadratic Twists of Elliptic Curves}\label{sec:L fnc o}

Throughout our discussion exclusively for this family, we consider $q$ is a prime power with $q \equiv 1 \pmod 4$ and $(q, 6) = 1$. We refer the reader to \cite{BFKRG, BH, CL} for further details in this direction.

Given an elliptic curve $E: y^2=x^3+Ax+B$ and a polynomial $D\in \mathbb{F}_q[t]$, the elliptic curve twisted by $D$ is defined as 
\[
E_D: y^2=x^3+AD^2 x+BD^3.
\]
As $D$ varies, these equations give distinct elliptic curves if and only if $D$'s are square-free and $(D, \Delta)=1$, where $\Delta=4A^3+27B^2$ is discriminant of the elliptic curve $E$. We thus define the set
\[
\mathcal{H}^{\Delta}_{n} := \left\{ D \in \mathbb{F}_q[t] :  
\begin{array}{l}
D \text{ is monic and square-free,} \deg D = n, 
(D, \Delta) =1 
\end{array} 
\right\}.
\]

The conductor of the curve $E$ is defined by
\[
N_E=\prod_{P \text{ irreducibles }}P^{f_P(E)},
\]
where 
\[
f_P(E)=\begin{cases}
0 & \mbox{if } \, E \text{ has good reduction at } P,\\
1 & \mbox{if } \, E \text{ has multiplicative reduction at } P,\\
2 & \mbox{if } \, E \text{ has additive reduction at } P,
\end{cases}
\]
including the prime at infinity. Write $N_E=M_E A_E^2$, where $M_E$ is the product of primes of multiplicative bad reduction and $A_E$ is the product of primes of additive bad reduction.
 
The {\it normalized} $L$-function associated to the elliptic curve $E/K$ has the following 
Dirichlet series.
\[
L(s, E) := \mathcal{L}(u, E) = \sum_{f \in \mathcal{M}} \lambda(f) u^{d(f)}
\quad \text{for } \, \operatorname{Re}(s) > 1, \, \text{ with } \, u = q^{-s}. 
\]
This $L$-function is a polynomial in $u$ with integer coefficients 
of degree
\[
 \deg \big( \mathcal{L}(u, E) \big) 
= d(M_E) + 2 d(A_E) - 4.
\]

Given an elliptic curve $E$, a twisted elliptic curve $L$-function is defined via the Euler product 
\begin{equation}\label{eq:L defn elliptic}
L(s, E\otimes \chi_D)=\prod_{P}\left(1-\frac{\alpha(P)\chi_D(P)}{|P|^s}\right)^{-1}\left(1-\frac{\beta(P)\chi_D(P)}{|P|^s}\right)^{-1}  \quad \text{for }\, \operatorname{Re}(s)>1,
\end{equation}
where $\alpha(P)+\beta(P)=\lambda(P)$ with a complex number $\alpha(P)$ and $\beta(P)$ of magnitude $1$, and
\[
\alpha(P)\beta(P)=\begin{cases}
1 & \mbox{if } \, P\nmid N_E,\\
0 & \mbox{if } \, P\mid N_E.
\end{cases}
\]
Note that since $E$ has trivial nebentypus then $\alpha(P)$ and $\beta(P)$ comes either as purely complex conjugate pairs or both real numbers, so that $\lambda\in \mathbb{R}$.   
An equivalent form is
\begin{align*}
\mathcal{L}(u, E\otimes \chi_D)
&=\prod_{P\nmid N_E}
\left(1-\alpha(P)\chi_D(P)u^{d(P)}\right)^{-1}\left(1-\beta(P)\chi_D(P)u^{d(P)}\right)^{-1}
 \prod_{P\mid N_E} \left(1-\lambda(P)\chi_D(P)u^{d(P)}\right)^{-1}.
\end{align*}
This is a polynomial of degree 
\begin{align}\label{m and n connection for elliptic curve}
m=2d(D)+d(N_E)-4,
\end{align}
and satisfies the functional equation
\begin{equation} \label{eq:functional_eq}
\mathcal{L}(u, E \otimes \chi_D) 
= \epsilon\, (\sqrt{q} u)^{m}\, \mathcal{L}\Big(\frac{1}{qu}, E \otimes \chi_D\Big),
\end{equation}
where
\[
\epsilon = \epsilon(E_D) = \epsilon_{d(D)}\, \epsilon(E)\, \chi_D(M_E).
\]
Here $\epsilon_{d(D)} \in \{\pm 1\}$ is a sign depending on the parity of $d(D)$ and $\epsilon(E)$ is the root number of the elliptic curve $E$. See \cite[Lemma 2.3 and Proposition 4.3]{BH} for more details.

Now, taking logarithms of both sides of \eqref{eq:L defn elliptic} gives
\[
\log L\big(s, E\otimes \chi_D\big)
= -\sum_{P}\log \left(1-\frac{\alpha(P)\chi_D(P)}{|P|^s}\right)-\sum_{P}\log \left(1-\frac{\beta(P)\chi_D(P)}{|P|^s}\right).
\]
Differentiating both sides of the above, we obtain
\[
\frac{L'}{L}\big(s, E\otimes\chi_D\big)=-\log q\sum_{n\geq 0}\frac{\lambda_D(n)}{q^{ns}},
\]
where
\[
\lambda_D(n)=\sum_{j\mid n}\sum_{j d(P)=n}d(P)\left(\alpha(P)^j+\beta(P)^j\right)\chi_D(P)^j.
\]
It is further known that, see~\cite[Proposition 4.1]{CL},
\[
L\big(s, E\otimes \chi_D\big)
=\prod_{j=1}^m \big(1-\alpha_j q^{\frac12-s}\big),
\]
for some $\alpha_j$ of unit modulus. From this formula, the logarithmic derivative can be written as
\[
\frac{L'}{L}
\big( s, E\otimes \chi_D\big)
=  \bigg(
-\frac{m}{2}+
\sum_{j=1}^m \Big(\frac{1}{1-\alpha_j q^{\frac12-s}}-\frac12
\Big)\bigg) 
\log{q}.
\]
From Rankin-Selberg convolution and Fourier coefficients of modular forms, using information at irreducibles, for, say $P\in \mathcal{P}$,  $\lambda(P)^2=\lambda(P^2)+1$, one can find that  
\[
\mathcal{L}\big(u, E\otimes E\big)
=\sum_{f\in \mathcal{M}}\lambda(f)^2u^{d(f)}=\mathcal{L}(u, \text{sym}^2(E))\zeta_{\mathbb{F}_q[t]}(u),
\]
where $u=q^{-s}$ and $\mathcal{L}\big(u, \text{sym}^2(E)\big)=\sum_{f\in \mathcal{M}}\lambda(f^2)u^{d(f)}$ is the symmetric square $L$-function over function fields. This leads to the estimate
\[
\sum_{f\in \mathcal{M}_n}\lambda(f)^2=q^n+O(q^{n/2}).
\]
Therefore, by using the partial summation formula, we obtain
\begin{align}\label{lambda on primes}
\sum_{d(P)\leq X}\frac{\lambda(P)^2}{|P|}=\log X+O(1).
\end{align}

For any $D\in  \mathcal{H}^{\Delta}_{n}$, we see that
\[
\epsilon(E_D)=\epsilon_{n}\epsilon(E) \chi_D(M_E),
\]
 
By \eqref{eq:functional_eq}, $\epsilon(E_D)=-1$ implies that $L\big( \frac12, E\otimes \chi_D\big)=0$. In order to study the logarithmic behavior of these $L$-functions near the central point, we focus on the slightly modified version of $\mathcal{H}^{\Delta}_{n}$ below that consists of polynomials with root number $1$.
\begin{align}\label{positive family}
\mathcal{H}^{\Delta, +}_{n}:
=\{D\in  \mathcal{H}^{\Delta}_{n}: \chi_D(M_E)=\epsilon_{n}\epsilon(E)\}.
\end{align}
From \cite[Lemma 3.4]{MS}, we know that
\[
| \mathcal{H}^{\Delta, +}_{n}|=
\begin{cases}
\frac{1}{2}| \mathcal{H}^{\Delta}_{n}|+O_{\Delta}(q^{n/2}) & \mbox{if } \, M_E\neq 1,\\
| \mathcal{H}^{\Delta}_{n}| & \mbox{if } \, M_E=1 .
\end{cases}
\]
However, directly working with $\mathcal{H}^{\Delta, +}_{n}$ is challenging. To overcome this, we decompose the family as
\begin{align}
&\mathcal{H}^{\Delta}_{n}
=\bigcup_{(C, N_E)=1} \mathcal{H}^{\Delta}_{n}(C), \notag
\\
&\mathcal{H}^{\Delta, +}_{n}=\bigcup_{\substack{(C, N_E)
		=1
		\\ 
		\chi_C(M_E)
		=\epsilon_{n}\epsilon(E)}} \mathcal{H}^{\Delta}_{n}(C), \label{relation bet. full fam and subfam}
\end{align}
where 
\[
\mathcal{H}^{\Delta}_{n}(C)=\{D\in  \mathcal{H}^{\Delta}_{n}: D\equiv C \pmod{N_E}\}.
\]
This decomposition of polynomials into arithmetic progressions allows us to apply orthogonality (see Lemmas \ref{sum over squares} and \ref{sum over non-squares}).

Since $\chi_{M_E}$ is a Dirichlet character of conductor  $M_{E}$, it is constant on each $ \mathcal{H}^{\Delta}_{n}(C)$ when $M_E\mid N_E$. Thus, it suffices to focus on these individual subsets of certain arithmetic progressions.

From~\cite[Lemma 6.2]{CL}, for $(C, N_E)=1$ and any $\epsilon>0$, one has
\begin{align}\label{cardinality in AP}
| \mathcal{H}^{\Delta}_{n}(C)|=\frac{| \mathcal{H}_{n}|}{\big|\left(\mathbb{F}_q[t]/N_E\right)^*\big|}\prod_{P\mid N_E}\frac{|P|}{|P|+1}+O_{E, q, \epsilon}\left(q^{(\frac14+\epsilon)n}\right),
\end{align}
where $\left(\mathbb{F}_q[t]/N_E\right)^*$ is the multiplicative subgroup of $\mathbb{F}_q[t]$ modulo $N_E$.


\section{Some Technical Lemmas}\label{technical lemmas}
In this section, we present some lemmas which will later be crucial to the proofs of the main theorems.
We first note the following lemma.


\begin{lem}\label{upper bound for truncated sum over primes }
	Let $K\geq 2$ and $\sigma_0 >\tfrac12$ be such that $K\big(\sigma_0-\tfrac12\big)<\frac{1}{2\log q}$. Then we have
	\[
	\sum_{d(P)\leq K}\frac{1}{|P|^{2\sigma_0}}=\log K +O(1).
	\]
\end{lem}
\begin{proof}
	See \cite[Lemma 3.4]{DL}.
\end{proof}


\subsection{Orthogonality of characters over the families} 

Our computations for the symplectic family will need the following two lemmas. 


\begin{lem}\label{square term evaluation}
	Let $f\in \CMcal{M}_n$. Then 
	\[
	\frac{1}{|\CMcal{H}_{n}|}\sum_{D\in \CMcal{H}_{n}}\chi_D(f^2)
	=\prod_{P|f}\left(1+\frac{1}{|P|}\right)^{-1}+O(q^{-n}).
	\]
\end{lem} 

\begin{proof}
For $n=2g+1$, this is Lemma 3.7 of \cite{BF}. For even $n$, one can write down a similar proof. 
\end{proof}


\begin{lem}[P\'{o}lya--Vinogradov inequality]\label{polya-vinogradov inequality}
	For $\ell\in \CMcal{M}_n$ not a square polynomial, let $\ell=\ell_1 \ell_2^2$ with $\ell_1$ square-free. Then we get
	\[
\sum_{D\in \CMcal{H}_{n}}\chi_D(\ell)\ll_{\epsilon} q^{(\frac12+\epsilon)n}|\ell_1|^{\epsilon}.
	\]
\end{lem}


\begin{proof}
	The proof follows from that of Lemma 3.5 of \cite{BF2018 Hybrid}. 
\end{proof}

For the orthogonal family, we focus on character sums over arithmetic progressions. In particular, for any $(C, N_E)=1$, our goal to evaluate  
$
\sum_{D\in  \mathcal{H}^{\Delta}_{n}(C)}\chi_D(f)
$
depending on whether $f$ is square-free or not.


\begin{lem}\label{sum over squares}
Let $\ell\in \mathbb{F}_q[t]$. For polynomials $C$ with $(C, N_E)=1$ and any $\epsilon>0$, 
	\[
	\sum_{D\in  \CMcal{H}^{\Delta}_{n}(C)}\chi_D(\ell^2)
	=\frac{| \CMcal{H}_{n}|}{\big|\left(\mathbb{F}_q[t]/N_E\right)^*\big|}\prod_{P\mid \ell N_E}\left(1+\frac{1}{|P|}\right)^{-1}
	+O\big(q^{(\frac14+\epsilon)n}\big).
	\]
\end{lem}


\begin{proof}
We will prove the lemma for prime powers $\ell = P^m$ for some $m\geq 0$. It can then be seen to also hold for other $\ell$. Write $n=k d(P)+r$, where $0\leq r< d(P)$. If $(P, N_E)=1$, then we have
	\[
	\sum_{D\in  \CMcal{H}^{\Delta}_{n}(C)}\chi_D(\ell^2)=\sum_{\substack{D\in  \CMcal{H}^{\Delta}_{n}(C),\\ (D, P)=1}}1=| \CMcal{H}^{\Delta}_{n}(C)|-\sum_{\substack{D\in  \CMcal{H}^{\Delta}_{n-d(P)}(CP^{-1}),\\ (D, P)=1}}1,
	\]
where $P^{-1}$ is the inverse of $P$ modulo $N_E$. Repeating this process and using \eqref{cardinality in AP} many times, 
	\begin{align*}
	\sum_{D\in  \CMcal{H}^{\Delta}_{n}(C)}\chi_D(\ell^2)
	&=| \CMcal{H}^{\Delta}_{n}(C)|-| \CMcal{H}^{\Delta}_{n-d(P)}(CP^{-1})|+\cdots+(-1)^k | \CMcal{H}^{\Delta}_{n-kd(p)}(CP^{-k})|\\
	&=| \CMcal{H}^{\Delta}_{n}(C)|-\frac{| \CMcal{H}^{\Delta}_{n}(C)|}{|P|}+\cdots +(-1)^k \frac{| \CMcal{H}^{\Delta}_{n}(C)|}{|P|^k}+O\left(q^{(\frac14+\epsilon)n}\sum_{k\geq 1}q^{-kd(P)}\right)\\
	&= | \CMcal{H}^{\Delta}_{n}(C)|\sum_{j=0}^{k}(-1)^j\frac{1}{|P|^j}+O\left(q^{(\frac14+\epsilon)n}\right)=| \CMcal{H}^{\Delta}_{n}(C)|(1+|P|^{-1})^{-1}+O\left(q^{(\frac14+\epsilon)n}\right) .
	\end{align*}
Here, the last line follows from the choice of $k$. In the case where $P\mid N_E$, notice that 
	\[
	\sum_{D\in  \CMcal{H}^{\Delta}_{n}(C)}\chi_D(\ell^2)=| \CMcal{H}^{\Delta}_{n}(C)|,
	\]
	which completes the proof.
\end{proof}


\begin{lem}\label{sum over non-squares}
Assume that $\ell_1$ is square-free and $\ell_2 \in \mathbb{F}_q[t]$. Then for polynomials $C$ with $(C, N_E)=1$ and any $\epsilon>0$,
	\[
	\sum_{D\in  \CMcal{H}^{\Delta}_{n}(C)}\chi_D(\ell_1 \ell_2^2)
	\ll_{E, q, \epsilon}|\ell_1|^{\epsilon} q^{(1/2+\epsilon)n}.
	\]
\end{lem}


\begin{proof}
Note that if the power series $\sum_{f\in \CMcal{M}_n} a(f)u^{d(f)}$ converges absolutely for $|u| \leq  R < 1$, then
	\[
	\sum_{f\in \CMcal{M}_n}a(f)
	=\frac{1}{2\pi i}\oint_{|u|=R} \left(\sum_{f\in \CMcal{M}_n} a(f)u^{d(f)}\right)\frac{du}{u^{n+1}}.
	\]
Using the orthogonality of characters, we see that
	\begin{align}\label{use of character orthogonality}
	\sum_{D\in  \CMcal{H}^{\Delta}_{n}(C)}\chi_D(\ell_1\ell_2^2)
	=\frac{1}{\big|(\mathbb{F}_q[t]/N_E)^*\big|}\sum_{\psi \text{ (mod }N_E)}\overline{\psi(C)}\sum_{D\in  \CMcal{H}_{n}}\chi_D(\ell_1\ell_2^2)\psi(D).
	\end{align}
Since $q\equiv 1 \pmod{4}$, by the quadratic reciprocity
	\begin{equation}\label{eq:quad rec}
	\sum_{D\in \CMcal{H}_{n}} \chi_D(\ell_1\ell_2^2) \psi(D)
	=\frac{1}{2\pi i}\oint_{|u|=r}A_{\psi}(u)\frac{du}{u^{n+1}},
	\end{equation}
where $r<1$ and the generating series can be written as an Euler product as follows. For $F= \ell_1\ell_2^2$, 
	\begin{align*}
	A_{\psi}(u)
	&=\prod_P\left(1+\psi(P)\chi_F(P)u^{d(P)}\right)
	=\frac{\prod_{P\nmid F}\left(1-\psi^2 (P)u^{2d(P)}\right)}{\prod_{P}\left(1-\psi(P)\chi_F(P)u^{d(P)}\right)}\\
	&=\frac{\prod_{P}\big(1-\widetilde{\psi} (P)\chi_{\ell_1}(P)u^{d(P)}\big)^{-1} 
	\prod_{P\mid Q_{\psi}}\big(1-\widetilde{\psi}(P)\chi_{\ell_1}(P)u^{d(P)}\big)}{\prod_{P}\big(1-\widetilde{\psi}^2(P)u^{2d(P)}\big)^{-1} \prod_{P\mid F R_{\psi}}\big(1-\widetilde{\psi}^2(P)u^{2d(P)}\big)}.
	\end{align*}
Here, we used the unique decomposition $\psi=\widetilde{\psi}\psi_0$, $\psi^2=\widetilde{\psi}^2 \psi^2_0$, where $\widetilde{\psi}, \widetilde{\psi}^2$ are primitive characters and $\psi_0, \psi_0^2$ are principal characters with minimal moduli $Q_{\psi}$ and $R_{\psi}$, respectively. 
	
\noindent If $\widetilde{\psi}^2$ is trivial, then 
	\[
	\frac{1}{\prod_{P}\big(1-\widetilde{\psi}^2(P)u^{2d(P)}\big)^{-1}}=1-qu^2.
	\]
If $\widetilde{\psi}^2$ is nontrivial, then 
	\[
	\frac{1}{\prod_{P}\big(1-\widetilde{\psi}^2(P)u^{2d(P)}\big)^{-1}}
	=(1-u^2)^{-\eta}\prod_{j=1}^{M}\left(1-\sqrt{q}e^{i\theta_j}u^2\right)^{-1},
	\]
where $M\leq d(N_E)-1$, and therefore it has poles on the disc $|u|=q^{-\frac14}$. 
	
Notice that $(\ell_1, N_E)=1$. So $\widetilde{\psi} \chi_{\ell_1}$ is also a primitive and non-principal character modulo $N_E \ell_1$, since $\chi_{\ell_1}$ is always a non-principal character. We use the Lindel\"{o}f bound (see \cite[Theorem 3.4]{AT}) on $|u|=q^{-1/2}$ to get 
	\[
	\prod_{P}\left(1-\widetilde{\psi} (P)\chi_{\ell_1}(P)u^{d(P)}\right)^{-1}\ll |N_E \ell_1|^{\epsilon}.
	\]
These computations allow us to bound
	\[
	A_{\psi}(u)
	\ll_{\ell_1,\ell_2, E, q, \epsilon} |\ell_1|^{\epsilon}.
	\]
Using the above estimates, we shift the contour of integration to $|u|=q^{-\frac12-\epsilon}$ to obtain
	\[
	\left|\frac{1}{2\pi i}\oint_{|u|
	=q^{-1/2-\epsilon}}A_{\psi}(u)\frac{du}{u^{n+1}}\right|
	\ll_{\ell_1,\ell_2, E, q, \epsilon} |\ell_1|^{\epsilon}q^{(\frac12+\epsilon)n},
	\]
since on $|u|=q^{-\frac12-\epsilon}$, we have $u^{-n-1}\ll q^{(\frac12+\epsilon)n}$.
We substitute this bound in \eqref{eq:quad rec}. The result then follows by \eqref{use of character orthogonality}.
\end{proof}


\subsection{Sums Involving Trigonometric Functions}

The following mean values will appear in our calculation of moments of real parts of linear combinations of Dirichlet polynomials. 


\begin{lem}\label{logarithmic cosine sum}
	We have
	\begin{align*}
	\sum_{n\leq X}\frac{\cos (2nt)}{n}
	=\log \Big(\min\big\{X, \tfrac{1}{2|t|}\big\}\Big)+O(1),
	\quad \text{ and } \quad \sum_{n\leq X}\frac{\sin (2nt)}{n}
	\ll 1. 
	\end{align*}
\end{lem}


\begin{proof}
	The proof is analogous to that of Lemma 9.1 in~\cite{Florea 4th moment}. First, if $|t| \leq \frac{1}{2X}$, then since $\cos u=1+O(u^2)$ we have 
	\[
	\sum_{n\leq X}\frac{\cos (2nt)}{n}
	= \sum_{n\leq X} \frac{1}{n}+ O\bigg(t^2 \sum_{n\leq X} n\bigg).
	\]
	This is $\log X + O(1)$. 
	
	Now suppose that $|t| > \frac{1}{2X}$. Then we separate the sum as 
	\begin{equation}\label{eq:separate n}
	\sum_{n\leq X}\frac{\cos (2nt)}{n}
	= \sum_{n\leq \frac{1}{2|t|}}\frac{\cos (2nt)}{n}
	+
	\sum_{ \frac{1}{2|t|} < n \leq X}\frac{\cos (2nt)}{n}.
	\end{equation}
	By the first case, 
	\[
	\sum_{n\leq \frac{1}{2|t|}}\frac{\cos (2nt)}{n}
	= \log \Big(\frac{1}{2|t|}\Big)+O(1). 
	\]
	We compare the last sum in \eqref{eq:separate n} with the integral 
	\[
	\int_{\frac{1}{2|t|}}^X  \frac{\cos (2ut)}{u} \mathop{du},
	\]
	which is $O(1)$ by integration by parts. This proves our first claim. 
	
	The second claim follows similarly, the difference being the use of the bound $\sin(2nt)\leq 2n|t|$ for $|t|\leq \frac{1}{2X}$ and $n\leq X$. 
\end{proof}


\subsection{Dirichlet Polynomial Approximation for $\log{L(s, \mathcal{F})}$}\label{logformula}

In the rest of the paper, we let $\cF$ stand for either $\chi_D$ or $E\otimes \chi_D$, depending on whether we deal with the symplectic or the orthogonal family. Similarly, $\kappa:=\kappa_{\mathcal{F}}$ will be either $2g$ or $m$ depending on the family.


\begin{lem}\label{lemma Altug Tsim}
For $\sigma\ge \sigma_0>\frac12$,
	\[
	\bigg|\frac{{\alpha_j}^{-1}q^{\sigma-\frac12}}{(1-{\alpha_j}^{-1}q^{\sigma-\frac12})^2}\bigg|
	\leq \frac{1}{(\sigma_0-\frac12)\log q} \operatorname{Re}\Big(\frac{1}{1-\alpha_j q^{\frac12-\sigma_0}}\Big).
	\]
\end{lem}


\begin{proof}
	This is a generalization of Lemma 3.2 in \cite{AT} to prime powers $q$. 
\end{proof}

Note that by using \eqref{eq: L and L ast} and \eqref{eq:L as polyl}, we can write
\begin{align}\label{eq:log der symp zeros}
\frac{L'}{L}(s, \mathcal{F})
= \log q \bigg(\eta_\cF\frac{q^{-s}}{1-q^{-s}}-\kappa_\cF 
+ \sum_{j=1}^{\kappa_\cF} \frac{1}{1-\alpha_j q^{1/2-s}}\bigg),
\end{align}
where
\begin{equation}\label{eq:defn eta}
\eta_{\mathcal{F}}=
\begin{cases}
\eta \quad &\text{if } \, \, \cF=\chi_D ,\\
0 \quad &\text{if } \cF=E\otimes \chi_D.
\end{cases}
\end{equation}
This holds for all $s$ except the zeros of the corresponding $L$-function.

\begin{lem}\label{bnd-sum-zeros}
For $\sigma\ge \sigma_0 > \frac12$,
	\[
	\sum_{j=1}^{\kappa_\CMcal{F}}\bigg|\frac{(\alpha_jq^{\frac12-\sigma-it})^X(1-(\alpha_jq^{\frac12-\sigma-it})^X)^2}{(1-\alpha_j^{-1}q^{\sigma-\frac12+it})^3}\bigg|
	\le 
	\frac{4q^{X(\frac12-\sigma)}}{(\sigma_0-\frac12)^2\log^3 q}\bigg(\bigg|\operatorname{Re}\frac{L'}{L}\big(\sigma_0+it,\mathcal{F}\big)\bigg|+\frac{3\kappa_{\mathcal{F}} \log q}{2}\bigg). 
	\]
\end{lem}


\begin{proof}
This proof is similar to that of Lemma 3.9 in \cite{DL}. We start with
	\begin{align*}
	\bigg|\frac{(\alpha_jq^{\frac12-\sigma-it})^X(1-(\alpha_jq^{\frac12-\sigma-it})^X)^2}{(1-\alpha_j^{-1}q^{\sigma-\frac12+it})^3}\bigg| 
	=
	\bigg|\frac{(\alpha_jq^{\frac12-\sigma-it})^{X+1}(1-(\alpha_jq^{\frac12-\sigma-it})^X)^2}
	{(1-\alpha_j^{-1}q^{\sigma-\frac12+it})}\bigg| \bigg|\frac{{\alpha_j}^{-1} q^{\sigma-\frac12+it}}{(1-\alpha_j^{-1}q^{\sigma-\frac12+it})^2}\bigg|.
	\end{align*}
Here, $\big|(\alpha_jq^{\frac12-\sigma-it})^{X+1}(1-(\alpha_jq^{\frac12-\sigma-it})^X)^2 \big| \leq 4 q^{X(\frac12-\sigma)}$.  Also, we write 
	\begin{align*}
	|1-\alpha_j^{-1}q^{\sigma-\frac12+it}|
	&= |1-q^{\sigma-\frac12}\cos(2\pi\theta_j+t\log q)
	-iq^{\sigma-\frac12}\sin(2\pi\theta_j+t\log q)| \\
	&= \sqrt{1+q^{2\sigma-1}-2q^{\sigma-\frac12}\cos(2\pi\theta_j+t\log q)}.
	\end{align*}
For $\sigma\geq \sigma_0$,
	\[
	\frac{1}{|1-\alpha_j^{-1}q^{\sigma-\frac12+it}|}
	\leq 
	\frac{1}{\sqrt{1+q^{2\sigma_0-1}-2q^{\sigma_0-\frac12}}}
	\leq \frac{1}{(\sigma_0-\frac12)\log q}. 
	\]
Now, let $\widetilde{\alpha}_j={\alpha}_j q^{-it}$. Then 
	\[
	\left|\frac{{\alpha_j}^{-1} q^{\sigma-\frac12+it}}{(1-\alpha_j^{-1}q^{\sigma-\frac12+it})^2}\right|
	=  \left|\frac{{\widetilde{\alpha}_j}^{-1} q^{\sigma-\frac12}}{(1-\widetilde{\alpha}_j^{-1} q^{\sigma-\frac12})^2}\right| .
	\]
By Lemma \ref{lemma Altug Tsim}, this is 
	\[
	\leq \frac{1}{(\sigma_0-\frac12)\log q} 
	\operatorname{Re} \Big(\frac{1}{1-\widetilde{\alpha}_j q^{\frac12-\sigma_0}} \Big).
	\]
Therefore,
	\[
	\sum_{j=1}^{\kappa_{\mathcal{F}}}\bigg|\frac{(\alpha_jq^{\frac12-\sigma-it})^X(1-(\alpha_jq^{\frac12-\sigma-it})^X)^2}{(1-\alpha_j^{-1}q^{\sigma-\frac12+it})^3}\bigg|
	\le 
	\frac{4 q^{X(\frac12-\sigma)}}{(\sigma_0-\frac12)^2\log^2 q} 
	\sum_{j=1}^{\kappa_{\mathcal{F}}} \operatorname{Re} \Big(\frac{1}{1-\widetilde{\alpha}_j q^{\frac12-\sigma_0}} \Big).
	\]
By \eqref{eq:log der symp zeros} with $s=\sigma_0+it$, 
	\[
	\sum_{j=1}^{\kappa_{\mathcal{F}}} \operatorname{Re} \Big(\frac{1}{1-\widetilde{\alpha}_j q^{\frac12-\sigma_0}} \Big)
	= \frac{1}{\log q} \operatorname{Re} \frac{L'}{L}\big(\sigma_0+it, \mathcal{F}\big)
	-\operatorname{Re} \frac{\lambda q^{-\sigma_0-it}}{1-q^{-\sigma_0-it}}+\kappa_\cF.
	\]
Here, 
	\begin{align*}
	\operatorname{Re} \frac{ q^{-\sigma_0-it}}{1-q^{-\sigma_0-it}}
		= \operatorname{Re}\bigg( \frac{ q^{-\sigma_0-it}({1-q^{-\sigma_0+it}})}
	{|1-q^{-\sigma_0-it}|^2}\bigg)
		=\frac{q^{-\sigma_0}(\cos(t\log q)-q^{-\sigma_0})}{|1-q^{-\sigma_0-it}|^2}.
	\end{align*}
This is, by the triangle inequality,
	\[
	\leq 
	\frac{q^{-\sigma_0}}{1-q^{-\sigma_0}}
	\ll q^{-\sigma_0} \leq \frac{\kappa_\cF}{2}.
	\]
Combining these estimates, we conclude that  
	\begin{align*}
\sum_{j=1}^{2g}
	\Big|\frac{(\alpha_jq^{\frac12-\sigma-it})^X(1-(\alpha_jq^{\frac12-\sigma-it})^X)^2}{(1-\alpha_j^{-1}q^{\sigma-\frac12+it})^3}\Big|
	\leq
	\frac{4q^{X(\frac12-\sigma)}}{(\sigma_0-1/2)^2\log^3 q}\Big(\Big|\operatorname{Re} \frac{L'}{L}\big(\sigma_0+it,\mathcal{F}\big)\Big|+\frac{3\kappa_\cF\log q}{2}\Big).
	\end{align*}
This completes the proof. 
\end{proof}

By taking the logarithmic derivative of the Euler product in \eqref{eq:Euler product symplectic}, we obtain
\begin{align}\label{eq:log der symp}
\frac{L'}{L}(s, \cF)
=-\log q \sum_{f\in \CMcal{M}}\frac{\Lambda_{\cF}(f)\chi_D(f)}{|f|^s} \quad \text{for} \,\, \operatorname{Re}(s) >1,
\end{align}
where the generalized von Mangoldt function is defined as
\[
\Lambda_{\cF}(f)=
\begin{cases}
\lambda_{\cF}(f)\log{d(f)} \quad &\text{if } \, \, f = P^d ,\\
0 \quad &\text{else},
\end{cases}
\,\, \text{ for } \,\,
\lambda_{\mathcal{F}}(P^d)=
\begin{cases}
1 \quad &\text{if }  \cF=\chi_D ,\\
 \lambda(P^d)-\lambda(P^{d-2})\quad &\text{if } \cF=E\otimes \chi_D.
\end{cases}
\]
Note that the sum above is over irreducible polynomials whereas the one in \eqref{eq:log der symp zeros} involves a sum over zeros. In order to approximate $ \frac{L'}{L}(s, \cF)$ near the critical line $\operatorname{Re}(s) =\frac12$ by a truncation of this Dirichlet series, we introduce a weighted von Mangoldt function.

\begin{equation}\label{Lambda X def} 
\Lambda_{X,\cF}(f)=
\begin{cases}
2X^2\Lambda_{\cF}(f) & \text{ if } d(f) \le X, \\ 
(X^2-(d(f))^2+2d(f)X-3X-d(f)-2)\Lambda_{\cF}(f) & \text{ if } X< d(f) \le 2X, \\ 
(3X-d(f)+1)(3X-d(f)+2)\Lambda_{\cF}(f) & \text{ if } 2X <d(f) \le 3X.
\end{cases}
\end{equation}
Observe that this weight is applied to polynomials with degree up to $X$.  We will first approximate the logarithm of the $L$-function by the 
Dirichlet polynomial
    \begin{equation}\label{eq:defn D X}
    \mathcal{D}_X(s,\cF)
    =\sum_{f\in \CMcal{M}_{\leq X}}\frac{\Lambda_{\cF}(f)\chi_D(f)}{d(f)|f|^{s}}.
    \end{equation}
We will study the distribution of this polynomial at $\operatorname{Re}(s)=\sigma_0=\tfrac12+\tfrac{c}{X}$ with $0<c<\frac1{2\log q}$.

We now prove that this polynomial can be used to approximate $\log L\big(\sigma_0+it,\cF\big)$. 


\begin{prop}\label{logL-as-PX}
Let $X\ge1$, and $\eta_{\mathcal F}$ be as defined in \eqref{eq:defn eta}. Then for any $t \in \mathbb R$ and $\sigma_0=\tfrac12+\tfrac{c}{X}$, 
	\begin{multline*}
	\log L\big(\sigma_0+it,\cF\big)
	= \mathcal{D}_X\big(\sigma_0+it,\cF\big)
	+O\Big(\frac{\kappa_\cF}{X}+\frac{\eta_{\cF}}{X^3}\Big) \\
	+O\bigg(\frac{1}{X^2}\bigg|\sum_{f\in \CMcal{M}_{\leq 3X}\setminus 
	\CMcal{M}_{\leq X} }\frac{\Lambda_{X, \cF}(f)\chi_D(f)}{d(f) |f|^{\sigma_0+it}}\bigg|\bigg)+	O\bigg(\frac{1}{X^3}\Big|\sum_{f\in\CMcal{M}_{\leq 3X}}\frac{\Lambda_{X, \cF}(f)\chi_D(f)}{|f|^{\sigma_0+it}}\Big|\bigg), 
	\end{multline*}
where $\Lambda_{X, \cF}(f)$ and $ \mathcal{D}_X\big(\sigma_0+it,\cF\big)$ are defined by \eqref{Lambda X def} and \eqref{eq:defn D X} respectively.
\end{prop}


\begin{proof}
We first establish the result for the real part that 
	\begin{equation}\label{eq:logLformula1}
	\begin{split}
	\log \big|L\big(\sigma_0+it,\cF\big)\big|
	=&\, \frac{1}{2X^2}
\operatorname{Re} \sum_{f\in\CMcal{M}_{\leq 3X}}
	\frac{\Lambda_{X,\cF}(f)\chi_D(f)}{d(f)|f|^{\sigma_0+it}} \\
	&+O\bigg(\frac{1}{X^3}\Big|\sum_{f\in\CMcal{M}_{\leq 3X}}\frac{\Lambda_{X,\cF}(f)\chi_D(f)}{|f|^{\sigma_0+it}}\Big|\bigg)
	+O\left(\frac{\kappa_\cF}{X}+ \frac{\eta_{\cF}}{X^3}\right).
	\end{split}
	\end{equation}
By \cite[(4.7)]{DL} for $s= \sigma+it$ with $\operatorname{Re}(s)\geq \frac12$ (the case $\mathcal{F} = E \otimes \chi_D$ follows the same path, but it requires further information about zeros),
	\begin{equation}\label{log derivative version 3}
	\begin{split}
	-\frac{L'}L(s,\mathcal{F})
	=&\, \frac{\log q}{2X^2}\sum_{f \in \CMcal{M}_{\leq 3X}}\frac{\Lambda_{X,\cF}(f)\chi_D(f)}{|f|^{s}}\\ 
	& +\frac{\log^2 q}{2X^2}\bigg(\frac{\eta_\cF \, q^{-Xs} (1-q^{-Xs})^2}{(1-q^{-s})^3}+\sum_{j=1}^{\kappa_\cF}
	\frac{(\alpha_j q^{\frac12-s})^X(1-(\alpha_j q^{\frac12-s})^X)^2}{(1-\alpha_j^{-1}q^{s-\frac12})^3}\bigg).
	\end{split}
	\end{equation}
By Lemma \ref{bnd-sum-zeros}  with $\sigma_0=\tfrac12+\tfrac{c}{X}$,
	\begin{equation}\label{bound on kappa}
	\begin{split}
	\sum_{j=1}^{\kappa_{\mathcal{F}}}  \bigg|\frac{(\alpha_jq^{\frac12-\sigma-it})^X(1-(\alpha_jq^{\frac12-\sigma-it})^X)^2}{(1-\alpha_j^{-1}q^{\sigma-\frac12+it})^3}\bigg| & \\
	\le 
	\frac{4X^2q^{X(\frac12-\sigma)}}{c^2\log^3 q} & \bigg(\bigg|\operatorname{Re}\frac{L'}{L}\big(\sigma_0+it,\cF\big)\bigg|+\frac{3\kappa_\cF \log q}{2} \bigg).
	\end{split}
	\end{equation}
Then by \eqref{log derivative version 3}, for $v$ such that $|v|\leq 1$ we have
	\begin{equation}\label{eq:interm 1}
	\begin{split}
	-\operatorname{Re} \frac{L'}L(\sigma+it,\mathcal{F})
	  = &\,\, \frac{\log q}{2X^2} 
	\operatorname{Re} \sum_{f \in \CMcal{M}_{\leq 3X}}\frac{\Lambda_{X, \cF}(f)\chi_D(f)}{|f|^{\sigma+it}} 
	 +\frac{2\nu q^{X(\frac12-\sigma)}}{c^2\log q}
	 \operatorname{Re} \frac{L'}{L}\big(\sigma_0+it,\cF\big)
	\\
	&+O\bigg(\kappa_\cF\, q^{X(\frac12-\sigma)}+ \frac{\eta_\cF\, q^{-X\sigma} |1-q^{-Xs}|^2}{X^2}\bigg).
	\end{split}
	\end{equation}
Evaluating this at $\sigma=\sigma_0$ and by rearranging the terms, we obtain
	\begin{align*}
	&-  \Big(1+\frac{2\nu q^{-X(\sigma_0-\frac12)}}{c^2\log q}\Big)
	\operatorname{Re} \frac{L'}{L}\big(\sigma_0+it,\cF\big)
	\\
	&\leq   \frac{\log q}{2X^2} \operatorname{Re} \sum_{f \in \CMcal{M}_{\leq 3X}}\frac{\Lambda_{X, \cF}(f)\chi_D(f)}{|f|^{\sigma_0+it}}
	+O\bigg(\kappa_\cF\, q^{-X(\sigma_0-\tfrac12)}+ \frac{\eta_\cF\, q^{-X\sigma_0}  |1+q^{-X\sigma_0}|^2}{X^2}\bigg).
	\end{align*}
We choose $c$ such that $c < \frac{1}{2\log q}$ so that $\bigg|1+\frac{2\nu q^{-X(\sigma_0-\frac12)}}{c^2\log^2 q} \bigg| >1-\frac2{c^2q^c\log q} \gg 1$. This proves that
	\begin{equation}\label{LprimeL-in primes}
	\operatorname{Re} \frac{L'}L\big(\sigma_0+it, \cF\big)
	=O\bigg(\frac1{X^2}\bigg|\sum_{f \in \CMcal{M}_{\leq 3X}}
	\frac{\Lambda_{X, \cF}(f)\chi_D(f)}{|f|^{\sigma_0+it}}\bigg|\bigg)+O(\kappa_\cF).
	\end{equation}
Also, by combining the above expression with \eqref{eq:interm 1}, we find that
	\begin{align*}
	-\operatorname{Re} \frac{L'}L(\sigma+it,\cF)
	=&\, \frac{\log q}{2X^2} \operatorname{Re} 
	\sum_{f \in \CMcal{M}_{\leq 3X}}\frac{\Lambda_{X, \cF}(f)\chi_D(f)}{|f|^{\sigma+it}}+O\big(\kappa_\cF q^{X(\frac12-\sigma)}\big)
	\\
	&+O\bigg(\frac{q^{X(\frac12-\sigma)}}{X^2}
\bigg|\sum_{f \in \CMcal{M}_{\leq 3X}}\frac{\Lambda_{X, \cF}(f)\chi_D(f)}{|f|^{\sigma_0+it}}\bigg| \bigg)+O\Big(\frac{\eta_\cF \, q^{-X\sigma}  |1+q^{-X\sigma}|^2}{X^2}\Big).
	\end{align*}
Integrating with respect to $\sigma$ from $\sigma_0$ to $\infty$ completes the proof of \eqref{eq:logLformula1}. 
	
For the imaginary part, again by \eqref{log derivative version 3} and \eqref{bound on kappa}, we have
	\begin{align*}
	 \operatorname{Im} \frac{L'}{L}\big(\sigma+it,\mathcal{F}\big)
	=& \,  \frac{\log q}{2X^2} \operatorname{Im} \sum_{f \in \CMcal{M}_{\leq 3X}}\frac{\Lambda_{X, \cF}(f)\chi_D(f)}{|f|^{\sigma+it}}+O\big(\kappa_{\mathcal{F}}\, q^{X(\frac12-\sigma)}\big)\\
	&+O\bigg(q^{X(\frac12-\sigma)}
	\Big| \operatorname{Re} \frac{L'}{L}\big(\sigma_0+it,\cF\big)\Big|\bigg)
+O\Big( \frac{\eta_{\cF}\, q^{-X\sigma} |1-q^{-X(\sigma+it)}|^2}{X^2}\Big).
	\end{align*}
Then by \eqref{LprimeL-in primes}, we obtain    
	\begin{align*}
	\operatorname{Im} \frac{L'}{L}\big(\sigma+it,\cF\big)
	=&\,   \frac{\log q}{2X^2} \operatorname{Im} \sum_{f \in \CMcal{M}_{\leq 3X}}\frac{\Lambda_{X, \cF}(f)\chi_D(f)}{|f|^{\sigma+it}}
	\\
	&
	+ O\bigg(\frac{q^{X(\frac12-\sigma)}}{X^2}\bigg|\sum_{f \in \CMcal{M}_{\leq 3X}}  \frac{\Lambda_{X, \cF}(f)\chi_D(f)}{|f|^{\sigma_0+it}}\bigg|  \bigg)
	+O\bigg(\kappa_\cF\, q^{X(1/2-\sigma)} \bigg)         
	\\
	& +O\bigg(\kappa_\cF\, q^{X(\frac12-\sigma)}
	+ \frac{\eta_\cF\, q^{-X\sigma} |1-q^{-X(\sigma+it)}|^2}{X^2}\bigg).
	\end{align*}
Note that the second error term in the above will be absorbed by the third error term.    
	
	Moreover, by choosing $\sigma=\sigma_0$ in the above, since $q$ is fixed, we obtain a similar bound to \eqref{LprimeL-in primes} for the imaginary part. 
	\begin{equation}\label{LprimeL-in primes:im}
	\operatorname{Im} \frac{L'}L\big(\sigma_0+it,\cF\big)
	=O\bigg(\frac1{X^2}\bigg|\sum_{f \in \CMcal{M}_{\leq 3X}}
	\frac{\Lambda_{X, \cF}(f)\chi_D(f)}{|f|^{\sigma_0+it}}\bigg|\bigg)+O(\kappa_\cF).
	\end{equation}
We then integrate both sides of the above equation over $\sigma \in [\sigma_0, \infty]$ to find 
	\begin{align*}
	\arg L\big(\sigma_0+it,\cF\big)
	=\,&\frac{1}{2X^2}\operatorname{Im} \sum_{f \in \CMcal{M}_{\leq 3X}}
	\frac{\Lambda_{X, \cF}(f)\chi_D(f)}{d(f)|f|^{\sigma_0+it}} \\
	&+O\bigg(\frac1{X^3}\bigg|\sum_{f \in \CMcal{M}_{\leq 3X}}\frac{\Lambda_{X, \cF}(f)\chi_D(f)}{|f|^{\sigma_0+it}}\bigg|\bigg)
	+O\Big(\frac{\kappa_\cF}{X}+\frac{\eta_{\mathcal{F}}}{X^3}\Big).
	\end{align*}
	Upon combining this with \eqref{eq:logLformula1}, and then separating the contribution of $f$ with $X < d(f) \leq 3X$ to the Dirichlet polynomial in the main term, the result follows.
\end{proof}


\subsection{Moments of Tails of Dirichlet Polynomials Near the Critical Line}

We will need to study the first error term in the above proposition, which was 
\[
\frac{1}{2X^2} \bigg| \sum_{f\in\CMcal{M}_{\leq3X}\setminus\CMcal{M}_{\leq X}} \frac{\Lambda_{X, \cF}(f)\chi_D(f)}{d(f)|f|^{\sigma_0+it}}\bigg| ,
\]
to be small on average. 
We set
\begin{align}\label{general set}
\mathcal{H}(\mathcal{F})=\begin{cases}
\mathcal{H}_n \quad &\text{if }\, \mathcal{F}=\chi_D,\\
\mathcal{H}_{n}^{\Delta}(C) \quad &\text{if }\, \mathcal{F}=E\otimes \chi_D, (C, N_E)=1.
\end{cases}
\end{align}
For a fixed $(C, N_E)=1$, the size of the set $\mathcal{H}_{n}^{\Delta}(C)$ is given by \eqref{cardinality in AP}.


\begin{lem}\label{moments of lambda polyl}
	For $k$ a positive integer, let $2< X \leq \tfrac{n}{4k}$ and $\sigma_0$ as in the hypothesis of Proposition \ref{logL-as-PX}. Then  
	\[
	\frac{1}{|\mathcal{H}(\mathcal{F})|}
	\sum_{D\in \mathcal{H}(\mathcal{F})}
	\bigg| \frac{1}{X^2} 
	\sum_{f\in\CMcal{M}_{\leq3X}\setminus\CMcal{M}_{\leq X}}
	\frac{\Lambda_{X, \mathcal{F}}(f) \chi_D(f)}{d(f) |f|^{\sigma_0+it}}
	\bigg|^k
	= O \bigg( \Big(\frac{48k}{e}\Big)^{\frac{k}2}+(24)^k \bigg).
	\] 
\end{lem}


\begin{proof}
We can bound the $k$-th moment as follows. 
	\begin{align*}
	&\sum_{D\in \mathcal{H}(\mathcal{F})}
	\bigg| \frac{1}{X^2} 
	\sum_{f\in\CMcal{M}_{\leq3X}\setminus\CMcal{M}_{\leq X}}
	\frac{\Lambda_{X, \mathcal{F}}(f) \chi_D(f)}{d(f) |f|^{\sigma_0+it}}
	\bigg|^k 
	\\
	& \ll
	2^k \bigg| \frac{1}{X^2} 
	\sum_{D\in \mathcal{H}(\mathcal{F})}
	\sum_{P\in \CMcal{P}_{\le 3X}\setminus\CMcal{P}_{\le X}}
	\frac{\Lambda_{X, \mathcal{F}}(P) \chi_D(P)}{d(P) |P|^{\sigma_0+it}}
	\bigg|^k 
	+ 2^k \bigg| \frac{1}{X^2} 
	\sum_{D\in \mathcal{H}(\mathcal{F})}
	\sum_{P^2\in \CMcal{P}_{\le 3X}\setminus\CMcal{P}_{\le X}}
	\frac{\Lambda_{X, \mathcal{F}}(P^2) \chi_D(P^2)}{d(P^2) |P^2|^{\sigma_0+it}}
	\bigg|^k  
	\\
	&\, := S_1+S_2.
	\end{align*}
To further bound $S_2$, we note the easy bound $|\Lambda_{X, \mathcal{F}}(P^2)| \leq 8X^2 d(P)$, which uses $|\lambda(P)|\leq 2$. We also use the fact that $d(P^2)>X$, since $P^2\in \CMcal{P}_{\le 3X}\setminus\CMcal{P}_{\le X}$, to obtain
	\[
	S_2 \leq  16^k \sum_{D\in \mathcal{H}(\mathcal{F})} 
	\bigg( 
	\sum_{P^2\in\CMcal{P}_{\le 3X}\setminus\CMcal{P}_{\le X}}
	\frac{d(P) }{X |P|^{2\sigma_0}}
	\bigg)^k.
	\] 
Then by the prime polynomial theorem, 
	\[
	S_2 \leq \Big( \frac{16}{X}\Big)^k |\mathcal{H}(\mathcal{F})|  \bigg(  \sum_{X< 2m \leq 3X}
	\frac{m q^m}{m q^{2m\sigma_0}}\bigg)^k
	\ll (24)^k\, |\mathcal{H}(\mathcal{F})| .
	\]
For $S_1$, we have 
	\[
	\frac{S_1}{|\mathcal{H}(\mathcal{F})|}
	= \frac{1}{X^{2k}|\mathcal{H}(\mathcal{F})|} \sum_{P_1, \ldots, P_k \in \CMcal{P}_{\le 3X}\setminus\CMcal{P}_{\le X}} 
	\frac{\Lambda_{X, \mathcal{F}}(P_1) \cdots \Lambda_{X,\mathcal{F}}(P_k) }{d(P_1) \cdots d(P_k) |P_1 \cdots P_k |^{\sigma_0+it}}
	\sum_{D\in \mathcal{H}(\mathcal{F})}  \chi_D(P_1 \cdots P_k) .
	\]
	We first consider the case where $P_1 \dots P_k$ is a perfect square. Then $k$ is even and the irreducibles must pair up. By Lemmas \ref{square term evaluation} and \ref{sum over squares} depending on the family, the above is
	\begin{align*}
	\frac{1}{X^{2k}} \sum_{P_1, \ldots, P_k \in \CMcal{P}_{\le 3X}\setminus\CMcal{P}_{\le X}}
	\frac{\Lambda_{X, \mathcal{F}}(P_1) \cdots \Lambda_{X,\mathcal{F}}(P_k) }{d(P_1) \cdots d(P_k) |P_1 \cdots P_k |^{\sigma_0+it}}
	\Big(\prod_{1\leq j \leq k}\left(1+\frac{1}{|P_j|}\right)^{-1}+\mathcal{E}_{\mathcal{F}}\Big),
	\end{align*}
where  
	\begin{align}\label{error for square over familes}
	\mathcal{E}_{\mathcal{F}} \ll 
	\begin{cases}
	q^{-n} & \text{if } \mathcal{F} = \chi_D, \\[6pt]
	q^{-(3/4+\epsilon)n} & \text{if } \mathcal{F} = E \otimes \chi_D,
	\end{cases}
	\end{align}
Again, using $|\Lambda_{X, \mathcal{F}}(P)|\leq 4X^2  \Lambda(P)$ for $X \leq d(P) \leq 3X$ and the prime polynomial theorem, we obtain
	\begin{align*}
	\frac{S_1}{|\mathcal{H}(\mathcal{F})|} 
	&\ll
	\frac{k!}{(\frac{k}2)! 2^{\frac{k}2}}
	\frac{1}{X^{2k}}
	\bigg(
	\sum_{P\in \CMcal{P}_{\le 3X}\setminus\CMcal{P}_{\le X}} 
	\frac{16X^4 \Lambda(P)^2 }{d(P)^2 |P |^{2\sigma_0}}
	\frac{|P|}{|P|+1}\bigg)^{\frac{k}2}
	\\
	& \ll \frac{k!}{(\frac{k}2)! } 
	\Big( \sum_{X\leq m\leq 3X}\frac{8q^m}{mq^{2m(\sigma_0-\frac12)} (q^m+1)}\Big)^{\frac{k}2}
	\ll \frac{k!}{(\frac{k}2)! X^{\frac{k}2}} 
	\Big( \sum_{X\leq m\leq 3X}  \frac{1}{q^{2m(\sigma_0-\frac12)}}\Big)^{\frac{k}2}.
	\end{align*}
Thus, $S_1 \ll \frac{k! (24)^{\frac{k}2}}{ (\frac{k}2)! }$, which further gives by Stirling's approximation 
	\[
	S_1\ll \Big(\frac{48k}{e}\Big)^{\frac{k}{2}}.
	\]
	
	For the other case where $P_1 \dots P_k$ is not a perfect square, we apply Lemma \ref{polya-vinogradov inequality}, or \ref{sum over non-squares} depending on the family, to write 
	\[
	\frac{S_1}{|\mathcal{H}(\mathcal{F})|}
	\ll \frac{q^{-n/2}}{X^{2k}} \sum_{P_1, \ldots, P_k \in \CMcal{P}_{\le 3X}\setminus\CMcal{P}_{\le X}}
	\frac{\Lambda_{X, \mathcal{F}}(P_1) \cdots \Lambda_{X,\mathcal{F}}(P_k) }{d(P_1) \cdots d(P_k) |P_1 \cdots P_k |^{\sigma_0}} |P_1 \cdots P_k|^{\epsilon}.
	\]
 Then by $\Lambda_{X, \mathcal{F}}(P) \leq 4X^2 d(P)$, the above is
	\begin{align*}
	\leq \frac{q^{-n/2}}{X^{2k}} \sum_{P_1, \ldots, P_k \in \CMcal{P}_{\le 3X}\setminus\CMcal{P}_{\le X}}
	\frac{4^k X^{2k} }{ |P_1 \cdots P_k |^{\sigma_0-\epsilon}} .
	\end{align*}
Again, by the prime polynomial theorem, this is  
	\[
	\leq 4^k q^{-n/2} \Big(\sum_{X \leq m \leq 3X} \frac{ q^m }{m q^{m\sigma_0-m\epsilon}}\Big)^k
	\leq 
	4^k q^{-n/2} \Big(\sum_{X \leq m \leq 3X} \frac{ q^{(1/2+\epsilon)m}}{m}\Big)^k
	\leq 4^k q^{-n/2} 
	q^{3kX\big(\tfrac12+2\epsilon\big)}. 
	\]
For $X \leq \tfrac{n}{4k}$, we have our sum is bounded by $ \leq  4^k$. This completes the proof.
\end{proof}

Finally, we estimate the moments of the third error term in the approximation in Proposition \ref{logL-as-PX}.


\begin{lem}\label{moments of lambda polyl2}
In the setting of Lemma \ref{moments of lambda polyl} and under the same hypotheses, we have 
	\[
	\frac{1}{|\CMcal{H}(\mathcal{F})|}
	\sum_{D\in \CMcal{H}(\mathcal{F})}
	\bigg| \frac{1}{X^3} 
	\sum_{f\in \CMcal{M}_{\leq 3X}}
	\frac{\Lambda_{X, \mathcal{F}}(f) \chi_D(f)}{|f|^{\sigma_0+it}}
	\bigg|^k
	= O\bigg( \Big(\frac{72k}{e} \Big)^{\frac{k}{2}}+\Big(\frac4{X}\Big)^k\bigg).
	\] 
\end{lem}


\begin{proof}
	This is similar to the proof of Lemma \ref{moments of lambda polyl}.
\end{proof}


\section{$\log L(s, \mathcal{F})$ Near vs at the Critical Line}\label{discrepancy}

We have so far proven an approximate formula for $\log L(\sigma_0+it,\mathcal{F})$. The purpose of this section is to show that we may study these values at $1/2+it$ in place of those of $\log L(\sigma_0+it,\mathcal{F})$.   

Our first result considers microscopic $t$. This means that we are very close to the real line, and therefore, our result and its proof are very similar to those of Proposition 6.1 in \cite{DL}. We do this for the symplectic one first, and point out the differences for the proof for the orthogonal family as necessary.


\subsection{Microscopic Regime}


\begin{prop}\label{difference between shift of logarithm}
Consider $t\in \mathbb{R}$ such that $|t|g<\infty \pmod{\frac{4\pi }{\log q}}$. Assume that the low lying zeros hypothesis holds and $g\big(\sigma_0-\tfrac12\big)\to \infty$ but $g\big(\sigma_0-\tfrac12\big)=o(\sqrt{\log n})$ as $g\to \infty$. Then for a subset $ \CMcal{H}'_{n}$ of $ \CMcal{H}_{n}$ with measure $1-o(1)$, we have
	\[
	\frac{1}{| \CMcal{H}_{n}|}\sum_{ D \in  \CMcal{H}'_{n}} \Big|\log\big| L\big(\sigma_0+it,\chi_D\big)\big| -\log \big| L\big(\tfrac12+it,\chi_D\big)\big|\Big|
	=o\big(\sqrt{\log g}\big).
	\]
\end{prop}
Note that as in \cite{DL}, we won't study the argument of the $L$-function within the microscopic regime. The reason for this is that the variance of the imaginary part of $\mathcal{D}_X\big(\sigma_0+it,\chi_D\big)$ is $0$ within this regime, and our method doesn't provide a result on the distribution of $\arg{ L\big(\sigma_0+it,\chi_D\big)}$. It would be interesting to determine the exact normalization factor for which a central limit theorem holds in this regime.

\begin{proof}
Given that $t$ lies in the regime, by the low lying zeros hypothesis there exists a subset
 $\mathcal{H}'_{n} \subset \mathcal{H}_{n},$
of measure $1 - o(1)$, such that for all $D \in \mathcal{H}'_{n}$ and for all zeros of $L(s,\chi_D)$, we have 
\begin{equation*}
\min_{1 \leq j \leq 2g} \left| \theta_{j,D}\right| > \frac{1}{yX} 
\implies 
\min_{1 \leq j \leq 2g} \left| \theta_{j,D} + \frac{t \log q}{2\pi} \right| > \frac{1}{yX},
\end{equation*}
where  $y = y(g) \to \infty$ as  $g \to \infty$, and will be chosen later in an appropriate manner.

Following the proof in~\cite[Proposition 6.1]{DL} and using \eqref{LprimeL-in primes}, we obtain
\begin{align}
&\log \big| L\big(\sigma_0+it,\chi_D\big)\big|
- \log \big|L\big(\tfrac12+it,\chi_D\big)\big| \nonumber\\ 
&=\log q\Big(g \big(\tfrac12-\sigma_0\big)
+\eta  \operatorname{Re} \int_0^{\sigma_0-\frac12} 
\frac{ q^{-\sigma-1/2-it}d\sigma }{1-q^{-\sigma-1/2-it}}
+ \sum_{j=1}^{2g} \, \operatorname{Re} \int_0^{\sigma_0-\frac12}
\frac{(1+\alpha_j q^{-\sigma-it})d\sigma}{2(1-\alpha_j q^{-\sigma-it})}\Big)  \label{eq:at sigma0 at onehalf}\\ 
&\ll \frac{y\big(\sigma_0-\tfrac12\big)}{X^2}
\bigg|
\sum_{f \in \mathcal{M}_{\leq 3X}} 
\frac{\Lambda_{X, \mathcal{F}}(f)\chi_D(f)}{|f|^{\sigma_0+it}}
\bigg|
+\, gy\big(\sigma_0-\tfrac12\big)
+ g\big(\sigma_0-\tfrac12\big)\log q
+ 1.\nonumber
\end{align}

We now choose the parameters $\sigma_0, X,$ and $y$ as functions of $g$ such that 
\[
\sigma_0-\tfrac12 = \frac{c}{X} 
\quad \text{and} \quad 
\frac{yg}{X} = o(\sqrt{\log n}).
\]
With this choice, the hypotheses of the proposition are satisfied. Note that the condition 
$\tfrac{yg}{X} = o(\sqrt{\log n})$ further implies that
$
yg\big(\sigma_0-\tfrac12\big) = o(\sqrt{\log n}),$
and thus, the above inequality will yield the desired bound once we show that
\[
\frac{1}{| \mathcal{H}_n |}
\sum_{D \in \mathcal{H}'_n}
\bigg|
\frac{y\big(\sigma_0-\tfrac12\big)}{X^2}
\sum_{f \in \mathcal{M}_{\leq 3X}} 
\frac{\Lambda_{X, \mathcal{F}}(f)\chi_D(f)}{|f|^{\sigma_0+it}}
\bigg|
= O(y^2) = o(\log n),
\]
for a subset  $\mathcal{H}'_n \subset \mathcal{H}_n$ with  $| \mathcal{H}_n \setminus \mathcal{H}'_n | = o(1)$.
This follows from the proof of \cite[Proposition~6.1]{DL} together with Lemma~\ref{moments of lambda polyl2}.
\end{proof}

We now turn to linear combinations and record the following corollary. 


\begin{cor}\label{difference between logarithm micro}
Assume that $\vec{t}$ and $\sigma_0$ are as in Proposition \ref{difference between shift of logarithm} and that the low lying zeros hypothesis holds. Then for a subset $ \CMcal{H}'_{n}$ of $ \CMcal{H}_{n}$ with measure $1-o(1)$ as $g\to\infty$, 
	\[
	\frac{1}{| \CMcal{H}_{n}|}\sum_{ D \in  \CMcal{H}'_{n}} 
	\Big| \log\big| \mathfrak{L}_{\vec{a}, \vec{t}}\big(\sigma_0, \chi_D\big)\big| 
	- \log\big| \mathfrak{L}_{\vec{a}, \vec{t}}\big(\tfrac12, \chi_D\big)\big|\Big|
	=o\big(\sqrt{\log g}\big),
	\]
	where $\mathfrak{L}_{\vec{a}, \vec{t}}\left(s, \chi_D\right)$ is defined by \eqref{def-L-linear combo}.
\end{cor}


\begin{proof}
	This follows immediately by applying the triangle inequality together with Proposition \ref{difference between shift of logarithm}.
\end{proof}

We now briefly address the differences from the symplectic family, focusing on the distribution of $\log |L(\sigma_0+it, E\otimes \chi_D)|$ for suitable $\sigma_0$ near one-half within an arithmetic progression.


\begin{prop}\label{close to 1/2-line}
	Assume that $t\in \mathbb{R}$ lies in the microscopic regime. Suppose that, low lying zeros hypothesis for twists of elliptic curves holds for $\{L(s, E\otimes \chi_D)\}_{D\in  \CMcal{H}_{n}^{\Delta}(C)}$ for each $(C, N_E)=1$. Also assume that $n\left(\sigma_0-1/2\right)\to \infty$ and $n\left(\sigma_0-1/2\right)=o\left(\sqrt{\log n}\right)$ as $n \to \infty$. Then there exist a subset $\widetilde{ \CMcal{H}}_{n}^{\Delta}(C) \subset  \CMcal{H}_{n}^{\Delta}(C)$ of measure $1-o(1)$ for which
	\[
	\frac1{\big| \CMcal{H}_{n}^{\Delta, +}\big|}\sum_{D\in \widetilde{ \CMcal{H}}_{n}^{\Delta}(C)} \Big|\log\big|L\big(\sigma_0+it, E\otimes \chi_D\big)\big|
	- \log \big|L\big(\tfrac{1}{2}+it, E\otimes \chi_D\big)\big|
	 \Big|=o\left(\sqrt{\log n}\right).
	\]
\end{prop}


\begin{proof}
The low lying zeros hypothesis for twists of elliptic curves implies that there is a subset $\widetilde{ \CMcal{H}}_{n}^{\Delta}(C) \subset  \CMcal{H}_{n}^{\Delta}(C)$ of measure $1-o(1)$ so that for all $D\in \widetilde{ \CMcal{H}}_{n}^{\Delta}(C)$ and for all zeros of $\mathcal{L}(u, E\otimes \chi_D)$, the condition
	\[
	\min_{j}\left|\theta_{j, C}(E\otimes \chi_D)+\frac{t \log q}{2 \pi}\right|>\frac{1}{yX}
	\]
	holds for each $(C, N_E)=1$.

	The remainder of the proof proceeds analogously to that of Proposition~\ref{difference between shift of logarithm}, using the same choice of parameters $ \sigma_0$, $X := X(n)$, and $y := y(n)$, together with the relations given in~\eqref{relation bet. full fam and subfam} and~\eqref{cardinality in AP}.
\end{proof}


\subsection{Mesoscopic and Macroscopic Regimes}

Next, we consider mesoscopic and macroscopic $t$. Recall that $\mathcal{F}$ denotes the quadratic character $\chi_D$ or the twisted elliptic curve $E\otimes \chi_D$.


\begin{prop}\label{difference between shift of logarithm2}
	Let $\vec{t}$ lie in either of the mesoscopic or macroscopic regimes. Let $\kappa_\cF$ be either $2g$ or $m$ depending on the family such a way that  $\tfrac{\kappa_\mathcal{F}}X=o(\sqrt{\log(\kappa_\cF)}).$
	 Then we have 
	\[
	\frac{1}{| \CMcal{H}(\mathcal{F})|}\sum_{ D \in  \CMcal{H}(\mathcal{F})} 
	\Big|\log\big| \mathfrak{L}_{\vec{a}, \vec{t}}\big(\sigma_0, \mathcal{F}\big)\big| 
	-\log\big| \mathfrak{L}_{\vec{a}, \vec{t}}\big(\tfrac12, \mathcal{F}\big)\big|\Big|
	=o\big(\sqrt{\log(\kappa_\cF)}\big),
	\]
	and 
	\[
	\frac{1}{| \CMcal{H}(\mathcal{F})|}\sum_{ D \in  \CMcal{H}(\mathcal{F})}
	\Big| \arg \mathfrak{L}_{\vec{a}, \vec{t}}\big(\sigma_0, \mathcal{F}\big) 
	- \arg \mathfrak{L}_{\vec{a}, \vec{t}}\big(\tfrac12, \mathcal{F}\big)\Big|
	=o\big(\sqrt{\log(\kappa_\cF)}\big),
	\]
	where $\mathfrak{L}_{\vec{a}, \vec{t}}\big(\sigma_0, \mathcal{F}\big)$ are defined by  \eqref{def-L-linear combo} and \eqref{def-L-linear comboe} depending on family.
\end{prop}


\begin{proof} 
	It suffices to establish the result for a single shift 
	$t$; the general case then follows immediately from the triangle inequality. 
	
	Suppose that $\theta_j:=\theta_{j,D}(\mathcal{F}) \neq \pm \tfrac{t\log q}{2\pi}$ for all $1 \leq j \leq \kappa_\cF$, where 
	\begin{align}\label{eigen phases}
	\theta_{j,D}(\mathcal{F}) =
	\begin{cases}
	\theta_{j,D} & \text{if } \mathcal{F} = \chi_D, \\[6pt]
	\theta_{j,C}(E \otimes \chi_D) & \text{if } \mathcal{F} = E \otimes \chi_D.
	\end{cases}
	\end{align}
	 Using this condition and following arguments in the proof of Proposition \ref{difference between shift of logarithm}, we have
	\begin{equation}\label{real differences}
	\begin{split}
	\Big|\log\big| L\big(\sigma_0+it,\mathcal{F}\big)\big| 
	&-\log \big|L\big(\tfrac12+it, \mathcal{F}\big)\big|\Big|
	\ll 	
	\sum_{j=1}^{\kappa_\cF}	\frac{q^{\sigma_0-\frac12}+q^{\frac12-\sigma_0}-2}{1- \cos(2\pi\theta_j+t\log q)}+\kappa_\cF  \big(\sigma_0-\tfrac12\big) \\
	&\ll \sum_{j=1}^{\kappa_\cF}
	\frac{\big(\sigma_0-\frac12\big)\left(1-q^{1-2\sigma_0}\right)}{1-q^{\frac12-\sigma_0}\cos(2\pi\theta_j+t\log q)+q^{1-2\sigma_0}}+\kappa_\cF  \big(\sigma_0-\tfrac12\big)
	\\
	& \ll \sigma_0\sum_{j=1}^{\kappa_\cF}
	\operatorname{Re}\Big( \frac{1}
	{1-\alpha_j q^{\frac12-\sigma_0-it}} -\frac12 \Big)+\kappa_\cF \big(\sigma_0 -\tfrac12\big)
	\\
	 &\ll \frac{\sigma_0-\tfrac12}{X^2}\bigg|\sum_{f \in \CMcal{M}_{\leq 3X}} \frac{\Lambda_{X, \mathcal{F}}(f)\chi_D(f)}{|f|^{\sigma_0+it}}\bigg|+\kappa_\cF  \big(\sigma_0-\tfrac12\big).
	\end{split}
	\end{equation}
	Using Lemma \ref{moments of lambda polyl2} and the fact that $\sigma_0=\tfrac12+\frac{c}{X}$, we conclude that
	\[
	\frac{1}{| \CMcal{H}(\mathcal{F})|}\sum_{ D \in  \CMcal{H}(\mathcal{F})} \Big|\log\big| L\big(\sigma_0+it,\mathcal{F}\big)\big| -\log \big|L\big(\tfrac12+it,\mathcal{F}\big)\big|\Big|\ll_{\mathcal{F}} \kappa_\cF\big(\sigma_0-\tfrac12\big)=o( \sqrt{\log(\kappa_\cF)}),
	\]
	which proves our claim for the real part of the logarithms.
	
	Now we shall prove the second part of the proposition. We have
	\begin{equation}\label{imaginary part estimate}
	\begin{split}
	&\arg L\big(\tfrac12+it,\mathcal{F}\big)-\arg L\big(\sigma_0+it,\mathcal{F}\big) 
	=  -\int_{0}^{\sigma_0-\frac12} 
	\operatorname{Im} \frac{L'}{L}\Big(\frac12+\sigma+it, \mathcal{F}\Big)
	d{\sigma} \\
	&= -\big(\sigma_0-\tfrac12\big) 	\operatorname{Im} \frac{L'}{L}\big(\sigma_0+it, \mathcal{F}\big)+\operatorname{Im} \int_{0}^{\sigma_0-\frac12} \Big( \frac{L'}{L}\big(\sigma_0+it, \mathcal{F}\big)- \frac{L'}{L}\big(\tfrac12+\sigma+it, \mathcal{F}\big) \Big)
	d{\sigma}. 
	\end{split}
	\end{equation}
	From \eqref{LprimeL-in primes:im}, the first term of the above expression can be estimated as
	\[
	\big(\sigma_0-\tfrac12\big) \,	\operatorname{Im} \frac{L'}{L}\big(\sigma_0+it, \mathcal{F}\big)\ll  \frac{\sigma_0-\tfrac12}{X^2}
	\, \bigg|
	\sum_{f \in \CMcal{M}_{\leq 3X}}\frac{\Lambda_{X, \mathcal{F}}(f)\chi_D(f)}{|f|^{\sigma_0+it}}\bigg|+\kappa_\cF\big( \sigma_0-\tfrac12\big).
	\]	
Using \eqref{eq:log der symp zeros}, we write 
	\begin{align*}
	&\operatorname{Im}  \frac{L'}{L}\big(\sigma_0+it, \mathcal{F}\big)
	- \operatorname{Im} \frac{L'}{L}\big(\tfrac12+\sigma+it, \mathcal{F}\big)
	\\
	&= \sum_{j=1}^{\kappa_\cF}\operatorname{Im}\left(\frac{1}{1-\alpha_j q^{\frac12-\sigma_0-it}}-\frac{1}{1-\alpha_j q^{-\sigma-it}}\right)
	+\eta_{\mathcal{F}} \log q \operatorname{Im}\bigg( \frac{q^{-\sigma_0-it}}{1-q^{-\sigma_0-it}}- \frac{q^{-\sigma-1/2-it}}{1-q^{-\sigma-1/2-it}}\bigg)
	\\
	&=\sum_{j=1}^{\kappa_\cF}\frac{\left(q^{-\sigma}-q^{\frac12-\sigma_0}\right)\left(1-q^{\frac12-\sigma-\sigma_0}\right)\sin \left(2\pi \left(\theta_{j}+\frac{t\log q}{2\pi}\right)\right)}{\left(1+q^{-2\sigma}-2q^{-\sigma}\cos \left(2\pi \left(\theta_{j}+\frac{t\log q}{2\pi}\right)\right)\right)\left(1+q^{1-2\sigma_0}-2q^{\frac12-\sigma_0}\cos \left(2\pi \left(\theta_{j}+\frac{t\log q}{2\pi}\right)\right)\right)}
	\\
	&+\eta_{\mathcal{F}} \log q \, \operatorname{Im}\bigg( \frac{q^{-\sigma_0-it}}{1-q^{-\sigma_0-it}}- \frac{q^{-\sigma-1/2-it}}{1-q^{-\sigma-1/2-it}}\bigg) .
	\end{align*}
We need to estimate the following. 
	\[
	\operatorname{Im} \int_0^{\sigma_0-\frac12} 
	\frac{ q^{-\sigma-1/2-it}}
	{1-q^{-\sigma-1/2-it}}d\sigma 
	=\operatorname{Im} \frac{\log(1-q^{-\sigma-1/2-it})}{\log q} \bigg|_0^{\sigma_0-\frac12} .
	\]
Note that for $\sigma\in [0, \sigma_0-\tfrac12]$, we have the bounds 
	\[
	q^{-\sigma}-q^{\frac12-\sigma_0}\leq 1-q^{\frac12-\sigma_0}\ll 1-q^{1-2\sigma_0}
	\quad
	\text{and}
	\quad
	1-q^{\frac12-\sigma-\sigma_0}\ll \sigma_0-\tfrac12.
	\]
	Putting all these together, the second integral in \eqref{imaginary part estimate} is
	\begin{align*}
	& \ll \Big(\sigma_0-\tfrac12\Big) \sum_{j=1}^{\kappa_\cF}
	\frac{1-q^{1-2\sigma_0}}{2(1+q^{1-2\sigma_0}-2q^{\frac12-\sigma_0}\cos(2\pi\theta_{j}+t\log q))}
	\\
	&\quad \times \int_0^{\sigma_0-\frac12}\frac{\left|\sin \left(2\pi \theta_{j}+t\log q\right)\right| d\sigma}{(q^{\sigma}-\cos (2\pi\theta_{j}+t\log q))^2+\sin^2(2\pi\theta_{j}+t\log q))}\\
	&\ll \Big(\sigma_0-\tfrac12\Big)  \, \Big|\frac{L'}{L}\big(\sigma_0+it, \mathcal{F}\big)\Big|
	\ll \Big(\sigma_0-\tfrac12\Big) \bigg(\frac{1}{X^2}\Big|
	\sum_{f \in \CMcal{M}_{\leq 3X}}\frac{\Lambda_{X, \mathcal{F}}(f)\chi_D(f)}{|f|^{\sigma_0+it}}\Big|+\kappa_\cF \bigg), 
	\end{align*}
since the integral above is bounded, and where the last inequality followed from combining \eqref{LprimeL-in primes} and \eqref{LprimeL-in primes:im}. The remaining portion of the proof is identical to that of the real case. Lastly, if $\theta_{j}=\pm \frac{t\log q}{2\pi}$ for some $1\leq j\leq \kappa_\cF$, then we can use the continuity at $\frac12+it$ and the discreteness of the set of all zeros of $L(s, \chi_D)$. Indeed, the estimates \eqref{real differences} and \eqref{imaginary part estimate} are uniform in terms of $t$. So, if we define a function
	\[
	g(t, \mathcal{F})=\log L\big(\sigma_0+it, \mathcal{F}\big) 
	-\log L\big(\tfrac12+it, \mathcal{F}\big),
	\] 
then we can define the value of $g(t, \mathcal{F})$ for $\theta_{j}=\pm \frac{t\log q}{2\pi}$ to be
	\begin{align*}
	\operatorname{Re} g(t, \mathcal{F}) =\lim_{\epsilon \to 0}\frac{\operatorname{Re} g(t+\epsilon, \mathcal{F})+\operatorname{Re} g(t-\epsilon, \mathcal{F})}{2}, \quad 
		\operatorname{Im} g(t, \mathcal{F}) =\lim_{\epsilon \to 0}\frac{\operatorname{Im} g(t+\epsilon, \mathcal{F})+\operatorname{Im} g(t-\epsilon, \mathcal{F})}{2}.
	\end{align*}
This completes the proof.
\end{proof}


\section{Moments of Linear Combinations of Dirichlet Polynomials}\label{moments of lin comb}
Let $\vec{a}=(a_1, \ldots, a_k)\in \mathbb{R}^k$ and $\vec{t}=(t_1, \ldots, t_k)\in \mathbb{R}^k$ be fixed. We set
\begin{align}
\mathcal{P}_X\left(s, \mathcal{F}\right) 
:=\sum_{P\in \CMcal{P}_{\leq X}}
\frac{\Lambda_{\mathcal{F}}(P)\chi_D(P)}{|P|^{s}},
\end{align}
where $\Lambda_{\cF}$ is defined by \eqref{eq:log der symp}.
 Also define 
\begin{align*}
\mathcal{P}_X^{\vec{a}, \vec{t}}\big(s, \mathcal{F}\big)
&:=
a_1 \mathcal{P}_X\big(s+it_1, \mathcal{F}\big)
+\cdots+
a_k	\mathcal{P}_X\big(s+it_k, \mathcal{F}\big).
\end{align*}


\begin{prop}\label{moment computation 1 re}
	Assume that  $1\leq X\le \frac{\kappa_\cF}{2r}$. Uniformly for all  odd natural numbers  $r\leq {\frac{\kappa_\cF}{2}}$ and for every $\epsilon>0$, 
	\begin{align*}
	\frac{1}{| \CMcal{H}(\mathcal{F})|}\sum_{D\in  \CMcal{H}(\mathcal{F})} 
	\Big(\operatorname{Re} \mathcal{P}_X^{\vec{a}, \vec{t}}\big(\sigma_0, \mathcal{F}\big) \Big)^r
	\ll_{\mathcal{F}, \epsilon, \vec{a}} \frac{q^{-n/2+\left(1+\epsilon-\sigma_0\right)Xr}}{X^r} \left(2(|a_1|+\cdots+|a_k|)\right)^r.
	\end{align*}
Uniformly for all even natural numbers  $r\ll (\log (\kappa_{\mathcal{F}}))^{\frac{1}{3}}$, 
	\begin{align*}
	&\frac{1}{|\CMcal{H}(\mathcal{F})|}\sum_{D\in  \CMcal{H}(\mathcal{F})} \Big(\operatorname{Re} \mathcal{P}_X^{\vec{a}, \vec{t}}\big(\sigma_0, \mathcal{F}\big) \Big)^r
	=\frac{r!}{(\frac{r}2)! 2^{\frac{r}{2}}}
	\Big(\mathcal{V}_{\operatorname{Re}}(\vec{a}, \vec{t}, X)\Big)^{\frac{r}2}\bigg(1+O\Big(\frac{r^3}{\log X}\Big)\bigg),
	\end{align*}
where $\mathcal{V}_{\operatorname{Re}}(\vec{a}, \vec{t}, X)$ is defined by \eqref{variance for real}.	
\end{prop}


\begin{proof}
	We introduce the following notations.
	\[
	S_r
	: = \frac1{|\CMcal{H}(\mathcal{F})|} \sum_{D\in  \CMcal{H}(\mathcal{F})} \Big(\operatorname{Re} \mathcal{P}_X^{\vec{a}, \vec{t}}\big(\sigma_0, \mathcal{F}\big)\Big)^r 
	= 
	\sum_{\substack{P_j\in\CMcal{P}_{\leq X}\\ j=1, \ldots, r}}\frac{b_\mathcal{F}(P_1) \dots b_\mathcal{F}(P_r)}{|P_1\ldots P_r|^{\sigma_0}}s_r(P_1,\dots, P_r) ,
	\]
	where we set
	\[
	s_r(P_1,\dots, P_r) :=\frac1{| \CMcal{H}(\mathcal{F})|}
	\sum_{D\in  \CMcal{H}(\mathcal{F})}\chi_D(P_1\cdots P_r),
	\]
	and 
	\[
	b_\mathcal{F}(P_j)= \lambda_\mathcal{F}(P_j)\Big(a_1\cos(t_1\log|P_j|)+\cdots+a_k\cos(t_k\log|P_j|)\Big) \quad \text{for} \quad j=1, 2, \dots, r.
	\]	
	Note that the character sum $s_r(P_1,\dots, P_r)$ has a different size depending on whether or not $P_1 \cdots P_r=\square$.
	Therefore, we consider three cases.
	\begin{enumerate}
		\item[$(1)$] $r$ is even and $P_1\cdots P_r = \square$,
		\item[$(2)$] $r$ is even and $P_1\cdots P_r \neq \square$,
		\item[$(3)$] $r$ is odd.
	\end{enumerate}
	
	\medskip
	
	\noindent
	\textbf{Case (1).} \label{even and square case:re}
	Suppose that $r$ is even and $P_1 \cdots P_r = \square$.  
	
	\begin{enumerate}
		\item[$(a)$] Assume that there are exactly $\tfrac{r}{2}$ distinct primes called 
		$Q_1, \ldots, Q_{r/2}$ such that 
		$
		P_1\cdots P_r = (Q_1\cdots Q_{r/2})^2,
		$
		with 
		\[
		d(Q_1) \leq d(Q_2) \leq \cdots \leq d(Q_{r/2}).
		\]
	By Lemma \ref{square term evaluation} and Lemma \ref{sum over squares} depending on the family, 
			\[
			s_r(P_1,\dots, P_r) := \prod_{j=1}^{r/2} \frac{|Q_j|}{|Q_j|+1} + \mathcal{E}_{\mathcal{F}},
			\]    
			where $\mathcal{E}_{\mathcal{F}}$ is defined as in \eqref{error for square over familes}. By using the bounds $|b_\mathcal{F}(P_j)|\leq 2(|a_1|+\cdots+|a_k|)$ and $|\lambda_\mathcal{F}(P)|\leq 2$, and Lemma \ref{upper bound for truncated sum over primes }, we find that the contribution of such primes is 
\begin{equation}\label{eq:s_n case 1}
\begin{split}
& \frac{r!}{(\frac{r}2)! 2^{\frac{r}2}}
\sum_{\substack{Q_j\in\CMcal{P}_{\leq X}\\ j=1, \ldots, \frac{r}2}}\prod_{j=1}^{\frac{r}2}
\frac{b^2_\mathcal{F}(Q_j)}{(|Q_j|+1)|Q_j|^{2\sigma_0-1}} 
+ O\bigg(\mathcal{E}_{\mathcal{F}}\frac{r!\big(2(|a_1|+\cdots+|a_k|)\big)^{r}}{(\frac{r}2)! 2^{\frac{r}{2}}} \sum_{\substack{Q_j\in\CMcal{P}_{\leq X}\\ j=1, \ldots, \frac{r}2}}\frac{1}{|Q_1\cdots Q_{\frac{r}2}|^{2\sigma_0}}\bigg)    \\
&=\frac{r!}{(\frac{r}2)! 2^{\frac{r}2}}\sum_{\substack{Q_j\in\CMcal{P}_{\leq X}\\ j=1, \ldots, \frac{r}2}}\prod_{j=1}^{\frac{r}2}\frac{{b_\mathcal{F}}^2(Q_j)}{(|Q_j|+1)|Q_j|^{2\sigma_0-1}} 
+O\bigg(\mathcal{E}_{\mathcal{F}} \frac{r!\big(2(|a_1|+\cdots+|a_k|)\big)^{r}}{(\frac{r}{2})! 2^{\frac{r}2}}(\log X)^{\frac{r}{2}}\bigg),			
\end{split} 
\end{equation}
			Now, we focus on the sum in the main term of the last line above. If we fix $Q_1, \ldots, Q_{\frac{r}{2}-1}$ and vary over $Q_{\frac{r}2}$, then this sum can be seen to be
			\begin{align}\label{prime sum for cosine}
			\sum_{Q_{\frac{r}2}\in\CMcal{P}_{\leq X}}\frac{b_{\mathcal{F}}^2(Q_{\frac{r}2})}{(|Q_{\frac{r}2}|+1)|Q_{\frac{r}2}|^{2\sigma_0-1}}+O_k\Big(\frac{r}2\Big).
			\end{align}
Note that
			\begin{align*}
			b_{\mathcal{F}}^2(Q_{\frac{r}2})=
			\lambda_{\mathcal{F}}^2(Q_{\frac{r}2})\Bigg(\sum_{j=1}^ka_j^2\cos^2\big(t_j\log|Q_{\frac{r}2}|\big)
			+2\sum_{\substack{i, j=1\\i<j}}^k a_ja_j\cos\big(t_i\log|Q_{\frac{r}2}|\big)\cos\big(t_j\log|Q_{\frac{r}2}|\big)\Bigg),
			\end{align*}
and we have the identity
			\[
			2\cos\big(t_i\log|Q_{\frac{r}2}|\big)\cos\big(t_j\log|Q_{\frac{r}2}|\big)
			=\cos\big((t_i+t_j)\log|Q_{\frac{r}2}|\big)
			+
			\cos\big((t_i-t_j)\log|Q_{\frac{r}2}|\big).
			\]
By the prime number theorem, we have
			\[
			\sum_{d(P)\leq X}\frac{\cos^2(td(P))}{|P|}=  \sum_{n\leq X}\frac{\cos^2 (nt)}{n}
			+O\bigg(\sum_{n\le X}\frac{\cos^2(nt)}{nq^{n/2}}\bigg)
			=\frac12\sum_{n\leq X}\frac{1}{n}+\frac12\sum_{n\leq X}\frac{\cos(2nt)}{n}+O(1).
			\]
For the orthogonal family, from \eqref{lambda on primes} and prime polynomial theorem, we will use the following estimate instead.
			\[
			\sum_{d(P)\leq X}\frac{\lambda(P)^2\cos(td(P))}{|P|}=\sum_{n\leq X}q^{-n}\cos(nt)\sum_{d(P)=n}\lambda(P)^2=\sum_{n\leq X}\frac{\cos(nt)}{n}+O(1).
			\]
Then by using Lemma \ref{logarithmic cosine sum}, we see that the sum in \eqref{eq:s_n case 1} is
\begin{align*}
&
\bigg(\frac{a_1^2+\cdots+a_k^2}{2}\log X
+ \frac12\sum_{j=1}^k a_j^2 \log\Big(\min\Big\{ X, \frac{1}{2|t_j|}\Big\}\Big)
			\\
& \quad +\sum_{\substack{i,j=1\\i<j}}^k a_i a_j \left(\log\Big(\min\Big\{ X, \frac{1}{|t_i+t_j|}\Big\}\Big)
+\log\Big(\min\Big\{ X, \frac{1}{|t_i-t_j|}\Big\}\Big)\right)
+O(r) \bigg)^{\frac{r}{2}}, 
\end{align*}
which is, as per our definition earlier, denoting $\mathcal{V}:=\mathcal{V}_{\operatorname{Re}}(\vec{a}, \vec{t}, X)$      
			\[
\mathcal{V} ^{\frac{k}{2}}\left(1+O\bigg(\frac{r}{\mathcal{V}}\bigg)\right)^{\frac{r}{2}}.
			\]
		\item[$(b)$] Let there be $Q_1, \ldots, Q_\ell$ distinct primes with $\ell< \frac{r}2$ and  
			\[
			d(Q_1)\leq \ldots \leq d(Q_\ell) \quad \text{and} \quad P_1\cdots P_r=(Q_1\cdots Q_{r/2})^2.
			\]
Using Lemmas \ref{upper bound for truncated sum over primes }, \ref{square term evaluation}, and \ref{sum over squares}, we obtain
\begin{align*}
S_r = \, &  \sum_{\ell<\frac{r}{2}}\frac{r!}{\ell!2^\ell}\binom{r-\ell}{\ell}
\bigg(\sum_{P\in\CMcal{P}_{\leq X}}\frac{\lambda_{\mathcal{F}}^2(P)\big(a_1 \cos(t_1 \log |P|)+\cdots+ a_k \cos(t_k \log |P|)\big)^2}{(|P|+1)|P|^{2\sigma_0-1}}\bigg)^r\\
&+O\bigg(\mathcal{E}_{\mathcal{F}}\sum_{\ell<\frac{r}{2}}\frac{r!}{\ell!2^\ell}\binom{r-\ell}{\ell}
\mathcal{V}^\ell\bigg) \leq \sum_{\ell<\frac{r}{2}}\frac{r!}{\ell!2^\ell}\binom{r-\ell}{\ell}\big(|\mathcal{V}|+O(1)\big)^\ell
+O\bigg(\mathcal{E}_{\mathcal{F}}\sum_{\ell<\frac{r}{2}}\frac{r!}{\ell!2^\ell}\binom{r-\ell}{\ell}
\mathcal{V}^\ell\bigg)\\
&\ll_{ \vec{a}, k} \frac{r^3 r!}{(\frac{r}2)!2^{\frac{r}2}\log X}
\big|\mathcal{V}\big|^{\frac{r}2}
+O\bigg(\mathcal{E}_{\mathcal{F}}\frac{r^3 r!}{(\frac{r}2)!2^{\frac{r}2}} 
\mathcal{V}^{\frac{r}{2}}\bigg).
\end{align*}
\end{enumerate}
	
	\medskip
	
	\noindent
	\textbf{Case (2).}\label{even and nonsquare case:re}
	Let $r$ be even and suppose that $P_1\ldots P_r\neq \square$. Then we use Lemma \ref{polya-vinogradov inequality} for the symplectic family to get
		\[
		S_r
		\ll q^{(-1/2+\epsilon)n} \sum_{\substack{P_j\in\CMcal{P}_{\leq X}\\ j=1, \ldots, r}}\frac{\left(|a_1|+\cdots +|a_k|\right)^r}{|P_1\ldots P_r|^{\sigma_0-\epsilon}}
		\ll_{\vec{a}}  \frac{q^{-\frac{n}{2}+\left(1-\sigma_0+\epsilon \right) Xr}}{X^r}\left(|a_1|+\cdots+|a_k|\right)^r.
		\]
 For the orthogonal family, by using Lemma \ref{sum over non-squares} and \eqref{cardinality in AP}, we obtain
	\[
	S_r
	\ll_E q^{(-\frac12+\epsilon)n} \sum_{\substack{P_j\in\CMcal{P}_{\leq X}\\ j=1, \ldots, r}}\frac{\left(2(|a_1|+\cdots+|a_k|)\right)^r}{|P_1\ldots P_r|^{\sigma_0-\epsilon}}
	\ll_{E, \vec{a}}  \frac{q^{-\frac{n}{2}+\left(1-\sigma_0 +\epsilon\right) Xr}}{X^r}\left(2(|a_1|+\cdots+|a_k|)\right)^r.
	\] 
	The above two cases prove the second part of the proposition for even integers.

\noindent
\textbf{Case (3).} Let $r$ be odd. This implies that  $P_1\ldots P_r\neq \square$.
		Similarly to the above case, we obtain
		\[
		S_r
		\ll \frac{q^{-n/2+\left(1-\sigma_0+\epsilon \right) Xr}}{X^r}
		\big(2(|a_1|+\cdots+|a_k|)\big)^r.
		\]
\end{proof}


\begin{prop}\label{moment computation 1 im}
Assume that  $1\leq X\le \frac{\kappa_\cF}{2r}$ for a positive integer $r$. Then for $\vec{t}=(t_1, \ldots, t_k)\in \mathbb R^k$, uniformly for all odd natural numbers $r\leq {\frac{\kappa_\cF}{2}}$ and for every $\epsilon>0$, 
	\begin{align*}
	\frac{1}{| \CMcal{H}(\mathcal{F})|}\sum_{D\in  \CMcal{H}(\mathcal{F})} 
	\Big(\operatorname{Im} \mathcal{P}_X^{\vec{a}, \vec{t}}\big(\sigma_0, \mathcal{F}\big) \Big)^r
	\ll_{\mathcal{F}, \epsilon, \vec{a}} \frac{q^{-n/2+\left(1-\sigma_0+\epsilon\right)Xr}}{X^r} \left(2(|a_1|+\cdots+|a_k|)\right)^r.
	\end{align*}
	If $\vec{t}$ falls into either of the mesoscopic or macroscopic regimes, then uniformly for all even natural numbers $r\ll \log^{\frac13}(\kappa_\cF)$, 
	\begin{align*}
	&\frac{1}{| \CMcal{H}(\mathcal{F})|}\sum_{D\in  \CMcal{H}(\mathcal{F})} 
	\Big(\operatorname{Im} \mathcal{P}_X^{\vec{a}, \vec{t}}\big(\sigma_0, \mathcal{F}\big) \Big)^r
	=\frac{r!}{(\frac{r}2)! 2^{\frac{r}{2}}}
	\Big(\mathcal{V}_{\operatorname{Im}}(\vec{a}, \vec{t}, X)\Big)^{\frac{r}2}\Big(1+O\Big(\frac{r^3}{\log X}\Big)\Big),
	\end{align*}
	where $\mathcal{V}_{\operatorname{Im}}(\vec{a}, \vec{t}, X)$ is defined by \eqref{variance for imaginary}.
\end{prop}


\begin{proof}
	Similarly to the argument in the proof of Proposition \ref{moment computation 1 re}, it suffices to analyze the behaviour of the following sum, compared to the sum in \eqref{prime sum for cosine}. For
	\begin{align*}
	c_{\mathcal{F}}^2(Q_{\frac{r}2}):=
	\lambda^2_{\mathcal{F}}(Q_{\frac{r}2})
	\bigg(\sum_{j=1}^{k}a_j^2\sin^2(t_j\log|Q_{\frac{r}2}|)+ 2\sum_{\substack{j_1, j_2=1\\j_1<j_2}}^k a_{j_1}a_{j_2}\sin(t_{j_1}\log|Q_{\frac{r}2}|)\sin(t_{j_2}\log|Q_{\frac{r}2}|)\bigg), 
	\end{align*}
let
	$
	J(\mathcal{F}):=\sum_{Q_{\frac{r}2}\in\CMcal{P}_{\leq X}}
	\frac{c_{\mathcal{F}}^2(Q_{\frac{r}2})}{(|Q_{\frac{r}2}|+1)|Q_{\frac{r}2}|^{2\sigma_0-1}}.
	$
Note the identity
	\[
	2\sin\big(t_{j_1}\log|Q_{\frac{r}2}|\big)\sin\big(t_{j_2}\log|Q_{\frac{r}2}|\big)
	=\cos\big((t_{j_1}+t_{j_2})\log|Q_{\frac{r}2}|\big)
	-
	\cos\big((t_{j_1}-t_{j_2})\log|Q_{\frac{r}2}|\big).
	\]
Using Lemma \ref{logarithmic cosine sum}, and employing a similar approach as in the proof of Proposition \ref{moment computation 1 re}, we obtain
	\begin{align*}
	J(\mathcal{F})&=\sum_{j=1}^k\frac{a_j^2}{2}\left(\log X-\log\Big(\min\Big\{X, \frac{1}{2|t_j|}\Big\}\Big)\right)\\
	& +\sum_{\substack{j_1, j_2=1\\j_1<j_2}}^k a_{j_1} a_{j_2} \left(\log \Big(\min\Big\{X, \frac{1}{|t_{j_1}+t_{j_2}|}\Big\}\Big)-\log \Big(\min\Big\{X, \frac{1}{|t_{j_1}-t_{j_2}|}\Big\}\Big)\right)+O(1)\\
	&=\mathcal{V}_{\operatorname{Im}}(\vec{a}, \vec{t}, X)+O(1).
	\end{align*}
\end{proof}
For the rest of our discussion, we define another Dirichlet polynomial by  
\begin{align*}
\mathcal{D}_X^{\vec{a}, \vec{t}}\big(s, \mathcal{F}\big)
&:=
a_1 \mathcal{D}_X\big(s+it_1, \mathcal{F}\big)
+\cdots+
a_k	\mathcal{D}_X\big(s+it_k, \mathcal{F}\big) \\
&	=a_1 \sum_{f\in  \CMcal{M}_{\leq X}}
\frac{\Lambda_{\mathcal{F}}(f)\chi_D(f)}{|f|^{s+it_1}}
+\cdots+ a_k \sum_{f\in  \CMcal{M}_{\leq X}}
\frac{\Lambda_{\mathcal{F}}(f)\chi_D(f)}{|f|^{s+it_k}}.
\end{align*}


\begin{prop}\label{moment computation 2}
	Let $1\leq X\le \frac{\kappa_\cF}{2r}$.  Uniformly for all odd natural numbers  $r\ll \tfrac{(\log X)^{1/3}}{(\log \log(\kappa_\cF))^{2/3}}$, 
	\begin{align*}
	&	\frac{1}{| \CMcal{H}(\mathcal{F})|}\sum_{D\in  \CMcal{H}(\mathcal{F})}
	\Big(\operatorname{Re}\mathcal{D}_X^{\vec{a}, \vec{t}}\big(\sigma_0, \mathcal{F}\big)
	- \epsilon_{\mathcal{F}}\mathcal{M}(\vec{a}, \vec{t}, X)\Big)^r
	\ll \frac{r!\, r^{\frac{3}{2}}(\log X)^{\frac{r}{2}}}{(\frac{r}2)!\, 2^{\frac{r}2}}\frac{(\log \log( \kappa_\cF))^{2}}{\sqrt{\log X}}.
	\end{align*}
The same result holds for the imaginary part with zero mean.

	Uniformly for all even natural numbers $r\ll \tfrac{(\log X)^{1/3}}{(\log \log(\kappa_\cF))^{2/3}}$, 
	\begin{align*}
	\frac{1}{| \CMcal{H}(\mathcal{F})|}
&\sum_{D\in  \CMcal{H}(\mathcal{F})}\Big(
	\operatorname{Re} \mathcal{D}_X^{\vec{a}, \vec{t}}\big(\sigma_0, \mathcal{F}\big)- \epsilon_{\mathcal{F}}\mathcal{M}(\vec{a}, \vec{t}, X) \Big)^r
	\\
	&=\frac{r!}{(\frac{r}2)!\, 2^{\frac{r}{2}}} \Big(\mathcal{V}_{\operatorname{Re}}(\vec{a}, \vec{t}, X)\Big)^{\frac{r}{2}}
	\bigg(1+O\Big(\frac{r^3}{\log X}\Big)+O\Big(\frac{r^{\frac{3}{2}}(\log \log(\kappa_\cF))^{2}}{\sqrt{\log X}}\Big)\bigg) ,
	\end{align*}
	If $\vec{t}$ falls into mesoscopic or macrascopic regime then uniformly for all even $r\ll \tfrac{(\log X)^{1/3}}{(\log \log(\kappa_\cF))^{2/3}}$, 
	\begin{align*}
	&\frac{1}{| \CMcal{H}(\mathcal{F})|}\sum_{D\in \CMcal{H}(\mathcal{F})} \Big(\operatorname{Im} \mathcal{D}_X^{\vec{a}, \vec{t}}\big(\sigma_0, \mathcal{F}\big)\Big)^r
	=\frac{r!}{(\frac{r}2)! 2^{\frac{r}2}} \Big(\mathcal{V}_{\operatorname{Im}} (\vec{a}, \vec{t}, X)\Big)^{\frac{r}2}\left(1+O\Big(\frac{r^3}{\log X}\Big)\right).
	\end{align*}
Here $\mathcal{M}(\vec{a}, \vec{t}, X)$, $ \mathcal{V}_{\operatorname{Re}}(\vec{a}, \vec{t}, X)$, $\mathcal{V}_{\operatorname{Im}}(\vec{a}, \vec{t}, X)$, and  $\epsilon_{\mathcal{F}}$ are defined by \eqref{mean for real}, \eqref{variance for real}, \eqref{variance for imaginary} and \eqref{sign of MV} respectively.
\end{prop}


\begin{proof}
	We first expand $\operatorname{Re}\mathcal{D}_X^{\vec{a}, \vec{t}}\big(\sigma_0, \mathcal{F}\big)$ and $\operatorname{Im} \mathcal{D}_X^{\vec{a}, \vec{t}}\big(\sigma_0, \mathcal{F}\big)$ as expressions in terms of $\mathcal{P}_X^{\vec{a}, \vec{t}}\big(\sigma_0, \mathcal{F}\big)$.
	
	For the real part, we have 
	\begin{align*}
	&\operatorname{Re} \mathcal{D}_X^{\vec{a}, \vec{t}}\big(\sigma_0, \mathcal{F}\big)
	= 
	 \, \operatorname{Re} \mathcal{P}_X^{\vec{a}, \vec{t}}\big(\sigma_0, \mathcal{F}\big)
	\\
	&+\sum_{\substack{P\in\CMcal{P}_{\leq \frac{X}{2}}\\ P\nmid D}}\frac{\lambda_{\mathcal{F}}(P^2)\left(a_1 \cos(2 t_1 \log |P|)+\cdots+a_k \cos(2 t_k \log |P|)\right)}{2|P|^{2\sigma_0}}+O_{\mathcal{F}, \vec{a}, k}\left(1\right),
	\end{align*}
where the error term stands for the contribution of $P^k$ with $k\geq 3$.
	
	For the symplectic case, applying Lemmas \ref{upper bound for truncated sum over primes } and \ref{logarithmic cosine sum}, we have 
	\begin{align*}
\operatorname{Re} \mathcal{D}_X^{\vec{a}, \vec{t}}\big(\sigma_0, \chi_D\big)
=\operatorname{Re} \mathcal{P}_X^{\vec{a}, \vec{t}}\big(\sigma_0, \chi_D\big)
+\CMcal{M}(\vec{a}, \vec{t}, X)+O(\log \log n),
	\end{align*}
where the error term $O(\log\log n)$ comes from the sum over $P$ such that $P| D$.
	
	For the orthogonal case, we can use the fact that $\lambda_{\mathcal{F}}(P^2)=\lambda(P^2)-1$ to compute
\[
\sum_{d(P)\leq X}\frac{\lambda(P^2)\cos(2td(P))}{|P|^{2\sigma_0}}=O\left(\log \log X\right),	
\] 
and hence again applying Lemma \ref{logarithmic cosine sum},
\[
\operatorname{Re} \mathcal{D}_X^{\vec{a}, \vec{t}}\big(\sigma_0, E\otimes \chi_D\big)
=\operatorname{Re} \mathcal{P}_X^{\vec{a}, \vec{t}}\big(\sigma_0, E\otimes \chi_D\big)
-\CMcal{M}(\vec{a}, \vec{t}, X)+O(\log \log m).
\]

For the imaginary part, applying Lemma \ref{logarithmic cosine sum} consists of sine function gives
	\[
	\operatorname{Im}  \mathcal{D}_X^{\vec{a}, \vec{t}}\big(\sigma_0, \mathcal{F}\big)
	=\operatorname{Im} \mathcal{P}_X^{\vec{a}, \vec{t}}\big(\sigma_0, \mathcal{F}\big)+O\big(\log\log(\kappa_\cF)\big),
	\]  
Thus, for the imaginary part of the logarithm the mean is zero, whereas $\CMcal{M}(\vec{a}, \vec{t}, X)$ is the mean value of the real part. 
	
The proof now can be completed by using binomial expansion, Propositions \ref{moment computation 1 re} and \ref{moment computation 1 im}, and following the proof of Lemma 5.3 of \cite{DL}. 
\end{proof}


\section{Proofs of Theorems \ref{Unconditional CLT for real near to half line:re}--\ref{main theorem conditional microscopic}, \ref{main theorem conditional mesoscopic and macroscopic}, and \ref{Uncon mes or micro}}\label{proof of main theorems}

We begin our discussion by studying the correlation between the real (imaginary) parts of Dirichlet polynomials at different shifts.


\subsection{Correlations of the Real Parts of Dirichlet Polynomials}\label{correlation of reals}
We would like to analyze the correlation between values of the real parts of the following Dirichlet polynomial at $\sigma_0+it_{j_1}$ and $\sigma_0+it_{j_2}$. Depending on the underlying family $\mathcal{F}$, we define
\[
\mathcal{R}_{X, t_j, D}
=
\operatorname{Re}
\sum_{f\in \CMcal{M}_{\leq X}}\frac{ \Lambda_{\mathcal{F}}(f)\chi_D(f)}{d(f)|f|^{\sigma_0+it_j}}  \quad \text{for} \quad j=j_1, j_2.
\]
We define
\begin{equation*} \label{eq:cov RXtjD}
\begin{split}
&\operatorname{Cov}(\mathcal{R}_{X, t_{j_1}, D}, \mathcal{R}_{X, t_{j_2}, D})
\\
&= \frac1{|\CMcal{H}(\CMcal{F})|}\sum_{D\in  \CMcal{H}(\CMcal{F})}
\big(  \mathcal{R}_{X, t_{j_1}, D} \, \mathcal{R}_{X, t_{j_2}, D}
\big) 
- \bigg(\frac1{|\CMcal{H}(\CMcal{F})|}\sum_{D\in  \CMcal{H}(\CMcal{F})} \mathcal{R}_{X, t_{j_1}, D} \bigg)
\bigg(\frac1{|\CMcal{H}(\CMcal{F})|}\sum_{D\in  \CMcal{H}(\CMcal{F})} \mathcal{R}_{X, t_{j_2}, D} \bigg).
\end{split}
\end{equation*}
Here for each $j=j_1, j_2$, following the computations carried in Section \ref{moments of lin comb}, we find that 
\[
\frac1{|\CMcal{H}(\CMcal{F})|}    \sum_{D\in \CMcal{H}(\CMcal{F})} \mathcal{R}_{X, t_j, D}
=\frac{1}{2} \log\Big(\min\Big\{ X, \frac{1}{|t_j|}\Big\}\Big)+O\left(q^{-n/2+(1/2+\epsilon)X}\right).
\] 
It thus remains to consider  
\[
\frac1{|\CMcal{H}(\CMcal{F})|}  \sum_{D\in  \CMcal{H}(\CMcal{F})} 
\operatorname{Re}
\sum_{f\in \CMcal{M}_{\leq X}}\frac{ \Lambda_{\mathcal{F}}(f)\chi_D(f)}{d(f)|f|^{\sigma_0+it_{j_1}}}
\operatorname{Re} \sum_{f\in \CMcal{M}_{\leq X}}\frac{ \Lambda_{\mathcal{F}}(f)\chi_D(f)}{d(f)|f|^{\sigma_0+it_{j_2}}}.
\]
Note that
\begin{align*}
& \frac{1}{|\CMcal{H}(\CMcal{F})|}	\sum_{D\in  \CMcal{H}(\CMcal{F})} \sum_{P_1\in\CMcal{P}_{\leq X}}
\sum_{P_2\in\CMcal{P}_{\leq X}}
\frac{\lambda_{\cF}(P_1)\lambda_{\cF}(P_2)
	\chi_D(P_1P_2)\cos(t_{j_1}\log|P_1|)\cos(t_{j_2}\log|P_2|)}
{|P_1P_2|^{\sigma_0}} \\
=& 	\frac{1}{| \CMcal{H}(\CMcal{F})|}
\sum_{P_1\in\CMcal{P}_{\leq X}} 
\sum_{P_2\in\CMcal{P}_{\leq X}}
\lambda_{\cF}(P_1)\lambda_{\cF}(P_2)\frac{\cos(t_{j_1}\log|P_1|)\cos(t_{j_2}\log|P_2|)}
{|P_1P_2|^{\sigma_0}}
\sum_{D\in  \CMcal{H}(\CMcal{F})}\chi_D(P_1 P_2).	
\end{align*}
If $P_1 P_2=\square$, then $P_1= P_2= Q$ for a prime $Q$ and
\begin{align*}
& \sum_{Q\in\CMcal{P}_{\leq X}} 
\frac{\lambda_{\cF}(Q)^2\cos(t_{j_1}\log|Q|)\cos(t_{j_2}\log|Q|)}{(|Q|+1)|Q|^{2\sigma_0-1}}	 \\
=&\frac12\sum_{Q\in\CMcal{P}_{\leq X}}
\frac{\lambda_{\cF}(Q)^2\cos((t_{j_1}+t_{j_2})\log|Q|)}{(|Q|+1)|Q|^{2\sigma_0-1}}
+
\frac12\sum_{Q\in\CMcal{P}_{\leq X}}
\frac{\lambda_{\cF}(Q)^2\cos((t_{j_1}-t_{j_2})\log|Q|)}{(|Q|+1)|Q|^{2\sigma_0-1}}.
\end{align*}    
From the contribution of the primes and prime squares, in a manner similar to our argument in Section \ref{moments of lin comb}, up to a negligibly small error term, we can deduce that the above is equal to
\begin{align*}
&\frac{1}{2} \log\bigg(\min\Big\{ X, \frac{1}{|t_{j_1}-t_{j_2}|}\Big\}\bigg)
+\frac{1}{2} \log\bigg(\min\Big\{ X, \frac{1}{|t_{j_1}+t_{j_2}|}\Big\}\bigg)\\
&+\frac{1}{4} \log\bigg(\min\Big\{ X, \frac{1}{2|t_{j_1}|}\Big\}\bigg) \log\bigg(\min\Big\{ X, \frac{1}{2|t_{j_2}|}\Big\}\bigg).
\end{align*}
By ignoring the terms $P_1 P_2\neq\square$, at the cost of a negligibly small error term, we finally obtain
\begin{equation}\label{computed-cov RXtjD}
\begin{split}
\operatorname{Cov}\big(\mathcal{R}_{X, t_{j_1}, D}, \mathcal{R}_{X, t_{j_2}, D}\big)=\frac{1}{2}\log\bigg(\min\Big\{ X, \frac{1}{|t_{j_1}-t_{j_2}|}\Big\}\bigg)
+\frac12 \log\bigg(\min\Big\{ X, \frac{1}{|t_{j_1}+t_{j_2}|}\Big\}\bigg). 
\end{split}
\end{equation}


\begin{prop}\label{Gaussian for real part of DP}
	Let $X\geq 1$ be sufficiently small such that $\frac{X}{n}\to 0$ as $n\to \infty$, but $\log X=\log n+o(\sqrt{\log n})$.
	Suppose that $0\leq \delta_1, \ldots, \delta_k\leq 1$ such that $|t_1|\asymp \frac{1}{X^{\delta_1}}, \ldots, |t_k|\asymp \frac{1}{X^{\delta_k}}$. Consider 
	\[
	\CMcal{M}_{X, t_j}
	=\frac{1}{2} \log\Big(\min\big\{X, \tfrac{1}{2|t_j|}\big\}\Big), 
	\quad \mathcal{V}_{X, t_j}=\frac{1}{2}\log X+\frac{1}{2} \log\Big(\min\big\{ X, \tfrac{1}{|t_j|}\big\}\Big),
	\]
	and 
	\[
	\widehat{\mathcal{R}}_{X, t_j, D}
	=\frac{\mathcal{R}_{X, t_j, D}- \epsilon_{\mathcal{F}} \CMcal{M}_{X,  t_j}}{\sqrt{\mathcal{V}_{X, t_j}}} \quad \text{for} \quad j=1, \ldots, k.
	\]
As $D$ varies in $ \CMcal{H}(\mathcal{F})$, the vector 
$\big(\widehat{\mathcal{R}}_{X, t_1, D}, \ldots, \widehat{\mathcal{R}}_{X, t_k, D}\big)$ converges in law to a Gaussian vector $(Z_1, \ldots, Z_k)$ with mean $0$ and covariance function  
	\[
	\operatorname{Cov}(Z_{j_1}, Z_{j_2})=
	\begin{cases}
	1 & \text{ if }\,\, t_{j_1}=\pm t_{j_2}, \\ 
	\frac{2\left(\delta_{j_1}\wedge \delta_{j_2}\right)}{\sqrt{1+\delta_{j_1}} \sqrt{1+\delta_{j_2}}} & \text{ otherwise}.
	\end{cases}
	\]
\end{prop}


\begin{proof}
From \eqref{computed-cov RXtjD}, up to a negligible small error term, we have
	\begin{align*} 
&\operatorname{Cov}\big(\widehat{\mathcal{R}}_{X, t_{j_1}, D}, \widehat{\mathcal{R}}_{X, t_{j_2}, D}\big)
	\\
	&=\mathbb{E}\big(\widehat{\mathcal{R}}_{X, t_{j_1}, D} \widehat{\mathcal{R}}_{X, t_{j_2}, D}\big)
	-\mathbb{E}\big(\widehat{\mathcal{R}}_{X,  t_{j_1}, D}\big)
	\mathbb{E}\big(\widehat{\mathcal{R}}_{X, t_{j_2}, D}\big) 
	= \mathbb{E}\big(\widehat{\mathcal{R}}_{X, t_{j_1}, D} \widehat{\mathcal{R}}_{X, t_{j_2}, D}\big)
	\\
	&= \frac{1}{\sqrt{\mathcal{V}_{X, t_{j_1}} \mathcal{V}_{X, t_{j_2}}}}
	\bigg(\frac{1}{2} \log\Big(\min\Big\{ X, \frac{1}{|t_{j_1}-t_{j_2}|}\Big\}\Big)
	+\frac{1}{2} \log\Big(\min\Big\{ X, \frac{1}{|t_{j_1}+t_{j_2}|}\Big\}\Big)\bigg),
	\end{align*} 
since $\mathbb{E}\big(\widehat{\mathcal{R}}_{X,a_j, t_j, D}\big)=\mathbb{E}\big(\widehat{\mathcal{R}}_{X, t_j, D}\big)=0$. Therefore, by the hypotheses on $t_{j_1}$ and $t_{j_2}$, 
	\[
	\operatorname{Cov}\big(\widehat{\mathcal{R}}_{X, t_{j_1}, D}, \widehat{\mathcal{R}}_{X, t_{j_2}, D}\big)
	\longrightarrow \operatorname{Cov}\big(Z_{j_1}, Z_{j_2}\big) \quad \text{ as } n\to \infty.
	\]
Let $\vec{a}=(a_1, \ldots, a_k)\in \mathbb{R}^k$. From Proposition \ref{moment computation 2} together with the hypothesis that $\frac{X}{n} \to 0$ but $\log X=\log n+ o(\sqrt{\log n})$, we conclude that the moments of $a_1\widehat{\mathcal{R}}_{X, t_1, D}+\cdots+a_k\widehat{\mathcal{R}}_{X, t_k, D}$ asymptotically match those of a Gaussian random variable with mean $0$ and variance $1$. Since the Gaussian distribution is determined by its moments, our theorem follows.

In fact, we used the known result from probability that, if $\{Z_j\}_{j=1}^k$ is a sequence of Gaussian random variables, then $\{Z_j\}_{j=1}^k$ is a $k$-variate Gaussian random variable if and only if any linear combination of $Z_j$'s has Gaussian distribution. In the case that $(Z_{j_1}, Z_{j_2})$ has a bivariate Gaussian distribution, $Z_{j_1}$ and $Z_{j_2}$ are independent if and only if they are uncorrelated.
\end{proof}


\begin{proof}[Proof of Theorem \ref{Unconditional CLT for real near to half line:re}] 
	Take $\sigma_0=\frac12+\frac{c}{X}$. We begin by using the Proposition \ref{logL-as-PX} to write $\log |\mathfrak{L}_{\vec{a}, \vec{t}}\left(s, \chi_D\right)|$ as a linear combination of  $\mathcal{R}_{X, t_j, D}$'s and then apply Proposition \ref{Gaussian for real part of DP}. Ve view $\log |\mathfrak{L}_{\vec{a}, \vec{t}}\left(s, \chi_D\right)|$ as sum of two random variables, say $Y_D+Z_D$, where $Y_D=a_1\mathcal{R}_{X, t_1, D}+\cdots+a_k\mathcal{R}_{X, t_k, D}$ and $Z_D$ is an error term. The hypothesis $g\left(\sigma_0-1/2\right)=o(\sqrt{\log n})$ imply that $\log X=\log n+O(\log_2 n)$, which satisfies the hypothesis of Proposition \ref{Gaussian for real part of DP}. Thus, by Proposition \ref{Gaussian for real part of DP} the proof of the theorem is complete once we are able to show that (see \cite[Lemma 2.9]{HoughSymplectic}) 
	\[
	\frac{1}{|\mathcal{H}_n|}\sum_{D\in \mathcal{H}_n}|Z_D|^2=o\left(\log n\right).
	\]
Indeed, thanks to Lemma \ref{moments of lambda polyl} with $k=2$ and Lemma \ref{moments of lambda polyl2}, the above claim holds in view of the hypothesis that $\frac{g}{X}= o(\sqrt{\log n})$.

An analogous argument applies to the orthogonal family, although we omit the details in order to streamline the presentation.

 \end{proof}

\begin{proof}[Proof of Theorem \ref{main theorem conditional microscopic}]
	This directly follows from combining Theorem \ref{Unconditional CLT for real near to half line:re}, Corollary \ref{difference between logarithm micro}, and Lemma 2.9 of \cite{HoughSymplectic}.
\end{proof}


\subsection{Correlations of the Imaginary Parts of Dirichlet Polynomials}\label{subsection 6.2}

We now want to analyze the imaginary parts similarly, so we define the following for each $j=1,\ldots, k$.
\[
\mathcal{I}_{X, t_j, D}
:=\, \operatorname{Im} \sum_{f\in \CMcal{M}_{\leq X}}\frac{ \Lambda_{\mathcal{F}}(f)\chi_D(f)}{d(f)|f|^{\sigma_0+it_j}}. 
\]
Following the same argument as in Section \ref{correlation of reals},  we are led to study a sum of the following form.
\begin{align*}
 \sum_{Q\in\CMcal{P}_{\leq X}} 
\frac{\sin(t_{j_1}\log|Q|)\sin(t_{j_2}\log|Q|)}{(|Q|+1)|Q|^{2\sigma_0-1}}=\frac12\bigg(\sum_{Q\in\CMcal{P}_{\leq X}}
\frac{\cos((t_{j_1}+t_{j_2})\log|Q|)}{(|Q|+1)|Q|^{2\sigma_0-1}}
-
\sum_{Q\in\CMcal{P}_{\leq X}}
\frac{\cos((t_{j_1}-t_{j_2})\log|Q|)}{(|Q|+1)|Q|^{2\sigma_0-1}}\bigg).
\end{align*}
Thus, up to a negligibly small error term, we have
\begin{align}\label{covariance for imaginary}
\operatorname{Cov}(\mathcal{I}_{X, t_{j_1}, D}, \mathcal{I}_{X, t_{j_2}, D})
=\frac{1}{2} \log\bigg(\min\Big\{ X, \frac{1}{|t_{j_1}-t_{j_2}|}\Big\}\bigg)-\frac{1}{2} \log\bigg(\min\Big\{ X, \frac{1}{|t_{j_1}+t_{j_2}|}\Big\}\bigg),
\end{align}
This indicates that the quantities $\operatorname{Im} \log L\!\left(\tfrac{1}{2}+it_j, \mathcal{F}\right)$ are always approximately independent.


\begin{prop}\label{Gaussian for imaginary part of DP}
	Assume the hypotheses of Proposition \ref{Gaussian for real part of DP} on $X, t_j \geq 0$ and $0\leq \delta_j<1$. We further introduce the following. 
	\[
	\widehat{\mathcal{I}}_{X, t_j, D}
	:=\frac{\mathcal{I}_{X, t_j, D}}{\sqrt{\mathcal{V}_{X, t_j}}}
	\quad \text{ and } \quad 
	\widehat{\mathcal{V}}_{X, t_j}
	:=\frac{1}{2}\log X-\frac{1}{2} \log\bigg(\min\Big\{ X, \frac{1}{|t_j|}\Big\}\bigg), 
	\quad j=1,\ldots, k.
	\]
	As $D$ varies in $ \CMcal{H}(\mathcal{F})$, the vector 
	$\big(\widehat{\mathcal{I}}_{X, t_1, D}, \ldots, \widehat{\mathcal{I}}_{X, t_k, D}\big)$ converges in law to a Gaussian vector $(Z_1, \ldots, Z_k)$ with mean $0$ and covariance 
	\[
	\textrm{Cov}(Z_{j_1}, Z_{j_2})=
	\begin{cases}
	1 & \text{ if } \,\, t_{j_1}= t_{j_2}, \\ 
	0 & \text{ otherwise}.
	\end{cases}
	\]
	\end{prop}
	
	\begin{proof}
Up to a negligible error term, we have 
		\begin{align*} 
		&\operatorname{Cov}\big(\widehat{\mathcal{I}}_{X, t_{j_1}, D}, \widehat{\mathcal{I}}_{X, t_{j_2}, D}\big)
		=\mathbb{E}\big(\widehat{\mathcal{I}}_{X, t_{j_1}, D} \widehat{\mathcal{I}}_{X, t_{j_2}, D}\big)
		-\mathbb{E}\big(\widehat{\mathcal{I}}_{X, t_{j_1}, D}\big)
		\mathbb{E}\big(\widehat{\mathcal{I}}_{X, t_{j_2}, D}\big) 
		= \mathbb{E}\big(\widehat{\mathcal{I}}_{X, t_{j_1}, D} \widehat{\mathcal{I}}_{X,  t_{j_2}, D}\big)
		\\
		&= \frac{1}{\sqrt{\mathcal{V}_{X, t_{j_1}} \mathcal{V}_{X, t_{j_2}}}}
		\bigg(\frac{1}{2} \log\Big(\min\Big\{ X, \tfrac{1}{|t_{j_1}-t_{j_2}|}\Big\}\Big)
		-\frac{1}{2} \log\Big(\min\Big\{ X, \tfrac{1}{|t_{j_1}+t_{j_2}|}\Big\}\Big)\bigg),
		\end{align*}
since $\mathbb{E}\big(\widehat{\mathcal{I}}_{X, t_{j_1}, D}\big)
		=\mathbb{E}\big(\widehat{\mathcal{I}}_{X, t_{j_2}, D}\big)=0$. 
Therefore,  
		\[
		\operatorname{Cov}\big(\widehat{\mathcal{I}}_{X, t_{j_1}, D}, \widehat{\mathcal{I}}_{X, t_{j_2}, D}\big)
		\longrightarrow \operatorname{Cov}(Z_{j_1}, Z_{j_2}) \quad \text{ as } n\to \infty
		\]
on the basis of our assumption on $t_{j_1}$ and $t_{j_2}$. 
		From Proposition \ref{moment computation 1 re} together with the hypothesis that $\frac{X}{n} \to 0$ but $\log X=\log n+ o(\sqrt{\log n})$ and the ranges of $t_j$'s, we conclude that the moments of $a_1\widehat{\mathcal{I}}_{X, t_1, D}+\cdots+a_k\widehat{\mathcal{I}}_{X, t_k, D}$ asymptotically match the moments of a Gaussian random variable. Since the Gaussian distribution is determined by its moments, our theorem follows.
	\end{proof}


\begin{proof}[Proof of Theorem \ref{Unconditional CLT for real near to half line:im}] 
The proof follows exactly the same way as the proof of Theorem \ref{Unconditional CLT for real near to half line:re} by using Propositions \ref{logL-as-PX} and \ref{Gaussian for imaginary part of DP}.
\end{proof}


\begin{proof}[Proof of Theorems \ref{main theorem conditional mesoscopic and macroscopic} and \ref{Uncon mes or micro}]

	The first statement in the theorems follow from Proposition \ref{logL-as-PX}, Corollary \ref{difference between logarithm micro} (Proposition \ref{close to 1/2-line} for orthogonal family), and Proposition \ref{Gaussian for real part of DP}. The second parts follow from combining Proposition \ref{logL-as-PX}, Proposition \ref{difference between shift of logarithm2}, and Proposition \ref{Gaussian for imaginary part of DP}.

\end{proof}


\section{Proofs of Unconditional Results on the Critical Circle}


\subsection{Lower Bound in Theorem \ref{mild conditioning lower bound}}

Since we will consider microscopic shifts, to obtain lower bounds for the distribution function, we need to understand nonvanishing of $L$-functions at these heights via one-level density estimates. Recently, Carneiro, Chirre and Milinovich \cite{CCM} found a general recipe for the proportion of nonvanishing of $L$-functions at low lying heights over different families based on the theory of reproducing kernel Hilbert spaces of entire functions. 

For the symplectic family, Rudnick~\cite{Rud} obtained an asymptotic formula for the one-level density 
\begin{align}\label{one level density}
\lim_{g\to \infty}\sum_{D\in  \CMcal{H}_{2g+1}}\sum_{j=1}^{2g}\phi(2g \theta_{j, D})=\int_{\mathbb{R}}\phi(x)W_{Sp}(x)dx,
\end{align}
where $\hat{\phi}$ is supported in $(-2, 2)$. Later, lower order terms for $\hat{\phi}$ with the same support has been established by Bui and Florea in \cite{BF}. One can obtain a similar result for $\mathcal{H}_{2g+2}$. 

As an application of \cite[Theorem 2]{CCM} and \eqref{one level density}, we have the following corollary.


\begin{cor}\label{no-vanishing result}
	Let $\mathfrak{H}_{Sp, 2\pi}$ denote the normed vector space of entire functions $h$ of exponential type at most $2\pi$ with finite norm as in
	\[
	||h||_{ \CMcal{H}_{Sp, 2\pi}}
	:=||h||_{L^2(\mathbb{R}, W_{Sp})}=\left(\int_{\mathbb{R}}|h(x)|^2 W_{Sp}(x)dx\right)^{1/2}<\infty.
	\]
Let $K= K_{Sp, 2\pi}$ be the reproducing kernel of the Hilbert space associated to $\mathfrak{H}_{Sp, 2\pi}$ given by \cite[Theorem 6, case (ii)]{CCM}. Then 
	\[
	\liminf_{g\to \infty}\frac{1}{| \CMcal{H}_{n}|}\sum_{\substack{D\in  \CMcal{H}_{n}
			\\ 
			L\left(\frac12 +i\frac{2\pi}{\log q}\frac{\alpha}{n}\right)\neq 0}}1
	\geq 1-\frac{1}{K(\alpha, \alpha)+|K(\alpha, -\alpha)|}.
	\]
\end{cor}

Note that, the analytic conductor of $L(s, \chi_D)$ is $c_{\chi_D}:=q^{n}$, and so $\log(c_{\chi_D})=n\log q$. Define
\begin{equation}\label{values of nonvanishing}
	r(\alpha):=1-\frac{1}{K(\alpha, \alpha)+|K(\alpha, -\alpha)|} \quad \text{for } \, \alpha>0.
\end{equation}	 
Its maximum value is its limit as $\alpha\to 0^+$, and $r(0)=0.9427\ldots$, as in \cite{BF}. The limit of $r(\alpha)$ is $\frac12$ as $\alpha\to \infty$ $($see \cite[Figure 5]{CCM} for more details$)$.


\begin{proof}[Proof of Theorem \ref{mild conditioning lower bound}]
Corollary \ref{no-vanishing result} implies that at least $r(\alpha)$ proportion of $D\in  \CMcal{H}_{n}$ for which 
	$
	L\big(\tfrac12+it, \chi_D\big)\neq 0.
	$
For the microscopic regime, note that taking $a_1=1, a_2=0$,
	\[
	\mathcal{M}(\vec{a}, \vec{t}, g)=\frac{\log g}{2}
	\quad \text{ and }\quad \mathcal{V}_{\operatorname{Re}}(\vec{a}, \vec{t}, g)=\log g.
	\]
We then consider the following events. For a real number $b$, let $\mathcal{A}_b$ be the set of $D\in  \CMcal{H}_{n}$ such that 
	\[
	\frac{1}{\sqrt{\log g}}\Big(\log \big|L\big(\tfrac{1}{2}+it, \chi_D\big)\big|-\frac12 \log g\Big)>b.
	\]
We also define $\mathcal{B}$ to be the set of $D\in  \CMcal{H}_n$ such that $L\big(\tfrac12+it, \chi_D\big)\neq 0$. Let $\mathcal{B}^C$ denote its complement in $ \CMcal{H}_n$. Observe that since $\mathbb{P}(\mathcal{A}_b\cap \mathcal{B}^C)=0$,
	\[
	\mathbb{P}(\mathcal{A}_b)
	=\mathbb{P}(\mathcal{A}_b \cap \mathcal{B})+\mathbb{P}(\mathcal{A}_b \cap \mathcal{B}^C)
	=\mathbb{P}(\mathcal{A}_b\cap \mathcal{B}) 
	=\mathbb{P}(\mathcal{A}_b | \mathcal{B}) \mathbb{P}(\mathcal{B}).
	\]
	Assume that the event $\mathcal{A}_b$ happens. We first use Proposition \ref{logL-as-PX} to write $\log L\big(\tfrac12+it, \chi_D\big)$ in terms of $\widetilde{P}_X(\tfrac12+it,\chi_D)$, and then apply Proposition \ref{moment computation 1 re}. Finally, we use Lemma  \ref{moments of lambda polyl} to show that the error term coming from Proposition \ref{logL-as-PX} is negligibly small. This argument gives 
	\[
	\mathbb{P}(\mathcal{A}_b \setminus \mathcal{B}) 
	=\frac{1}{\sqrt{2\pi}} \int_b^\infty e^{-\frac{u^2}{2}}du.
	\]
	Note that, the above discussion tells us $\mathbb{P}(\mathcal{B})\geq r(\alpha)$. Hence, we conclude that 
	\[
	\mathbb{P}(\mathcal{A}_b)\geq  \frac{r(\alpha)}{\sqrt{2 \pi }} \int_b^\infty e^{-\frac{u^2}{2}}du.
	\]
\end{proof}


\subsection{Lower Bound in Theorem \ref{unconditional result for orthogonal family}}

From \cite[Corollary 10.1]{CL}, one can obtain an asymptotic formula for the one-level density as 
\begin{align}\label{one level density o}
\lim_{m\to \infty}\sum_{D\in  \CMcal{H}_{n}^{\Delta, +}}\sum_{j=1}^{m}\phi(m \theta_{j}(E\otimes \chi_D))=\int_{\mathbb{R}}\phi(x)W_{\text{SO(even)}}(x)dx,
\end{align}
when $\hat{\phi}$ is supported in $(-1, 1)$. 

Observe that if $\deg(N_E)$ is even, and hence $m$ is even, then the twists from $\mathcal{H}_{n}^{\Delta, +}$ have $\text{SO}(\text{even})$ symmetry, and the twists from $\mathcal{H}_{n}^{\Delta, -}$ have $\text{SO}(\text{odd})$ symmetry. When $\deg(N_E)$ is odd, and hence $m$ is odd, the twists from $\mathcal{H}_{n}^{\Delta,-}$ have $\text{SO}(\text{odd})$ symmetry, and the twists from $\mathcal{H}_{n}^{\Delta,+}$ have $\text{SO}(\text{even})$ symmetry.
Therefore, the twists are taken from $\mathcal{H}_{n}^{\Delta, +}$ always have $\text{SO}(\text{even})$ symmetry.

As an application of \cite[Theorem 2]{CCM} and \eqref{one level density o}, we have the following corollary.


\begin{cor}\label{no-vanishing result o}
	Let $\mathfrak{H}_{\text{SO(even)}, 2\pi}$ be the normed vector space of entire functions $h$ of exponential type at most $2\pi$ with finite norm. That is, 
	\[
	||h||_{ \CMcal{H}_{\text{SO(even)}, 2\pi}}
	:=||h||_{L^2(\mathbb{R}, W_{\text{SO(even)}})}=\left(\int_{\mathbb{R}}|h(x)|^2 W_{\text{SO(even)}}(x)dx\right)^{1/2}<\infty.
	\]
Let $k = K_{\text{SO(even)}, 2\pi}$ be the reproducing kernel of the Hilbert space associated to $\mathfrak{H}_{\text{SO(even)}, 2\pi}$, determined by Case $(ii)$ of \cite[Theorem 5]{CCM}. Then 
	\[
	\liminf_{n\to \infty}\frac{1}{| \CMcal{H}_{n}^{\Delta, +}|}\sum_{\substack{D\in  \CMcal{H}_{n}^{\Delta, +}\\ L\left(\frac12 +\frac{2\pi i}{\log q}\frac{\alpha}{n}, E\otimes \chi_D\right)\neq 0}}1\geq 1-\frac{1}{k(\alpha, \alpha)+|k(\alpha, -\alpha)|}.
	\] 
\end{cor}
For positive real $\alpha$, we set 
\begin{align}\label{r_E}	
r_E(\alpha) :=1-\frac{1}{k(\alpha, \alpha)+|k(\alpha, -\alpha)|}.
\end{align}
Then the maximum value of $r_E(\alpha)$ is $1/4$ which is attained at $\alpha=0$, and its limit as $\alpha\to \infty$ is $0$. We refer the reader to~\cite[Figure 3]{CCM} for more details.
\begin{proof}[Proof of Theorem \ref{unconditional result for orthogonal family}]
Corollary \ref{no-vanishing result o} provides an $r_E(\alpha)$ proportion of the polynomials $D\in  \CMcal{H}_{n}^{\Delta, +}$ for which $L(1/2, E\otimes \chi_D)\neq 0$. Using this and Proposition \ref{logL-as-PX}, the proof follows similarly to the proof of Theorem \ref{mild conditioning lower bound}.
\end{proof}


\subsection{Upper Bounds in Theorems \ref{mild conditioning lower bound} and \ref{unconditional result for orthogonal family}}
Following the computation in \eqref{eq:at sigma0 at onehalf}, we obtain 
\begin{align*}
&\log \big|L\big(\tfrac{1}{2}+it, \mathcal{F}\big)\big|
- \log \big|L\big(\sigma_0+it, \mathcal{F}\big)\big|\\
&=\frac{\kappa_\cF (\sigma_0-\frac12) \log q}{2} -\frac{1}{2}\sum_{j=1}^{\kappa_\cF}
\log\bigg(\frac{q^{\frac12-\sigma_0}+q^{\sigma_0-\frac12} -2\cos(2\pi\theta_{j}+t\log q)}{2-2 \cos(2\pi\theta_{j}+t\log q)}\bigg)+O_{\eta_\cF}(1),
\end{align*}
where $\theta_j:=\theta_{j, D}(\mathcal{F})$'s are as defined in \eqref{eigen phases}.
Note that the summands in the second term are all nonnegative which leads to the negativity of the second term of the above expression. Therefore, 
\[
\log \big|L\big(\tfrac{1}{2}+it, \mathcal{F}\big)\big|
\leq \log \big|L\big(\sigma_0+it, \mathcal{F}\big)\big|+ \frac{\kappa_\cF (\sigma_0-\frac12) \log q}{2}+O_{\eta_\cF}(1).
\]
From the hypothesis that $\kappa_\cF \big(\sigma_0-\tfrac12\big)=o\big(\sqrt{\log(\kappa_\cF)}\big)$, we obtain
\[
\frac{\log \big|L\big(\tfrac{1}{2}+it, \mathcal{F}\big)\big|}{\sqrt{\log(\kappa_\cF)}}
\leq \frac{\log \big|L\big(\sigma_0+it, \mathcal{F}\big)\big|}{\sqrt{\log(\kappa_\cF)}}+ o(1).
\]
Hence, the proof of Theorem~\ref{mild conditioning lower bound} is completed using Theorem~\ref{Unconditional CLT for real near to half line:re}. 
The proof of Theorem~\ref{unconditional result for orthogonal family} proceeds in the same way, relying on the arguments in Section~\ref{proof of main theorems}, after establishing an analogue of Theorem~\ref{mild conditioning lower bound} for the orthogonal family.


\subsection{Proof of Corollary \ref{mesoscopic fluctuation}}
For a fixed $\delta\in [0, 1)$, let  
\[
\theta_{j}=\frac{\alpha_j}{\kappa^\delta}, \quad \vec{u}=(\theta_1, \theta_2), \quad \vec{v}=(\theta_3, \theta_4), \quad \vec{a}=(1, -1) \quad \text{for } \, j\in \{1, 2, 3, 4\}.
\]
With these choices, we have for the symplectic family
\[
\mathcal{V}_{\operatorname{Im}} (\vec{a}, \vec{u}, n)
=\mathcal{V}_{\operatorname{Re}}(\vec{a}, \vec{v}, n)=\sqrt{(1-\delta)\log n},
\]
and for the orthogonal family
\[
\mathcal{V}_{\operatorname{Im}}(\vec{a}, \vec{u}, m)
=\mathcal{V}_{\operatorname{Re}}(\vec{a}, \vec{v}, m)=\sqrt{(1-\delta)\log m}.
\]
By applying Theorem \ref{main theorem conditional mesoscopic and macroscopic} or Theorem \ref{Uncon mes or micro} according to the family, we conclude that the stochastic process $\mathcal{W}$ converges to a Gaussian process as $\kappa_\cF \to \infty$. For the covariance computation, we follow the approach in Section~\ref{subsection 6.2}. We find that, up to a negligible error term,
\begin{align*}
\operatorname{Cov}(\Delta(\theta_1, \theta_2, \mathcal{F}), \Delta(\theta_3, \theta_4, \mathcal{F}))
&=\mathbb{E}\left(\Delta(\theta_1, \theta_2, \mathcal{F}) \Delta(\theta_3, \theta_4, \mathcal{F})\right)-\mathbb{E}\left(\Delta(\theta_1, \theta_2, \mathcal{F})\right)\mathbb{E}\left(\Delta(\theta_3, \theta_4, \mathcal{F})\right)\\
&= \frac{1}{2}\bigg(\log \min\Big\{\kappa_\cF, \tfrac{1}{|\theta_1 -\theta_3|}\Big\}-\log \min\Big\{\kappa_\cF, \tfrac{1}{|\theta_1 +\theta_3|}\Big\}
\\
&\quad \quad \quad+\log \min\Big\{\kappa_\cF, \tfrac{1}{|\theta_2 -\theta_4|}\Big\}
 -\log \min\Big\{\kappa_\cF, \tfrac{1}{|\theta_2 +\theta_4|}\Big\}
 \\
&\quad \quad   \quad -\log \min\Big\{\kappa_\cF, \tfrac{1}{|\theta_1 -\theta_4|}\Big\}+\log \min\Big\{\kappa, \tfrac{1}{|\theta_1 +\theta_4|}\Big\}
\\
&\quad \quad \quad-\log \min\Big\{\kappa_\cF, \tfrac{1}{|\theta_2 -\theta_3|}\Big\}+\log \min\Big\{\kappa_\cF, \tfrac{1}{|\theta_2+\theta_3|}\Big\}\bigg),
\end{align*}
since $\mathbb{E}\left(\Delta(\theta_1, \theta_2, \mathcal{F})\right)
=\mathbb{E}\left(\Delta(\theta_3, \theta_3, \mathcal{F})\right)=0$.

With the choices of $\theta_j$'s as specified in the hypothesis, we complete the proof.


\section*{Acknowledgments}
 
The first author was partially supported by a postdoctoral fellowship of the Pacific Institute for the Mathematical Sciences. The second author acknowledges financial support from the Australian Research Council Grant DP230100534 and the Max Planck Institute for Mathematics, Bonn. The authors all thank Jon P. Keating for his quick comments on this paper and also Paul Bourgade for suggesting a study of the FHK conjecture as in Section \ref{subsubsec:future}.


\bibliographystyle{amsplain}

\end{document}